\documentclass[11 pt, leqno]{amsart}
\usepackage{amssymb}
\usepackage{amsmath, amscd}	
\usepackage{ bm}
\usepackage[small,nohug,heads=littlevee]{diagrams} 
\diagramstyle[labelstyle=\scriptstyle] 

\usepackage[utopia]{mathdesign}

 \calclayout

\def\ignore#1{\relax}

\usepackage{graphicx}
\usepackage{epstopdf}
\newcommand\inlinegraphic[2][{scale=1.0}]{\begin{array}{c} \includegraphics[#1]{./EPS/#2}\end{array}}

\usepackage[plainpages=false, pdfpagelabels,  colorlinks=true, pdfstartview=FitV, linkcolor=blue, citecolor=blue, urlcolor=blue]{hyperref}

\numberwithin{equation}{section}
\numberwithin{figure}{section}

\def\la{\lambda}
\def\La{\Lambda}

\def\Ga{\Gamma}

\def\vareps{\varepsilon}

\def\End{{\rm End}}

\def\B{\mathfrak{B}}
\def\S{\mathfrak S}

\def\Ind{{\rm Ind}}
\def\Res{{\rm Res}}
\def\inv{^{-1}}
\def\rank{{\rm rank}}
\def\dim{{\rm dim}}
\def\cl{{\rm cl}}

\def\uA{\breve A^\la}

\def\1{\bm 1}

\def\C{\mathcal{C}}
\def\a{\alpha}

\def\spn{{\rm span}}

\def\br{B}
\def\mathbold{\bm}
\def\p #1{ \bm {#1}}
\def\pbar #1{\overline{\p {#1}}}

\def\inv{^{-1}}

\def\Z{{\mathbb Z}}

\def\Q{{\mathbb Q}}
\def\C{{\mathbb C}}
\def\la{\lambda}
\def\eps{\varepsilon}

\def\ubold{{\bm u}}

\def\lambdabold{{\bm \la}} 
\def\labold{\lambdabold}
\def\mubold{{\bm \mu}}

\def\qbold{{\bm q}}

\def\rhobold{{\bm \rho}}
\def\powerpm{^{\pm 1}} 

\def\deltabold{{\mathbold \delta}}

\def\power #1{^{(#1)}}

\def\End{{\rm End}}

\def\ahec #1{\widehat{H}_{#1}}  
\def\hec #1{H_{#1}}
\def\bmw #1{W_{#1}}
\def\kt #1{{\rm KT}_{#1}}
\def\tl{T}
\def\TL{Temperley--Lieb }
\def\cont #1 #2 #3{C_{#1}^{#2}(#3)}

\theoremstyle{plain}
\newtheorem{theorem}{Theorem}[section]

\theoremstyle{plain}
\newtheorem{proposition}[theorem]{Proposition}

\theoremstyle{plain}
\newtheorem{corollary}[theorem]{Corollary}

\theoremstyle{plain}
\newtheorem{lemma}[theorem]{Lemma}

\theoremstyle{definition}
\newtheorem{definition}[theorem]{Definition}
\theoremstyle{definition}
\newtheorem{example}[theorem]{Example}

\theoremstyle{definition}
\newtheorem{remark}[theorem]{Remark}

\theoremstyle{definition}

\theoremstyle{remark}

\theoremstyle{claim}

\title{Cellularity and the Jones basic construction}

\author{Frederick M. Goodman}
\address{Department of Mathematics\\ University of Iowa\\ Iowa
City, Iowa}
\email{ goodman@math.uiowa.edu} 

\author{John Graber}
\address{Department of Mathematics\\ University of Iowa\\ Iowa
City, Iowa}
\email{ jgraber@math.uiowa.edu}


\dedicatory{Dedicated to Dennis Stanton on the occasion of his $60$th birthday.}

\begin{document}
 \begin{abstract}    We establish a framework for cellularity of algebras related to the Jones basic construction.
Our framework allows a uniform proof of cellularity of Brauer algebras, ordinary and cyclotomic  BMW algebras, walled Brauer algebras, partition algebras, and others.  Our cellular bases are labeled by paths on certain branching diagrams rather than by  tangles.  Moreover, for the class of algebras that we study, we show that the cellular structures are compatible with restriction and induction of modules.

\end{abstract}
 \maketitle
\medskip

To appear in {\em Advances in Applied Mathematics}

\setcounter{tocdepth}{1}
\tableofcontents

\section{Introduction}

Cellularity is a concept due to Graham and Lehrer ~\cite{Graham-Lehrer-cellular}  that is useful for studying non--semisimple specializations of certain algebras such as Hecke algebras, $q$--Schur algebras, etc.     A number of important examples of cellular algebras, including  the Hecke algebras of type $A$ and the Birman--Wenzl--Murakami (BMW) algebras,  actually occur in towers $A_0 \subseteq A_1 \subseteq A_2 \subseteq \dots$ with coherent cellular structures.   Coherence means that the cellular structures are well--behaved with respect to induction and restriction.

This paper establishes a framework for proving cellularity of towers of algebras $(A_n)_{n \ge 0}$ that are obtained by repeated Jones basic constructions from a coherent tower of cellular algebras $(Q_n)_{n \ge 0}$.    

Examples that fit in our framework include:  Temperley-Lieb algebras,  Brauer algebras,  walled Brauer algebras,  Birman--Wenzl--Murakami (BMW) algebras,   cyclotomic BMW algebras,  partition algebras,  and contour algebras.
We give a uniform proof of cellularity for  all of these algebras. 

We should alert the reader that we use a definition of cellular algebras that is slightly weaker than the original definition of  Graham and Lehrer.  The two definitions are equivalent in case $2$ is invertible in the ground ring, and we know of no consequence of cellularity that would not also hold with the weaker definition; in particular, all results of Graham and Lehrer   ~\cite{Graham-Lehrer-cellular}   go through with the modified definition.   See Section \ref{subsection: cellularity} for details.    Our contention is that the relaxed definition is in fact superior, as it allows one to deal more naturally with extensions of cellular algebras.  
For this reason, we have retained the terminology ``cellularity" for our weaker definition, rather than inventing some new terminology such as ``weak cellularity."  

 Once we have proved our abstract result (Theorem \ref
{main theorem}),  it is generally very easy to check that each example fits our framework, and thus that the tower $(A_n)_{n\ge 0}$ in the example is a coherent tower of cellular algebras.  What we need is, for the most part, already in the literature, or completely elementary.  The application of our method to the cyclotomic BMW algebras depends on a very recent result of Mathas regarding induced modules of cyclotomic Hecke algebras ~\cite{mathas-2009}.

For most of our examples,  cellularity has been established previously (but coherence of the cellular structures is  a new result).   Many of the existing proofs of cellularity for these algebras   follow the pattern made explicit by Xi in his paper on cellularity of the partition algebras ~\cite{Xi-Partition}.  The cellular bases obtained are pieced together from  cellular bases of the (quotient) algebras $Q_k$ and bases of certain $R$--modules $V_k$ of tangles or diagrams, where $R$ is the ground ring for $A_n$;  a formal method for piecing the parts together is K\"onig and Xi's method of ``inflation" 
~\cite{KX-Morita}.     It is not evident that the resulting  ``tangle bases''  yield coherent cellular structures.
By contrast, the cellular bases that we produce are indexed by paths on the branching diagram (Bratteli diagram) for the generic semisimple representation theory of the tower $(A_n)_{n \ge 0}$ over a field, and coherence is built into the construction. 

For example,  for the Brauer algebras, the BMW algebras,  and the cyclotomic BMW algebras, 
our  cellular basis of the $n$--th algebra is indexed by up--down tableaux of length $n$, and may be regarded as
an analogue of Murphy's cellular basis ~\cite{murphy-hecke95} for the Hecke algebra,   or the basis of Dipper, James and Mathas ~\cite{dipper-james-mathas} for the cyclotomic Hecke algebras.
 A Murphy type basis for the BMW and Brauer algebras has been constructed  by Enyang ~\cite{Enyang2}, but such a basis for the cyclotomic BMW algebras has not been obtained previously.  It would be fairly involved to extend Enyang's method to the cyclotomic case, but our method applies to  this case without difficulty.

Let us remark on the role played by the generic ground ring for our examples.  For each of our examples $(A_n)_{n \ge 0}$, there is a generic ground ring $R$  such that any specialization  $A_n^S$ to a ground ring $S$ is obtained as $A_n^S = A_n^R \otimes_R S$.  Moreover,  $R$ is an integral domain, and if $F$ denotes the field of fractions of $R$, then the algebras $(A_n^F)_{n \ge 0}$ are split semisimple  with a known representation theory and branching diagram.  It suffices for us to prove that the sequence of  algebras defined over the generic ground ring $R$ is a coherent cellular tower,  and we find that we can use the structure of the algebras defined over $F$ as a tool to accomplish this.

Our approach is influenced by  the work of K\"onig and Xi    ~\cite{KX-Morita} as well as by the work of Cox et.~al. on ``towers of recollement"  ~\cite{cox-towers}.  In fact,  the idea behind our approach is roughly  the following:  Each algebra $A_n$ (over the generic ground ring $R$)  contains an essential idempotent $e_{n-1}$  with the properties that $e_{n-1} A_n e_{n-1} \cong A_{n-2}$  and \break $A_n/(A_n e_{n-1} A_n)  \cong Q_n$,  where $Q_n$ is a cellular algebra.    Assuming that $A_{n-2}$ and $A_{n-1}$ are  cellular, we show that the
(generally non--unital) ideal  $I_n = A_n e_{n-1} A_n$ is a ``cellular ideal"  in $A_n$  by relating ideals of 
$A_{n-2}$ to ideals of $A_n$ contained in $I_n$.   This proof involves a new basis--free characterization of cellularity and also involves showing that $I_n \cong  A_{n-1} \otimes_{A_{n-2}}  A_{n-1}$  as
$A_{n-1}$ bimodules; thus $I_n$ is a sort of Jones basic construction for the  pair $A_{n-2} \subseteq A_{n-1}$.    Since our version of cellularity behaves well under extensions, we can conclude that
$A_n$ is cellular.    Our method is related to ideas introduced by  K\"onig and Xi  in their treatment of 
cellularity and Morita equivalence ~\cite{KX-Morita}.

 Following Cox et.~al.~\cite{cox-towers},  our approach employs the interaction between induction and restriction functors relating $A_{n-1}$--mod and $A_{n}$--mod, on the one hand, and localization and globalization functions relating $A_n$--mod and $A_{n-2}$--mod, on the other hand.  (Write $e = e_{n-1} \in A_n$.  The localization functor $F: A_n\text{--mod} \rightarrow e A_n e\text{--mod} \cong A_{n-2}\text{--mod}$ is $F: M \mapsto e M$.  The globalization function $G: A_{n-2}\text{--mod} \cong e A_n e\text{--mod} \rightarrow A_n\text{--mod}$ is $G: N \mapsto A_n e \otimes_{e A_n e} N$.)  
 
 Our framework and that of  Cox et.~al.  dovetail nicely;  in fact, our main result (Theorem ~\ref{main theorem}) says that if 
 $(A_n)$,  $(Q_n)$ are two sequences of algebras satisfying our framework axioms,  then $(A_n)$  satisfies a cellular version of the axioms for towers of recollement;  see ~\cite{cox-walled-brauer} for a discussion of   cellularity and towers of recollement.

Although our techniques do not seem to be adaptable to proving ``strict" cellularity  in the sense of 
~\cite{Graham-Lehrer-cellular},  by combining our results with previous proofs of ``strict" cellularity for our examples,  we can show the existence of ``strictly" cellular Murphy type bases, i.e. bases indexed by
paths on the generic branching diagram for the sequence of algebras $(A_n)_{n \ge 0}$.     We will indicate how this can be done for the cyclotomic BMW algebras;  other examples are similar.  

Several other  general frameworks have been proposed for cellularity which also successfully encompass many of our examples;  see   ~\cite{KX-Morita, green-martin-tabular, wilcox-cellular}.

In a companion paper ~\cite{GG2}, we refine the framework of this paper to take into account the role played by Jucys--Murphy elements. At the same time, we modify Andrew Mathas's theory  \cite{mathas-seminormal}  of  cellular algebras with Jucys--Murphy elements to take into account coherent sequences of such algebras.

\medskip
\noindent{\bf Acknowledgement.}  Part of this work was done while both authors were visiting MSRI in 2008.  We are grateful to the organizers of the program in Combinatorial Representation Theory and to the staff at MSRI for a  pleasant and stimulating visit.  We thank the referees for helpful suggestions which resulted in several improvements.   

\hfill \newpage

\section{Preliminaries}

\subsection{Algebras with involution}
Let $R$ be a commutative ring with identity.  In the following, assume $A$ is an $R$--algebra with an involution $i$  (that is, an $R$--linear algebra  anti--automorphism of $A$ with \def\id{{\rm id}} $i^2 = \id$). 

If $M$ is a left $A$--module, we define a right $A$--module $i(M)$ as follows.   As a set, $i(M)$ is a copy of $M$, with elements marked with the symbol $i$,  $i(M) = \{i(m) : m \in M\}$. The $R$--module structure of $i(M)$ is given by $i(m_1) + i(m_2) = i(m_1 + m_2)$, and $r i(m) = i(r m)$.  Finally,  the right $A$--module structure is defined by
$i(m) a = i( (i(a) m)$.  If $\alpha : M \to N$ is a homomorphism of left $A$--modules, define $i(\alpha) : i(M) \to i(N)$ by $i(\alpha)(i(m)) = i(\alpha(m))$.  Then $i :  A\text{--mod} \to \text{mod--}A$ is a functor. 
For any fixed $M$, 
$i : M \to i(M)$ given by $m \mapsto i(m)$ is, by definition, an isomorphism of $R$--modules.

If $\Delta$ is a left ideal in $A$,  we have two possible meanings for $i : \Delta \to i(\Delta)$, namely the restriction to $\Delta$ of the involution $i$, whose image is a right ideal in $A$,  or the application of the functor $i$.  However, there is no problem with this, as the right $A$--module obtained by applying the functor $i$ can be identified with the
right ideal $i(\Delta)$.

 The same construction gives a map from right $A$--modules to left $A$--modules.  Moreover, if 
$A$ and $B$ are $R$--algebras with involutions $i_A$ and $i_B$, and     $M$ is an $A$--$B$--bimodule,   then
$i(M)$, defined as above as an $R$--module has the structure of a $B$--$A$--bimodule with
$b\, i(m) a =i( i_A(a) m \, i_B(b))$.   
Note that $i\circ i(M)$ is naturally isomorphic to $M$, so $i$ is an equivalence between the categories of 
$A$--$B$--bimodules and the category of   $B$--$A$--bimodules.

\begin{lemma} \label{lemma; involutions and tensor products of bimodules}
 Suppose $A$, $B$, and $C$  are $R$--algebras with involutions  $i_A$,  $i_B$,  and $i_C$.  Let
$_B P_A$ and $_A Q_C$ be bimodules.   Then
$$
i(P \otimes_A Q) \cong i(Q) \otimes_A i(P),
$$
as $C$--$B$--bimodules.
\end{lemma}

\begin{proof}  It is straightforward to check that there is a well defined $R$--linear isomorphism
$f_0 : P \otimes_A Q \to  i(Q) \otimes_A i(P)$ such that $f_0(p \otimes q) = i(q) \otimes i(p)$.   Then
$$f = f_0 \circ i\inv : i(P\otimes_A Q) \to   i(Q) \otimes_A i(P)$$ is an $R$--linear isomorphism.  Finally, one can check that $f$ is a $C$--$B$--bimodule map.
\end{proof}

\begin{remark}  \label{remark:  i applied to tensor product}
Note that if we identify $i(P \otimes_A Q)$ with $i(Q) \otimes_A i(P)$ via $f$, then we have the 
formula $i(p \otimes q) = i(q) \otimes i(p)$.  In particular,  let $M$ be a $B$--$A$--bimodule,  and identify
$i\circ i(M)$ with $M$, and $i(M\otimes_A i(M))$  with $i\circ i(M) \otimes_A i(M) = M\otimes_A i(M)$.
Then we have the formula $i(x \otimes i(y)) = y \otimes i(x)$.  We will use these identifications throughout the paper.
\end{remark}

\subsection{Cellularity} \label{subsection: cellularity}

We recall the definition of {\em cellularity}  from ~\cite{Graham-Lehrer-cellular}; see also
~\cite{Mathas-book}.   The version of the definition given here is slightly weaker than the original definition in ~\cite{Graham-Lehrer-cellular}; we justify this below.

\begin{definition}  \label{gl cell}  Let $R$ be an integral domain and $A$ a unital $R$--algebra.  A {\em cell datum} for $A$ consists of  an algebra involution $i$ of $A$; a partially ordered set $(\Lambda, \ge)$ and 
for each $\la \in \Lambda$  a set $\mathcal T(\lambda)$;  and   a subset $
\mathcal C = \{ c_{s, t}^\la :  \la \in \Lambda \text{ and }  s, t \in \mathcal T(\la)\} \subseteq A$; 
with the following properties:
\begin{enumerate}
\item  $\mathcal C$ is an $R$--basis of $A$.
\item   \label{mult rule} For each $\la \in \Lambda$,  let $\breve A^\la$  be the span of the  $c_{s, t}^\mu$  with
$\mu > \la$.   Given $\la \in \Lambda$,  $s \in \mathcal T(\la)$, and $a \in A$,   there exist coefficients 
$r_v^s( a) \in R$ such that for all $t \in \mathcal T(\la)$:
$$
a c_{s, t}^\la  \equiv \sum_v r_v^s(a)  c_{v, t}^\la  \mod  \breve A^\la.
$$
\item  $i(c_{s, t}^\la) \equiv c_{t, s}^\la   \mod  \breve A^\la$ for all $\la\in \Lambda$ and, $s, t \in \mathcal T(\lambda)$.

\end{enumerate}
$A$ is said to be a {\em cellular algebra} if it has a  cell datum.  
\end{definition}

For brevity,  we will write that  $(\mathcal C, \La)$ is a cellular basis of $A$.

\begin{remark} \mbox{} \label{remark:  on definition of cellularity}
\begin{enumerate}
\item  The original definition in  ~\cite{Graham-Lehrer-cellular} requires that $i(c_{s, t}^\la) = c_{t, s}^\la $ for all $\la, s, t$.  However, one can check that  the results of \cite{Graham-Lehrer-cellular} remain valid with our weaker axiom.
In fact, we are not aware of any consequence of cellularity that would not also hold with our weaker definition. 
\item  In case $2 \in R$ is invertible, our definition is equivalent to the original.  Here is the proof:  Suppose that $2$ is invertible in the ground ring and that
$\{c_{s, t}^\la\}$ is a cellular basis in the sense of Definition \ref{gl cell}.  We want to produce a new cellular basis $\{a_{s, t}^\la\}$ satisfying the strict equality $i(a_{s, t}^\la) = a_{t, s}^\la $ for all $\la, s, t$.  By hypothesis, for each $\la, s, t$ there is a unique $f(\la, s, t) \in \breve A^\la$  such that  $i(c_{s, t}^\la) = c_{t, s}^\la + f(\la, s, t)$.   One easily checks that $i(f(\la, s, t)) = -f(\la, t, s)$.    Declare
$a_{s, t}^\la = c_{s, t}^\la + (1/2) f(\la, t, s)$ for all $\la, s, t$.  
Then $\{a_{s, t}^\la\}$ has the desired properties.  
\end{enumerate}
\end{remark}

We recall some basic structures related to cellularity, see ~\cite{Graham-Lehrer-cellular}.
Given $\la\in\La$.  Let $A^\la$ denote the span of the $c_{s,t}^{\mu}$ with $\mu \geq \la$.  It follows that both $A^\la$ and $\breve A^\la$ (defined above) are $i$--invariant two sided ideals of $A$.
If $t \in \mathcal T(\la)$, define $C_t^\la$ to be the $R$-submodule of $A^\la/\uA$ with basis $\{ c_{s,t}^\la + \uA : s \in \mathcal T(\la) \}$.  Then $C_t^\la$ is a left $A$-module by Definition \ref{gl cell} (\ref{mult rule}).  Furthermore, the action of $A$ on $C_t^\la$ is  independent of $t$, i.e $C_u^{\la}\cong C_t^{\la}$  for any $u,t \in \mathcal T(\la)$.    The {\em left  cell module} $\Delta^\la$ 
is defined as follows: as   an $R$--module, $\Delta^\la$  is free with basis  $\{c_s^\la$ : $s \in \mathcal T(\la)\}$;   for each $a \in A$, the action of $a$ on $\Delta^\la$ is defined by  $ ac_s^\la=\sum_v r_v^s(a)  c_v^\la$ where $r_v^s(a)$ is as  in Definition \ref{gl cell} (\ref{mult rule}).  Then $\Delta^\la \cong C_t^\la$, for any $t \in \mathcal T(\la)$.
 For all $s,t \in \mathcal T(\la)$, we have a canonical $A-A$--bimodule isomorphism $\alpha : A^\la/\uA \rightarrow \Delta^\la \otimes_R i(\Delta^\la)$ defined by $\alpha(c_{s,t}^{\la}+\uA)=c_s^\la \otimes_R  i(c_t^\la)$.  Moreover, we have
 $i \circ \alpha = \alpha \circ i$,  using Remark \ref{remark:  i applied to tensor product} and point (3) of Definition \ref{gl cell}.

\begin{definition}  Suppose $A$ is a unital $R$--algebra with involution $i$, and $J$ is an $i$--invariant ideal; then we have an induced algebra involution $i$ on $A/J$. 
Let us say that $J$ is a {\em cellular ideal} in $A$ if it satisfies the axioms for a  cellular algebra (except for being unital) with cellular basis 
$$ \{ c_{s, t}^\la :  \la \in \Lambda_J \text{ and }  s, t \in \mathcal T(\la)\} \subseteq J$$
and we have, as in point (2) of the definition of cellularity, 
$$
a c_{s, t}^\la  \equiv \sum_v r_v^s(a)  c_{v, t}^\la  \mod  \breve J^\la
$$
not only for $a \in J$ but also for $a \in A$.
\end{definition}

\begin{remark}  \label{remark on extensions of cellular algebras} (On extensions of cellular algebras.) 
If $J$ is a cellular ideal in $A$, and 
$H = A/J$ is  cellular  (with respect to the involution induced from the involution on $A$), then $A$ is cellular.   In fact, let $(\La_J, \ge)$ be the partially ordered set in the cell datum for $J$ and $\mathcal C_J$ the cellular basis.
Let $(\La_H, \ge)$  be the partially ordered set in the cell datum for $H$ and $\{\bar h_{u, v}^\mu\}$ the cellular basis.  Then $A$ has a cell datum with partially ordered set $\La = \La_J \cup \La_H$,  with partial order agreeing with the original partial orders on $\La_J$ and on $\La_H$ and with $\la > \mu$ if $\la \in \La_J$ and $\mu \in \La_H$.
A cellular basis of $A$ is $\mathcal C_J \cup \{h_{s, t}^\mu\}$, where $h_{s, t}^\mu$ is any lift of $\bar h_{s, t}^\mu$.

With the original definition of ~\cite{Graham-Lehrer-cellular}, the assertions of this remark  would be valid only  if the ideal
$J$ has an $i$--invariant $R$--module complement in $A$.
The ease of handling extensions is  our motivation for using the  weaker definition of cellularity.
\end{remark}

\subsection{Basis--free formulations of cellularity}
K\"onig and Xi have given a basis--free definition of cellularity ~\cite{KX-Morita}.  We describe a slight weakening of their definition, which corresponds exactly to our weaker form of Graham--Lehrer cellularity

\begin{definition}[K\"onig and Xi] \label{defKX}  Let $R$ be an integral domain and $A$ a unital $R$-algebra
with  involution $i$.    An $i$--invariant two sided ideal $J$ in $A$ is called a {\em split ideal} if, and only if,
there exists a left ideal $\Delta$ of $A$ contained in  $J$, with  $\Delta$  finitely generated and free over $R$, and 
there is an isomorphism of $A$--$A$--bimodules $\alpha : J \rightarrow \Delta \otimes_R i(\Delta)$  making the following diagram commute:
\begin{diagram}
J	&\rTo^{\alpha}	&&\Delta \otimes_R i(\Delta)\\
\dTo_{i} &&& \dTo_{i}\\
J	&\rTo^{\alpha}	&&\Delta \otimes_R i(\Delta)\\
\end{diagram}

\ignore{
$$\begin{CD}
J	@>\alpha>>	\Delta \otimes_R i(\Delta)\\
@VViV			@VViV\\
J	@>\alpha>>		\Delta \otimes_R i(\Delta).
\end{CD}$$
}
A finite chain of $i$--invariant  two sided ideals
$$0=J_0 \subset J_1 \subset J_2 \subset \cdots \subset J_n = A$$
is called a {\em cell chain} if for each $j$  ($1 \leq j \leq n$),   the quotient $J_j/J_{j-1}$ is a split ideal  of $A/J_{j-1}$ (with respect to the involution induced by $i$ on $A/J$). 
\end{definition} 

\begin{remark}  \label{remark on Konig Xi definition}
\mbox{}
\begin{enumerate}
\item  K\"onig and Xi call a split ideal a ``cell ideal."   We changed the terminology to avoid confusion with other concepts.  
\item  The definition of a cell chain differs from the one given by Konig and Xi in that we dropped the  requirement that $J_{j-1}$ have an $i$-invariant $R$-module complement in $J_j$.
\end{enumerate}
\end{remark}

\begin{lemma} Let $R$ be an integral domain and let $A$ be  a unital $R$--algebra with involution $i$.
An ideal  $J$ of $A$ is split if, and only if, there exists a left $A$--module $M$ that is finitely generated and free as an $R$--module, and there exists an
isomorphism of $A$--$A$--bimodules $\gamma : J \rightarrow M \otimes_R i(M)$  making the following diagram commute:
\begin{diagram}
J &\rTo^{\gamma} &&M \otimes_R i(M)\\
\dTo_{i} &&& \dTo_{i}\\
J &\rTo^{\gamma} &&M \otimes_R i(M)\\
\end{diagram}
\end{lemma}

\begin{proof}  If $J$ is split, it clearly satisfies the condition of the lemma.  Conversely,  suppose the condition of the lemma is satisfied.    Fix some element $b_0$ of the basis of $M$ over $R$ and define a left $A$--module map
$\beta : M \to A$ by $\beta(m) = \gamma\inv(m \otimes b_0)$.  Then $\beta$ is an isomorphism of $M$ onto a left
ideal $\Delta$ of $A$ contained in $J$.

\ignore{
Regard $M$ as an $A$--$R$--bimodule.
As in Remark \ref{remark:  i applied to tensor product},   identify $i\circ i(M)$ with $M$ and
$i(M \otimes_R i(M))$ with $M \otimes_R i(M)$.   Then the map $x \otimes i(y) \mapsto y \otimes i(x)$ is just
$i : M \otimes_R i(M) \to M \otimes_R i(M)$.  The same considerations apply to $\Delta$ as well.
}

Now we have $\beta \otimes i(\beta) :  M \otimes_R i(M) \to \Delta \otimes_R i(\Delta)$
is an isomorphism satisfying $(\beta \otimes i(\beta)) \circ i = i \circ (\beta \otimes i(\beta))$.  It follows that
$\alpha = (\beta \otimes i(\beta)) \circ \gamma : J \to \Delta \otimes_R i(\Delta)$ is an isomorphism of 
$A$--$A$--bimodules satisfying the requirement for a split ideal, namely, $ \alpha \circ i = i \circ \alpha$.
\end{proof}

\begin{lemma}[K\"onig and Xi] \label{lemma:  equivalence of GL and KX} Let $A$ be an $R$--algebra with involution.
$A$ is {cellular} if, and only if,  $A$ has a finite cell chain.
\end{lemma}

\begin{proof}  We sketch the proof from  ~\cite{KX-structure}, p.\ 372.

Suppose $A$ has a cell datum with partially ordered set $(\La, \ge)$ and cell basis $\{ c_{s, t}^\la\}$.  
Write $\La$ as a sequence $(\la_1, \la_2, \dots, \la_n)$, where $\la_1$ is maximal in $\La$,  and, for $1 \le j < n$, 
$\la_{j+1}$ is maximal in $\La \setminus \{\la_1, \dots, \la_j\}$. Then for each $j\ge 1$,  $\Ga_j = \{\la_1, \dots, \la_j\}$ is an order ideal in $\La$.  Set $\Ga_0 = \emptyset$.
 Define $A(\Ga_j)$ to be the $R$-submodule of $A$ 
spanned by the basis elements $c_{u, v}^\la$,  with $\la \in \Ga_j$.
Then $A(\Ga_j)$ is an $i$--invariant  two sided ideal in $A$, and
$$0=A(\Ga_0)\subset A(\Ga_1)\subset\cdots\subset A(\Ga_{n})=A.$$
Moreover (see ~\cite{Graham-Lehrer-cellular}, p. 6), 
$$A(\Ga_j)/A(\Ga_{j-1})\cong
A^{\la_j}/\breve A^{\la_j}     \cong
\Delta^{\la_j}\otimes_Ri(\Delta^{\la_j}),$$
and the isomorphism $\alpha : A(\Ga_j)/A(\Ga_{j-1}) \to \Delta^{\la_j}\otimes_Ri(\Delta^{\la_j})$ satisfies
$\alpha \circ i =  i \circ \alpha$.  Thus $(A(\Ga_j))_{1 \le j \le n}$ is a cell chain.

Conversely,  suppose $(J_j)_{0 \le j \le n}$ is a cell chain in $A$.   Then for each $j \ge 1$, we have an 
$A$--module $\Delta_j$  that is finitely generated and  free as an $R$--module, and an isomorphism of $A$--$A$--bimodules
$\alpha_j : J_j/J_{j-1} \to \Delta_j \otimes_R i(\Delta_j)$ satisfying $i \circ \alpha_j = \alpha_j \circ i$.
 Let $\{b^j_s : s\in \mathcal T(j)\}$ be an $R$--basis of $\Delta_j$ and let
$c_{s, t}^{\la_j}$ be any lift in $J_j$ of $\alpha_j\inv(b^j_s \otimes i(b^j_t))$.  Now take $\La'$ to be $\La$ with the order $\la_1 > \la_2 > \cdots > \la_n$.  Let  $\mathcal C = \{c_{s, t}^{\la_j} : 1\le j \le n; \ s, t \in \mathcal T(j)\}$.   Then 
$(\mathcal C, \La')$ is a cellular basis of $A$.
\end{proof}

\begin{remark}  In the Lemma, $A$ has a cellular basis $\{c_{s, t}^\la\}$  with
$i(c_{s, t}^\la) = c_{t, s}^\la$ if, and only if,  $A$ has a finite cell chain $(J_j)$  such that
for each $j\ge 1$,  $J_{j-1}$  has an $i$--invariant $R$--module complement in $J_j$.  
\end{remark}

Note that if we follow the procedure of the proof, starting with a cell datum on $A$ with partially ordered set
$(\La, \ge)$,  then the only information that we retain about $\La$ is that $\la_{j+1}$ is maximal in 
$\La \setminus \Ga_j$;  we cannot recover the partial order on $\La$ from this.   Moreover, if we continue to
produce a cellular basis $\{c_{s, t}^j\}$ from the cell chain $(A(\Ga_j))_{0 \le j \le n}$, the result will not necessarily have the properties of a cellular basis with respect to the original partially ordered set $(\La, \ge)$.

In order to prove our main results, we will need a different basis--free formulation of cellularity that allows us to pass back and forth between
the formulation of Definition \ref{gl cell} and the basis--free formulation without losing information about the partially ordered set.

\begin{definition} \label{definition: cell net}
Let $A$ be an $R$--algebra with involution $i$.  Let $(\La, \ge)$ be a finite partially ordered set. For $\la \in \La$, let
 $\Ga_{\ge \la}$ denote the order ideal $\{\mu : \mu \ge \la\}$  and  $\Ga_{> \la}$  the order ideal $\{\mu : \mu > \la\}$.
 
A  {\em $\La$--cell net}  is a map from the set of order ideals of $\La$ to the set of $i$--invariant two sided ideals of $A$, 
$\Ga \mapsto A_\Ga$, with the following properties:
\begin{enumerate}
\item $A_\emptyset = \{0\}$.  If $\Ga_1 \subseteq \Ga_2$, then $A_{\Ga_1} \subseteq A_{\Ga_2} $.
\item  For $\la \in \La$, write $A_{\ge \la} = A_{\Ga_{\ge \la}}$ and  $A_{> \la} = A_{\Ga_{> \la}}$.  Then
$$A =  \spn\{A_{\ge \mu} :  \mu \in \La \},$$    and for all  $\la \in \La$,      $$A_{> \la} =\spn \{ A_{\ge \mu} :  \mu > \la\}.$$
\ignore{
\item  If $\mu$ is not comparable with $\la_i$  ($1 \le i \le s$), then
$$A_{\ge \mu}  \cap  \spn\{A_{\ge \la_i}  :  1 \le i \le s\} \subseteq A_{> \mu}.$$
}
\item For each $\la \in \La$, there is an $A$--module $M^\la$, finitely generated and free as an $R$--module, such that whenever $\Ga \subseteq \Ga'$ are order ideals  of $\La$, with $\Ga' \setminus \Ga = \{\la\}$,  then there exists an
isomorphism of $A$--$A$--bimodules $$\alpha : A_{\Ga'}/A_{\Ga} \to M^\la \otimes_R i(M^\la),$$  satisfying
$i \circ \alpha = \alpha \circ i$.
\end{enumerate}
\end{definition}

\begin{proposition} \label{lemma: cell net characterization of cellularity}
 Let $A$ be an $R$--algebra with involution, and let $(\La, \ge)$ be a finite partially ordered set.
  Then $A$ has a cell datum with partially ordered set $\La$ if, and only if, $A$ has a  $\La$--cell net.
\end{proposition}

\begin{proof}  Suppose that $A$ has a cell datum with partially ordered set $\La$ and cell basis
$\{c_{s, t}^\la\}$.   For each order ideal $\Ga$ of $\La$,  let $A(\Ga)$ denote the span of those $c_{s, t}^\la$ with
$\la \in \Ga$.  Then $\Ga \mapsto A(\Ga)$ is a $\La$--cell net.

Conversely, suppose that $A$ has a $\La$--cell net, $\Ga \mapsto A_\Ga$.  For each $\la \in \La$, we have
an isomorphism of  $A$--$A$--bimodules $\alpha_\la : A_{\ge \la}/A_{ > \la}  \to M^\la \otimes_R i(M^\la)$.
Let $\{b_s^\la : s \in \mathcal T(\la)\}$ be an $R$--basis of $M^\la$ and let $c_{s, t}^\la$ be any lift of
$\alpha_\la\inv(b_s^\la \otimes i(b_t^\la))$ to $A_{\ge \la}$.  We claim that $$\mathcal C = \{c_{s, t}^\la : \la \in \La; s, t \in \mathcal T(\la)\}$$ is an $R$--basis of $A$.  

Let $A^\la$ be the span of those $c_{s, t}^\mu$ with $\mu \ge \la$ and $\breve A^\la$ the span of  those $c_{s, t}^\mu$ with $\mu > \la$.    If $\mu \ge \la$,  then for all $s, t \in \mathcal T(\mu)$,  $c_{s, t}^\mu \in A_{\ge \mu} \subseteq A_{\ge \la}$, using point (1) of Definition \ref{definition: cell net}.
Hence $A^\la \subseteq A_{\ge \la}$.  Similarly, $\breve A^\la \subseteq A_{> \la} $.

We claim that 
\begin{equation} \label{equation:  equality of ideals related to order ideals}
\text{for all} \ \la \in \La,  \quad A_{\ge \la} = A^\la.
\end{equation}
This is clear if $\la$ is a maximal element of $\La$.   (Note that $A_{> \la} = A_\emptyset = \{0\}$.)
 Now suppose that $\la$ is not maximal and that for all
 $\mu > \la$,   $A_{\ge \mu} = A^\mu$.  Then $$A_{> \la} = \spn\{A_{\ge \mu} : \mu > \la\}
 = \spn\{A^\mu : \mu > \la\} = \breve A^\la,$$
 where the first equality comes from (2) of Definition \ref{definition: cell net} and the second from the induction hypothesis.  By definition of $\{c_{s, t}^\la\}$, we have $$A_{\ge \la} = \spn\{c_{s, t}^\la\} + A_{> \la} =  \spn\{c_{s, t}^\la\} + \breve A^\la = A^\la.$$
 Assertion (\ref{equation:  equality of ideals related to order ideals}) now follows by induction. Point (2) of Definition \ref{definition: cell net} and  (\ref{equation:  equality of ideals related to order ideals}) imply that
  $A_{> \la} = \breve A^\la$ for all $\la \in \La$, and that $A = \spn(\mathcal C)$.
  
We now proceed to establish linear independence of $\mathcal C$.  
Write $\La$ as a sequence $(\la_1, \la_2, \dots, \la_K)$  with 
$\la_1$ maximal and $\la_{j+1}$ maximal in $\La \setminus \{\la_1, \dots, \la_j\}$ for  $1 \le j < K$.  Put 
$\Ga_j = \{\la_1, \dots, \la_j\}$ for $j \ge 1$ and $\Ga_0 = \emptyset$.   Then $(\Ga_j)_{0 \le j \le K}$ is a maximal chain of order ideals.  Since $\Ga_j \setminus \Ga_{j-1} = \{\la_j\}$, we have an isomorphism $\gamma_j : A_{\Ga_j}/A_{\Ga_{j-1}} \to M^{\la_j} \otimes_R 
i(M^{\la_j})$ with $i \circ \gamma_j = \gamma_j \circ i$.  Thus $(A_{\Ga_j})_{0 \le j \le K}$ is a cell chain
in $A$.  So by the proof of Lemma \ref{lemma:  equivalence of GL and KX},   $A$ has a cellular basis
$$\mathcal B = \{ b_{s, t}^\la : \la \in \La;\  s, t, \in \mathcal T(\la)\},$$  but with respect to the ``wrong" partial order on $\La$.   Since $\mathcal C$ is a spanning set of the same cardinality as the basis
$\mathcal B$,  it follows that $\mathcal C$ is linearly independent over $R$, and thus an $R$--basis of
$A$.  

\ignore{
We now proceed to establish linear independence of $\mathcal C$, making use of condition (3) of Definition   \ref{definition: cell net}.
  Suppose we have a non--trivial linear relation $\sum_{\la, s, t}  r(\la, s, t) c_{s, t}^\la = 0$ with coefficients in $R$.  Let $\mu$ be minimal among those $\la$ such that some $r(\la, s, t)$ is non--zero.  Rewrite the linear relation as
$\Sigma' + \Sigma'' + \Sigma''' = 0$, where $\Sigma'$ is the sum of those terms with $\la$ not comparable to $\mu$,
\ $\Sigma''$ is the sum of those terms with $\la > \mu$, and $\Sigma'''$ is the sum of those terms with $\la = \mu$.
Then $\Sigma'' + \Sigma'''  \in A^{\mu} = A_{\ge \mu}$.  Hence also $\Sigma' \in A_{\ge \mu}$.  But
 $\Sigma' $ is also in the span of those $A_{\ge \la}$ with $\la$ not comparable to $\mu$.  By point (3) of 
 Definition \ref{definition: cell net},  we have $\Sigma' \in A_{> \mu}$.  But then $\Sigma' + \Sigma'' \in  A_{> \mu}$. 
 Taking the quotient by $A_{> \mu}$, we get $\sum_{s, t} r(\mu, s, t) ( c_{s, t}^\mu + A_{> \mu}) = 0$.
 Since the set of $( c_{s, t}^\mu + A_{> \mu})$ is a basis of $A_{\ge \mu}/A_{ > \mu} $, it follows that all the 
 coefficients $r(\mu, s, t)$ are zero,  a contradiction.  
 }
 
 Because $A_{> \la} = \breve A^\la$ for all $\la \in \La$, it is now easy to see that properties (2) and (3) of Definition \ref{gl cell} are satisfied by $\mathcal C$.
\end{proof}

\begin{remark} \label{remark: conditions for cell net to give strict cellular basis}
 In the Proposition, the following are equivalent:
\begin{enumerate}
\item  $A$  has a cellular basis $\{c_{s, t}^\la\}$  with
$i(c_{s, t}^\la) = c_{t, s}^\la$.
\item $A$ has a $\Lambda$  cell net $\Gamma \to A_\Gamma$ such that for each pair $\Gamma \subseteq \Gamma'$,  $A_\Gamma$ has an $i$--invariant $R$--module complement in $A_{\Gamma'}$.  
\item  $A$ has a $\Lambda$  cell net $\Gamma \to A_\Gamma$ such that for each $\la \in \Lambda$, 
$A_{> \la}$  has an $i$--invariant $R$--module complement in $A_{\ge \la}$.  
\end{enumerate}
The implications (1) $\implies$  (2)  $\implies$  (3) are evident.   For (3) $\implies$ (1),  let $B_\la$ denote the $i$--invariant $R$--module complement of  $A_{> \la}$ in $A_{\ge  \la}$, and,  
in the 2nd
paragraph of the proof of the Proposition,    let  $c_{s, t}^\la$ be the unique lift of
$\alpha_\la\inv(b_s^\la \otimes i(b_t^\la))$ in $B_\la$.  
\end{remark}

\subsection{Coherent towers of cellular algebras}
\begin{definition}
Let $H_0 \subseteq H_1 \subseteq H_2 \subseteq \cdots$ be an increasing sequence of cellular algebras, with a common multiplicative identity element,  over an integral domain $R$.   Let $\Lambda_n$ denote the partially ordered set in the cell datum for $H_n$.   We say that $(H_n)_{n \ge 0}$ is a {\em  coherent tower of cellular algebras} if the following conditions are satisfied:
\begin{enumerate}
\item  The involutions are consistent; that is,  the involution on $H_{n+1}$,  restricted to $H_n$, agrees with the involution on $H_n$.
\item  For each $n\ge 0$ and for each $\la \in \Lambda_n$, the induced module $\Ind_{H_n}^{H_{n+1}} (\Delta^\la)$
has a filtration by cell modules of $H_{n+1}$. That is, there is a filtration
$$
\Ind_{H_n}^{H_{n+1}} (\Delta^\la) = M_t \supseteq M_{t-1} \supseteq \cdots \supseteq M_0 = (0)
$$
such that for each $j\ge1$,  there is a $\mu_j \in \Lambda_{n+1}$  with $M_j/M_{j-1} \cong \Delta^{\mu_j}$.
\item  For each $n\ge 0$ and for each $\mu \in \Lambda_{n+1}$, the restriction  $\Res_{H_n}^{H_{n+1}} (\Delta^\mu)$
has a filtration by cell modules of $H_{n}$. That is, there is a filtration
$$
\Res_{H_n}^{H_{n+1}} (\Delta^\mu) = N_s \supseteq N_{s-1} \supseteq \cdots \supseteq N_0 = (0)
$$
such that for each $i\ge1$, there is a $\la_i \in \Lambda_{n}$  with $N_j/N_{j-1} \cong \Delta^{\la_i}$.

\end{enumerate}
\end{definition}

The modification of the definition for a {\em finite} tower of cellular algebras is obvious.

We call a filtration as in (2) and (3) a {\em cell filtration}.
In the examples that we study, we will also have {\em uniqueness of the multiplicities} of the cell modules appearing as subquotients of the cell filtrations, and {\em Frobenius reciprocity} connecting the multiplicities in the two types of filtrations.  We did not include uniqueness of multiplicities and Frobenius reciprocity as requirements in the definition, as they will follow from additional assumptions that we will impose later; see Lemma \ref{lemma: multiplicities in cell filtrations}.\footnote{Hemmer and Nakano ~\cite{Hemmer-Nakano}  have obtained remarkable general results about uniqueness of multiplicities in Specht filtrations of modules over Hecke algebras of type A.   Hartmann and Paget \cite{Hartmann-Paget}  obtained analogous results for modules over Brauer algebras.  The assertions that we require here are much more special, applying only to induced modules of cell modules and restrictions of cell modules. }

\begin{example} \label{example: Hn coherent tower} {\em  The tower of Hecke algebras of type $A$  is a coherent tower of cellular algebras.}
Let $R$ be an integral domain  and $q$ an invertible element of $R$.  Let $H_n(R, q)$ denote the Hecke algebra of type $A$  generated by elements $T_1, \dots, T_{n-1}$ satisfying the braid relations and the quadratic relations   $(T_j-q)(T_j + 1) = 0$ for $1 \le j \le n-1$.  When $q = 1$,  $H_n(R, q)$ is the group algebra $R \S_n$  of the symmetric group $\S_n$.  
As is well known, $H_n(R, q)$ has  a basis $T_w$  ($w \in \S_n$) given by 
$T_w = T_{j_1} \dots T_{j_\ell}$ for any reduced expression $w = s_{j_1} \dots s_{j_\ell}$.
The map defined by $i(T_w) = T_{w\inv}$ is an algebra involution.   The map  defined by $(T_w)^\# =  (-q)^{\ell(w)}(T_{w\inv})\inv$ is an algebra  automorphism.   The assignment $T_w \mapsto T_w$ is an embedding of $H_n(R, q)$ into
$H_{n+1}(R, q)$.  The algebra involutions are consistent on $(H_n)_{n \ge 0}$.

Dipper and James ~\cite{dipper-james1, dipper-james2}  studied the representation theory of the Hecke algebras,  defining Specht modules  $S^\la$  which generalize Specht modules for symmetric groups. 
They showed that induced modules of Specht modules have a filtration by Specht modules
~\cite{dipper-james1}.   Jost  ~\cite{jost}  showed that restrictions of Specht modules have Specht filtrations.
 
 Murphy ~\cite{murphy-hecke95}  showed that the Hecke algebras are cellular  (before the formalization of the notion of cellularity in ~\cite{Graham-Lehrer-cellular}).   Murphy shows that his cell modules $\Delta^\la$ satisfy
 $\Delta^\la \cong  ({S^{\la'}})^\#$,  where $\la'$  is the transpose of $\la$  and the superscript $\#$  means that the module is twisted by the automorphism $\#$.    Thus it follows from the results of Dipper, James, and Jost cited above that restricted modules and induced modules of Murphy's cell modules have cell filtrations.
\end{example}

\subsection{Inclusions of split semisimple algebras and branching diagrams}
 A general \break source for the material in this section is ~\cite{GHJ}.    
 
 A finite dimensional split semisimple algebra over a field $F$ is one which is isomorphic to a finite direct sum of full matrix algebras over $F$.
 
Suppose
$A \subseteq B$ are finite dimensional split semisimple algebras over  $F$ (with the same identity element).  Let   $A(i)$,  $i \in I$, be the minimal ideals of $A$  and  $B(j)$,  $j \in J$,   the minimal ideals of $B$.  
We associate a  $J \times I$   {\em inclusion matrix}
$\Omega$ to the inclusion $A \subseteq B$, as follows.  Let $W_j$ be a simple $B(j)$--module.
Then $W_j$ becomes an $A$--module via the inclusion,  and $\Omega(j, i)$ is the multiplicity of  a simple $A_i$--module 
 in the decomposition of $W_j$ as an $A$--module.   
 An equivalent characterization of the inclusion matrix is the following.   Let $q_i$ be a minimal idempotent in $A(i)$  and let $z_j$  be the identity of 
$B(j)$  (a minimal central idempotent in $B$).  Then $q_i z_j$ is the sum of
 $\Omega(j, i)$ minimal idempotents in $B(j)$.

It is convenient to encode an inclusion matrix   by a bipartite graph, called the {\em branching diagram};  the branching diagram has vertices labeled by $I$  arranged on one horizontal line,  vertices labeled by  $J$  arranged along a second (higher) horizontal line,    and $\Omega(j, i)$ edges connecting
$j \in J$ to $i \in I$.

If $A_1 \subseteq A_2 \subseteq A_3  \cdots$ is a (finite or infinite) sequence of inclusions of finite dimensional split semisimple algebras over $F$,  then the branching diagram for the sequence is obtained by stacking the branching diagrams for each inclusion,  with the 
upper vertices of the diagram for $A_i \subseteq A_{i+1}$  being identified with the lower vertices of the diagram for $A_{i+1} \subseteq A_{i+2}$.

For our purposes, it suffices to restrict our attention to the case that $A_0 \cong F$.  In most of our examples,
 the entries in each inclusion matrix are all $0$ or $1$;  thus in the branching diagram there are no multiple edges between vertices.

\begin{definition} \label{def of branching}  An (infinite) abstract branching diagram $\B$  is an infinite graph  with vertex set
$V =   \coprod_{i \ge 0} V_i$, with the following properties
\begin{enumerate}
\item $V_0$ is a singleton and $V_i$ is finite for all $i$.
\item Two vertices $v \in V_i$ and $w \in V_j$ are adjacent only if $|i - j|=1$.  Multiple edges are allowed between adjacent vertices.  
\item  If $i \ge 1$ and $v \in V_i$,  then $v$ is adjacent to at least one vertex in $V_{i-1}$ and to at least one vertex in $V_{i+1}$.
\end{enumerate}
\end{definition}

The definition can be modified in the obvious way for a {\em finite} abstract branching diagram.  
When we treat the walled Brauer algebra in Section \ref{subsection: walled Brauer algebras}, we will loosen the definition by dropping the requirement that $V_0$ is a singleton.

The branching diagram for a sequence  of finite dimensional split semisimple algebras (with the restrictions mentioned above) is an abstract branching diagram, and conversely, given an abstract branching diagram $\B$, one can construct  a sequence  of finite dimensional split semisimple algebras (over any given field)  whose branching diagram is (isomorphic to) $\B$.

Let $\B$ be an abstract branching diagram with vertex set $V =   \coprod_{i \ge 0} V_i$.  We usually denote the unique element of $V_0$ by $\emptyset$.  We picture $\B$ with the elements of $V_i$ arranged on the horizontal line $y = i$ in the plane, and we call $V_i$ the $i$--th {\em row} of vertices in $\B$.  If $v \in V_i$ and $w \in V_{i+1}$ are adjacent, we write  $v \nearrow w$.  The subgraph of $\B$ consisting of $V_i$ and $V_{i+1}$ and the edges connecting them is called the $i$--th  {\em level} of $\B$.

Now suppose we are given an abstract branching diagram $\B_0$  with vertex set
\def\spp #1{^{(#1)}}
$V\spp  0 = \coprod_{i \ge 0} V_i\spp 0$.   We construct a new abstract branching diagram $\B$ as follows:
The vertex set of $\B$ is $V = \coprod_{k \ge 0} V_k$,  where
$$
V_k =  \coprod_{\substack{i\leq k\\k-i\text{ even}}} V_i\spp 0 \times \{k\}.
$$
Thus the $k$--th row of vertices of $\B$ consists of copies of rows $k$, $k-2$, $k-4$, \dots of vertices of
$\B_0$.  Now if $(\lambda, k) \in V_k$  and $(\mu, k+1) \in V_{k+1}$,  there exist $i \le k$ with $k -i$ even such that
$\lambda \in V_i \spp 0$,  and  $j \le k+1$ with $k+ 1 - j$ even such that $\mu \in V_j \spp 0$.
We declare $(\lambda, k) \nearrow (\mu, k+1)$ if, and only if, $|i - j| = 1$ and $\lambda$ and $\mu$ are adjacent
in $\B_0$.   The number of edges connecting  $(\lambda, k) $  and $(\mu, k+1)$ is the same as the number of edges connecting $\lambda$  and $\mu$  in $\B_0$.   

The first few levels of $\B$ is picture schematically  in Figure \ref{figure: branching diagram},  where each diagonal line represents
all the edges connecting vertices in $V_i^{(0)}$  with vertices in $V_{i \pm 1}^{(0)}$.  
\begin{figure}
$$ \inlinegraphic[scale=.5]{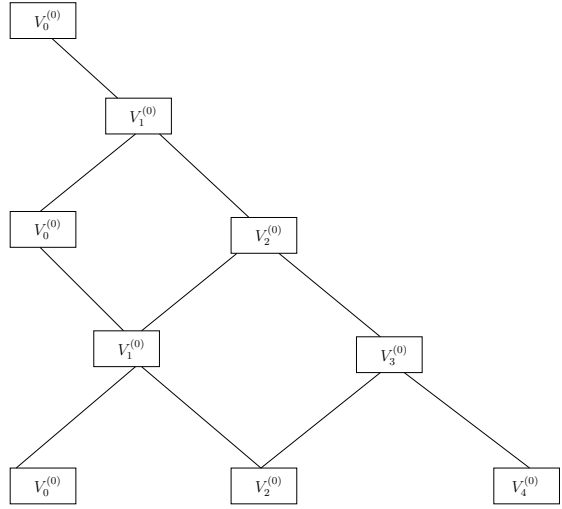}$$
\caption{Branching diagram obtained by reflections}
\label{figure: branching diagram}
\end{figure}
Note that  the $k$--th level of $\B$ is a folded copy of  the first $k$ levels of $\B_0$.
We call $\B$ the {\em  branching diagram obtained by reflections from $\B_0$}.

\begin{example}  Take $\B_0$ to be Young's lattice.  Thus $V_k\spp 0$ consists of Young diagrams of size
$k$, and $\lambda \nearrow \mu$ in $\B_0$ if $\mu$ is obtained from $\lambda$ by adding one box.
Then the  $k$--th row of vertices in the abstract branching diagram $\B$ obtained from $\B_0$ by reflections consists of all pairs $(\lambda, k)$,  where $\lambda$ is a  Young diagram of size $i \le k$, with $k - i$ even.
 Moreover,  $(\lambda, k) \nearrow (\mu, k+1)$ in $\B$ if, and only if, 
$\mu$ is obtained from $\lambda$ either by adding one box or by removing one box.
\end{example}

\subsection{The Jones basic construction} This paper could be written without ever mentioning the Jones basic construction.
  Nevertheless,  in our view, the   basic construction plays an essential role behind the scenes.

The Jones basic construction was introduced ~\cite{Jones-index} in the theory of von Neumann algebras and is crucial  in the analysis of von Neumann  subfactors.  Translated to the context of finite dimensional split semisimple algebras over a field, the basic construction was a fundamental ingredient  in  Wenzl's analysis of the generic structure of the Brauer algebras and the BMW algebras ~\cite{Wenzl-Brauer, Birman-Wenzl, Wenzl-BCD} .

The basic construction for finite dimensional split semisimple algebras can be described as follows (see ~\cite{GHJ}):  let $ A \subseteq B $ be finite dimensional split semisimple algebras over field $ F$, with the same multiplicative identity element.
The basic construction for the pair   $ A \subseteq B $ is the algebra $\End(B_A)$.
This algebra is also split semisimple and the inclusion matrix for the pair $B \subseteq \End(B_A)$ is a transpose of that for the pair $ A \subseteq B $.  Suppose now that $ B$  has a faithful $F$-- valued trace $ \varepsilon$ with faithful restriction to $ A$.   Here faithful means that the bilinear form $(x, y) \mapsto \varepsilon(x y)$ is non--degenerate.  In this case there is a unique trace preserving conditional expectation
$ \eps_A : B \to A$, i.e. a unital $A$--$A$--bimodule map satisfying $\eps\circ \eps_A = \eps$. 
Identify $ B$ with its image in $\End_F(B)$ under the left regular representation.
The basic construction $ \End(B_A)$  is equal to $ B \eps_A B =
\{ \sum_{i = 1}^n  b_i' \eps_A b_i'' :  n \ge 1,  b_i',  b_i'' \in B\}$.  Moreover,  $ B \eps_A B \cong
B \otimes_A B$,  where the latter is given the algebra structure determined by
$(b_1 \otimes b_2)(b_3 \otimes b_4) = b_1 \otimes \eps_A(b_2 b_3) b_4$.  Note that  we have three realizations for the basic construction,
$$
\End(B_A) \cong B \eps_A B \cong
B \otimes_A B,
$$
any of which could serve as a potential definition of the basic construction in a more general setting.

Suppose in addition that we are given an algebra $C$  with $B  \subseteq C$  and that $ C$ contains an idempotent  $ e$  such that  $ exe = \eps_{A}(x) e$ for $ x \in B$, and $x \mapsto x e$ is injective from $B$ to $Be \subseteq C$.     Note that
 $ BeB$  is a possibly non--unital  subalgebra of $C$.
By ~\cite{Wenzl-Brauer}, Theorem 1.3, $ BeB \cong B\eps_{A}B  \cong \End(B_{A})$, and,  in particular,  $ BeB$  is unital and semisimple.

Let's now describe how Wenzl used these ideas to show the generic semisimplicity  of the 
Brauer algebras.    We refer  the reader to Section \ref{subsection: Brauer algebras} for the definition of the Brauer algebras.
Consider the Brauer algebras $B_{n} = \br_n(F, \delta)$  over $F= \C$ or $F = \Q(\delta)$, in the first case with parameter
$\delta $ a non-integer complex number, and in the second case with parameter $\delta $
an indeterminant over $\Q$. The Brauer  algebras  have a canonical $F$--valued trace $\varepsilon$  and conditional expectations $\varepsilon_{n}: B_{n} \to B_{n-1}$  preserving the trace. 
 Each Brauer  algebra $B_{n}$ contains an essential idempotent $e_{n-1}$ with
 $e_{n-1}^{2} = \delta e_{n-1}$ and $e_{n-1} xe_{n-1} = \delta \eps_{n-1}(x) e_{n-1}$ for
 $x \in B_{n-1}$.  Moreover, $x \mapsto x e_{n-1}$ is injective from $B_{n-1}$ to $B_{n}$ and one has $B_{n}/ B_{n} e_{n-1} B_{n} \cong F \mathfrak S_{n}$,
 which is semisimple, since $F$ has characteristic 0.
  Let   $f_{n-1} =  \delta\inv e_{n-1}$;  then $f_{n-1}$  is an idempotent with$f_{n-1} xf_{n-1} =  \eps_{n-1}(x) f_{n-1}$ for
 $x \in B_{n-1}$.   We have $B_{0} \cong B_{1} \cong F$.   
 
 Suppose it is known for some $n$ that $B_{k}$ is split semisimple and that
 the trace $\eps$ is faithful on $B_{k}$  for $k \le n$.   By Wenzl's  observation applied to
 $B_{n-1}  \subseteq B_{n}  \subseteq B_{n+1}$  and the idempotent  $f_{n} \in  B_{n+1}$, we have
 $B_{n} e_{n} B_{n}  = B_{n} f_{n} B_{n} \cong B_{n} \eps_{n} B_{n} \cong \End((B_{n})_{B_{n-1}})$.  But it is elementary to check that   $B_{n} e_{n} B_{n}  = 
 B_{n+1} e_{n} B_{n+1}$.  Thus we have that the ideal  $B_{n+1} e_{n} B_{n+1} \subseteq 
 B_{n+1}$ is split semisimple, and the quotient of $B_{n+1}$ by this ideal ($\cong F \mathfrak S_{n+1}$) is also split semisimple, so $B_{n+1}$ is split  semisimple.  To continue the inductive argument, it is necessary to verify that the trace $\eps$ 
is faithful on $B_{n+1}$. Wenzl uses a Lie theory argument for this.

 In this paper, we develop a cellular analog of this argument. Let's continue to use the example of the Brauer algebras to illustrate this. Cellularity  is a property that is preserved under specializations, so it suffices to consider the Brauer algebras over the generic ring
 $R = \Z[ \delta ] $.  Let $F$ denote the field of fractions of $R$, $F = \Q(\delta)$.	
 Write $\br_{n}$ for $\br_{n}(R, \delta)$  and $\br_{n}^{F}$ for $\br_{n}(F, \delta)$.  By Wenzl's theorem, $\br_{n}^{F}$ is split semisimple. 
  We have $B_{0} \cong B_{1} \cong R$.   
 
 Suppose it is known for some $n$ that $B_{k}$ is cellular for $k \le n$.  We want to show that
$B_{n+1} e_{n} B_{n+1} = B_{n} e_{n} B_{n}$ is a cellular ideal in $B_{n+1}$.  It will then follow  that $B_{n+1}$  is cellular, because the quotient $B_{n+1}/B_{n+1} e_{n} B_{n+1} \cong R \mathfrak S_{n+1}$  is cellular.   Let $\Lambda_{n-1}$ denote the partially ordered set in the 
cell datum for $B_{n-1}$.  For each order ideal $\Ga$ of $\La_{n-1}$,   write $J(\Ga)$ for  the span in $B_{n-1}$ of all $c_{s, t}^\la$ with $\la \in \Ga$.   The crucial point is to show that
$\Ga \mapsto B_n e_n J(\Ga)  B_n = B_{n+1} e_n J(\Ga)B_{n+1}$  is a 
$\La_{n-1}$--cell net in $B_{n+1} e_n B_{n+1}$. Along the way to doing this, we  show that
 \begin{equation}\label{equation:  jones bc 1}
J'(\Ga) := B_{n} \otimes_{B_{n-1}} J(\Ga) \otimes_{B_{n-1}} B_{n} \cong B_{n}e_{n}J(\Ga)B_{n} 
 \end{equation} 
via $b' \otimes x \otimes b'' \mapsto b' e_{n} x b''$; consequently, if $\Ga_{1} \subseteq \Ga_{2}$,  then $J'(\Ga_{1})$ imbeds in $J'(\Ga_{2}) $.  In particular,
 \begin{equation} \label{equation:  jones bc 2}
 B_{n}  \otimes_{B_{n-1}} B_{n} \cong B_{n}e_{n}B_{n} =  B_{n+1}e_{n}B_{n+1}, 
 \end{equation}
and $J'(\Ga)$ imbeds as an ideal in the (non--unital) algebra $
 B_{n}  \otimes_{B_{n-1}} B_{n}$.  Essentially,  what we show is that 
 $B_{n+1}e_{n}B_{n+1} = B_{n}e_{n} B_{n}$ is isomorphic to the basic construction 
 $
 B_{n}  \otimes_{B_{n-1}} B_{n}$, and that
 $\Ga \mapsto J'(\Ga)$ is a $\La_{n-1}$--cell net in  $
 B_{n}  \otimes_{B_{n-1}} B_{n}$. 
 
  We note that   $B_{n}$ is {\em not}  a projective $B_{n-1}$--module, but the isomorphisms
 (\ref{equation:  jones bc 1})  and the embeddings  $J'(\Ga_{1})  \hookrightarrow J'(\Ga_{2}) $
 reflect the projectivity of $B_{n}^{F}$ over $B_{n-1}^{F}$.

\subsection{Coherent cellular towers and extension of the ground ring}

Let $R$ be an integral domain and let $F$ denote the field of fractions of $R$.  We will be interested in 
coherent towers $(H_n)_{n \ge 0}$  of cellular algebras over $R$ such that for all $n$, the $F$--algebra $H_n^F := H_n \otimes_R F$ is (split) semisimple.  We will see that in this situation we have uniqueness of multiplicities in the
filtrations of induced and restricted modules by cell modules, and Frobenius reciprocity connecting these multiplicities.

For any algebra $A$ over $R$, write $A^F$ for the $F$--algebra $A\otimes_R F$.  Moreover, for a left (or right)
$A$--module $M$,  write $M^F$ for the left (or right) $A^F$  module $M\otimes_R F$.

\begin{lemma} \label{first tensor iso}
Let $R$ be an integral domain and  $F$ its field of fractions.  Let   $A$ and $B$ be   $R$-algebras.   For modules $M_{A}$ and 
$_A N$, we have
\begin{equation} \label{Fisomorphism}
M\otimes_A N\otimes_R F\cong M^F\otimes_{A^F}N^F
\end{equation}
\noindent
as $F$-vector spaces.  The isomorphism
$$M\otimes_A N\otimes_RF\rightarrow M^F\otimes_{A^F}N^F$$
is determined by $(x\otimes_Ay\otimes_R f)\mapsto(x\otimes_R\1_F)\otimes_{A^F}(y\otimes_R f)$.  If ${}_{A}{N}_B$ is a bimodule, then the isomorphism in (\ref{Fisomorphism})  is an isomorphism of  right $B^F$--modules, and similarly, if 
$_B M_A$ is a bimodule, then the isomorphism is an isomorphism of left  $B^F$--modules. 
\end{lemma}

\begin{proof}
Note that
\begin{align}
M&\otimes_A(N\otimes_R F)\cong M\otimes_A A^F\otimes_{A^F}(N\otimes_R F) \notag \\
	&=(M\otimes_A A\otimes_R F)\otimes_{A^F}(N\otimes_R F) \notag \\
	&\cong (M\otimes_R F)\otimes_{A^F}(N\otimes_R F) \notag \\
	&=M^F\otimes_{A^F}N^F. \notag
\end{align}

\noindent
If we track a simple tensor through these equalities and isomorphisms, we see that
\begin{align}
&x\otimes_A y\otimes_R f\mapsto x\otimes_A\1_{A^F}\otimes_{A^F}(y\otimes_R f) \notag \\
	&\quad=x\otimes_A\1_A\otimes_R\1_F\otimes_{A^F}(y\otimes_R f)\mapsto(x\otimes_R\1_F)\otimes_{A^F}(y\otimes_R f). \notag
\end{align}
The final statement follows from this.
\end{proof}

\begin{lemma} \label{lemma injectivity of x to x tensor 1}  Let $R$ be an integral domain and $F$ its field of fractions.  If $M$ is a free $R$--module, then the map $M \to M \otimes_R F$ determined by 
$x \mapsto x \otimes 1_F$ is injective.
\end{lemma}

\begin{proof}  It follows from ~\cite{Jacobson}, Propositions 3.2 and 3.3 that the map $x \mapsto x \otimes 1$ takes an $R$--basis of $M$ to an $F$--basis of $M \otimes_R F$.  In particular, the map is injective.
\end{proof}

\begin{lemma}  \label{injectivity of iota tensor id(F) with free R modules}
Let $R$ be an integral domain and $F$ its field of fractions.  Let $N_1 \subseteq N_2$ be $R$--modules with $N_2$ free.  Let $\iota : N_1 \to N_2$ denote the injection.  Then
$\iota \otimes \id_F :  N_1 \otimes_R F \to N_2 \otimes_R F$ is injective.
\end{lemma}

\begin{proof}
 Any element of $N_1 \otimes_R F $ can be written as $y = (1/q) (x \otimes 1_F)$,
with $q \in R^\times$ and $x \in N_1$.  Then $\iota \otimes \id_F(y) =  (1/q)(\iota(x) \otimes 1_F) =
(1/q)\, \gamma \circ \iota (x)$,  where $\gamma : N_2 \to N_2 \otimes_R F$ is determined
by $z \mapsto z \otimes 1_F$.  Because $N_2$ is a free $R$--module, $\gamma$ is injective, by Lemma \ref{lemma injectivity of x to x tensor 1}, and it follows that $\iota \otimes \id_F$ is injective.
\end{proof}

\begin{lemma} \label{lemma: multiplicities in cell filtrations}
Let $R$ be an integral domain with field of fractions $F$.
 Suppose that $(H_n)_{n \ge 0}$ is a coherent tower of cellular algebras over  $R$ and that
 $H_n^F$ is split semisimple for all $n$.  Let $\Lambda_n$ denote the partially ordered set in the cell datum for
 $H_n$. 
  Then
 \begin{enumerate}
 \item 
 $\{(\Delta^\la)^F : \la \in \Lambda_n\}$ is a complete family of simple $H_n^F$--modules.
 \item Let $[\omega(\mu, \la)]_{\mu \in \Lambda_{n+1}, \,  \la \in \Lambda_n}$ denote the inclusion matrix for
 $H_n^F \subseteq H_{n+1}^F$.   Then for any $\la \in \Lambda_n$ and $\mu \in \Lambda_{n+1}$, 
  and any cell filtration of $\Res_{H_n}^{H_{n+1}}(\Delta^\mu)$,  the number of subquotients of the filtration isomorphic to  $\Delta^\la$ is $\omega(\mu, \la)$.
  
  \item  Likewise, for any $\la \in \Lambda_n$ and $\mu \in \Lambda_{n+1}$, 
 and any cell filtration of $\Ind_{H_n}^{H_{n+1}}(\Delta^\la)$,  the number of subquotients of the filtration isomorphic to  $\Delta^\mu$ is $\omega(\mu, \la)$.
 \end{enumerate}
\end{lemma}

\begin{proof}  For point (1),  $(\Delta^\la)^F$ is a cell module for $H_n^F$, and, for a semisimple cellular algebra, the cell modules are precisely the simple modules.

We have 
\begin{equation} \label{first direct sum decomp of Res Delta}
(\Res_{H_n}^{H_{n+1}}(\Delta^\mu))^F = \Res_{H_n^F}^{H_{n+1}^F}((\Delta^\mu)^F) \cong
 \bigoplus_{\la \in \Lambda_n}  \omega(\mu ,\la) (\Delta^\la)^F,
\end{equation}
by definition of the inclusion matrix.  On the other hand, if 
$$
\Res_{H_n}^{H_{n+1}} (\Delta^\mu) = N_s \supseteq N_{s-1} \supseteq \cdots \supseteq N_0 = (0)
$$
is a cell filtration, with $N_j/N_{j-1} \cong \Delta^{\la_j}$, then
$$
(\Res_{H_n}^{H_{n+1}}(\Delta^\mu))^F = N_s^F \supseteq N_{s-1}^F \supseteq \cdots \supseteq N_0^F = (0),
$$
by Lemma \ref{injectivity of iota tensor id(F) with free R modules}, because all the modules
$N_j$  are free as $R$--modules.  Moreover, 
$N_j^F/N_{j-1}^F \cong (N_j/N_{j-1})^F \cong (\Delta^{\la_j})^F$ by right exactness of tensor products.    Since $H_n^F$  modules are semisimple,
\begin{equation}\label{second direct sum decomp of Res Delta}
(\Res_{H_n}^{H_{n+1}}(\Delta^\mu))^F \cong  \bigoplus_{j = 1}^s (\Delta^{\la_j})^F.
\end{equation}
Comparing (\ref{first direct sum decomp of Res Delta})  and (\ref{second direct sum decomp of Res Delta}) and taking into account that $\Delta^\la \mapsto (\Delta^\la)^F$ is injective, we obtain conclusion (2).

Likewise,
$$
(\Ind_{H_n}^{H_{n+1}}(\Delta^\la))^F =  H_{n+1} \otimes_{H_n} \Delta^\la \otimes_R F \cong 
 H_{n+1}^F \otimes_{H_n^F} (\Delta^\la)^F,
$$
by Lemma \ref{first tensor iso}.  
But
$$
H_{n+1}^F \otimes_{H_n^F} (\Delta^\la)^F =  \Ind_{H_n^F}^{H_{n+1}^F}((\Delta^\la)^F)  \cong \bigoplus_{\mu \in \Lambda_{n+1}}  \omega(\mu ,\la) (\Delta^\mu)^F,
$$
using (\ref{first direct sum decomp of Res Delta}) and Frobenius reciprocity. The rest of the argument for point (3) is  similar to  that for point (2).
\end{proof}

\begin{lemma}  \label{lemma: cell basis indexed by paths}
Adopt the assumptions and notation of Lemma \ref{lemma: multiplicities in cell filtrations}.
Assume in addition that the branching diagram $\B$ for $(H_n^F)_{n\ge 0}$ has no multiple edges and that
$H_0^F = F$.   It follows that each $H_n$ has a cell datum (perhaps different from the one initially given)
with the same partially ordered set $\Lambda_n$ but with $\mathcal T(\la)$ equal to the set of paths
on $\B$ from $\emptyset$ to $\la$.  
\end{lemma}

\begin{proof}
Referring to the proof of Proposition \ref{lemma: cell net characterization of cellularity}, it suffices to show that, for each $n$ and for each $\la \in \La_n$, the cell module $\Delta^\la$  has
 an $R$--basis  indexed by the set  $\mathcal P(\la)$ of paths in $\B$ from $\emptyset$ to $\la$.   
But this says only that the rank of $\Delta^\la$ over $R$ is $|\mathcal P(\la)|$, and this is true because $\rank_R(\Delta^\la) = \dim_F(\Delta^\la \otimes_R F) = |\mathcal P(\la)|$.  See also the following remark.
\end{proof}

\begin{remark}  In principle,  in the situation of Lemma \ref{lemma: cell basis indexed by paths}, we can recursively build bases of cell modules, using the cell filtrations of restrictions.  Suppose we have bases of $\Delta^\la$ for all $\la \in \La_n$ for some $n$. 
Let $\mu \in \La_{n+1}$.  Then $\Delta^\mu$,  regarded as an $H_n$--module, has a filtration by cell modules of $H_n$,   
$$
\Delta^\mu = N_s \supseteq N_{s-1} \supseteq \cdots \supseteq N_0 = (0),
$$
 with $N_j/N_{j-1} \cong \Delta^{\la_j}$;  and $\la \in \La_n$  appears (exactly once) in the list 
 of $\la_j$,  if, and only if, $\la \nearrow \mu$.  Now we inductively build bases of the $N_j$ to obtain a basis of $N_s = \Delta^\mu$.  The isomorphism $N_1 \cong \Delta^{\la_1}$ provides a basis of $N_1$.  For $j \ge 2$, if we have a basis of $N_{j-1}$,  then  that basis together with any lift of a basis of  $N_j/N_{j-1} \cong \Delta^{\la_j}$ gives a basis of $N_j$.

\end{remark}

\section{A framework for cellularity}
In this section we describe our framework for cellularity of algebras related to the Jones basic construction.

\subsection{Framework Axioms} \label{subsection: framework axioms}
Let $R$ be an integral domain with field of fractions $F$.  We consider two sequences of $R$--algebras
$$
A_0 \subseteq A_1 \subseteq A_2 \subseteq \cdots, \quad\text{and} \quad Q_0 \subseteq Q_1 \subseteq Q_2 \subseteq \cdots,
$$
each with a common multiplicative identity element.  
We assume the following axioms:
\begin{enumerate}
\item \label{axiom Hn coherent}  $(Q_n)_{n \ge 0}$ is a coherent tower of cellular algebras.
\item \label{axiom: involution on An}  There is an algebra involution $i$  on $\cup_n A_n$ such that $i(A_n) = A_n$.
\item  \label{axiom: A0 and A1}  $A_0 = Q_0 = R$, and $A_1 = Q_1$  (as algebras with involution).
\item  \label{axiom: semisimplicity}
For all $n$,  $A_n^F : = A_n \otimes_R F$   is split semisimple.  
\item \label{axiom:  idempotent and Hn as quotient of An}
 For $n \ge 2$,  $A_n$ contains an  essential idempotent $e_{n-1}$ such that $i(e_{n-1}) = e_{n-1}$ and
$A_n/(A_n e_{n-1} A_n) \cong Q_n$,  as algebras with involution.

\item \label{axiom: en An en} For $n \ge 1$,   $e_{n}$ commutes with $A_{n-1}$ and $e_{n} A_{n} e_{n} \subseteq  A_{n-1} e_{n}$.
\item  \label{axiom:  An en}
For $n \ge 1$,  $A_{n+1} 	e_{n} = A_{n} e_{n}$,  and the map $x \mapsto x e_{n}$ is injective from
$A_{n}$ to $A_{n} e_{n}$.
\item \label{axiom: e(n-1) in An en An} For $n \ge 2$,   $e_{n-1} \in A_{n+1} e_n A_{n+1}$.
\end{enumerate}

\begin{remark} \mbox{}
\begin{enumerate}
\item
Let $\Lambda_n \spp 0$ denote the partially ordered set in the cell datum for $Q_n$.  It follows from axioms (\ref{axiom Hn coherent}) and (\ref{axiom: semisimplicity}) and Lemma \ref{lemma: multiplicities in cell filtrations} that $\Lambda_n \spp 0$ can be identified with the $n$--th row of vertices of the branching diagram for $(Q_n^F)_{n \ge 0}$.

\item Applying the involution in axiom (\ref{axiom:  An en}),  we also have $e_{n} A_{n+1} 	 = e_{n} A_n $, and the map $x \mapsto e_{n} x $ is injective from
$A_{n}$ to $e_{n} A_{n}$. 
\item  Since $e_n$ is an essential idempotent, there is a non--zero $\delta_n \in R$ with $e_n^2 = \delta_n e_n$.   Thus we have $e_n A_n e_n \supseteq e_n A_{n-1} e_n = A_{n-1} e_n^2 =
\delta_n A_{n-1}  e_n$.   Combining this with axiom  (\ref{axiom: en An en}), we have
$\delta_n A_{n-1}  e_n \subseteq e_n A_n e_n \subseteq A_{n-1} e_n$.  Hence
$e_n A_n^F e_n =  A_{n-1}^F e_n$.

\item From axiom (\ref{axiom: en An en}), we have for every
$x \in A_{n}$, there is a $y \in A_{n-1}$ such that $e_{n} x e_{n} = y e_{n}$; but by axiom
(\ref{axiom:  An en}),  $y$ is uniquely determined, so we have a map $\cl_{n} : A_{n} \rightarrow A_{n-1}$ 
with  $e_{n} x e_{n} = \cl_{n}(x) e_{n}$.    It is easy to check that $\cl_{n}$ is an $A_{n-1}$--$A_{n-1}$--bimodule map, but it is not unital in general;  if $e_{n-1}^2 = \delta_n e_{n-1}$,  then  
$\cl_{n}(\1) = \delta_n \1$.  If $\delta_n$ is invertible in $R$,  then $\vareps_{n} = (1/\delta_n) \cl_{n}$ is a conditional expectation, i.e., a unital $A_{n-1}$--$A_{n-1}$--bimodule map.
\item From axioms (\ref{axiom: semisimplicity}) and (\ref{axiom:  idempotent and Hn as quotient of An}),  we have $Q_n^F : = Q_n \otimes_R F$ is split semisimple. 
\item  In our examples, there is a single non--zero $\delta$ with $e_n^2 = \delta e_n$  for all $n$.  
\end{enumerate}
\end{remark}

\subsection{The main theorem}

\begin{theorem} \label{main theorem}
 Let $R$ be an integral domain with field of fractions $F$.  Let $(Q_k)_{k\ge 0}$ and
$(A_k)_{k\ge 0}$  be two towers of $R$--algebras satisfying the framework axioms of Section \ref{subsection: framework axioms}.  Then
\begin{enumerate}
\item $(A_k)_{k\ge 0}$ is a coherent tower of cellular algebras.
\item  For all $k$, the  partially ordered set in the cell datum for $A_k$ can be realized as
$$
\Lambda_k =  \coprod_{\substack{i\leq k\\k-i\text{ even}}} \Lambda_i\spp 0 \times \{k\},
$$
with the following partial order:  Let $\lambda \in \Lambda_i\spp 0$ and $\mu \in \Lambda_j\spp 0$,  with
$i$, $j$, and $k$ all of the same parity.  Then
 $(\lambda, k) > (\mu, k)$ if, and only if,   $i < j$,  or $i = j$ and $\lambda > \mu$ in $\Lambda_i\spp 0$.
 \item  Suppose $k \ge 2$ and $(\la, k)  \in \La_i\spp 0 \times \{k\} \subseteq \La_k$.  Let 
 $\Delta^{(\la, k)}$ be the corresponding cell module.  
  If  $i < k$,  
 then
 $(A_k e_{k-1} A_k \ \Delta^{(\la, k)})\otimes_R F = \Delta^{(\la, k)}\otimes_R F $, while if $i = k$  then    \break $A_k e_{k-1} A_k  \ \Delta^{(\la, k)}  = 0$.
\item The branching diagram $\B$ for $(A_k^F)_{k \ge 0}$ is that obtained by reflections from the branching diagram
$\B_0$ for $(Q_k^F)_{n \ge 0}$.

\end{enumerate}
\end{theorem}

 \begin{remark} \label{remark:  point 5 comes from other statements}
 In most of our examples,  the branching diagrams have no multiple edges.  In this case,
 for all $k$ and for all $(\lambda, k) \in \Lambda_k$,   the index set $\mathcal T((\lambda, k))$ in the cell datum for $A_k$ can be taken to be the set of paths on $\B$  from $\emptyset$ to $(\lambda, k)$.
This follows from (1) and (4),  using Lemma \ref{lemma: cell basis indexed by paths}.
 \end{remark}

\section{Proof of the main theorem} \label{section:  basic construction preserves cellularity}

 We will prove Theorem \ref{main theorem} in this section.  Our strategy is to prove 
 the following statement by induction on $n$:

\smallskip
\noindent {\bf Claim:}\quad
{\em  For all $n \ge 0$,  the statements  (1) --(4)  of Theorem \ref{main theorem}  hold for the finite tower $(A_k)_{0 \le k \le n}$.}
\smallskip

Of course, by statement (4) for the finite tower, we mean that the branching diagram for the
finite tower  $(A_k^F)_{0 \le k \le n}$ is that obtained by reflections from the branching diagram
of the finite tower $(Q_k^F)_{0 \le k \le n}$.  

The claim holds trivially for $n = 0$ and $n = 1$.  We assume that the claim holds 
for some $n \ge 1$ and prove that it also holds for $n + 1$.  

\subsection{$A_{n+1}$ is cellular}
We will show that $A_{n+1}$ is a cellular algebra.

Since $A_{n+1}/A_{n+1} e_n A_{n+1} \cong Q_{n+1}$ is cellular,  to prove that
 $A_{n+1}$ is cellular,  it suffices to show that  $A_{n+1} e_n A_{n+1}$ is a cellular ideal in $A_{n+1}$; see Remark
\ref{remark on extensions of cellular algebras}.

Recall that $\Lambda_k$ denotes the partially ordered set in the cell datum for $A_k$ for each $k$, $0 \le k \le n$.
Denote the elements of the cellular basis of $A_k$ by $c_{u, v}^\la$  for $\la \in \La_k$ and
$u, v \in \mathcal T(\la)$.

For each order ideal $\Ga$ of $\La_{n-1}$,  recall that  $A_{n-1}(\Ga)$ is the span in $A_{n-1}$ of all $c_{s, t}^\la$ with $\la \in \Ga$.
$A_{n-1}(\Ga)$ is an $i$--invariant two sided ideal of $A_{n-1}$.  
\ignore{
If $\Ga \subseteq \Ga'$ are two order ideals with
$\Ga' \setminus \Ga = \{\la\}$, then
$$A_{n-1}(\Ga')/A_{n-1}(\Ga)\cong
A_{n-1}^{\la}/\breve A_{n-1}^{\la}     \cong
\Delta^{\la}\otimes_Ri(\Delta^{\la}),$$
and the isomorphism $\alpha : A_{n-1}(\Ga')/A_{n-1}(\Ga) \to \Delta^{\la}\otimes_Ri(\Delta^{\la})$ satisfies
$\alpha \circ i =  i \circ \alpha$.
}
In the following, we will write $J(\Ga) = A_{n-1}(\Ga)$
and     $$\hat J(\Ga)=  A_n e_n J(\Ga)  A_n = A_{n+1} e_n J(\Ga)A_{n+1}, $$  which is  a two sided ideal in $A_{n+1}$.  
Our goal is to show that $\Ga \mapsto \hat J(\Ga)$ is a $\La_{n-1}$--cell net in $A_{n+1} e_n A_{n+1}$.

\begin{lemma} \label{second tensor iso}
Let $R$ be an integral domain and $F$ its field of fractions.  
Suppose that $A$ and $B$ are   $R$-algebras.  Let
$P_{A}$,   $_A M_A$  and $_A Q$
 be modules.  Then
$$P\otimes_AM\otimes_AQ\otimes_R F\cong P^F\otimes_{A^F}M^F\otimes_{A^F}Q^F$$
as $F$-vector spaces.  The isomorphism 
$$P\otimes_A M\otimes_A Q\otimes_R F\rightarrow P^F\otimes_{A^F} M^F\otimes_{A^F}Q^F$$
is determined by 
$$x\otimes_Ay\otimes_Az\otimes_R f \mapsto (x\otimes_R\1_F)\otimes_{A^F}(y\otimes_R\1_F)\otimes_{A^F}(z\otimes_R f).$$
If $_B P_A$ and $_A Q_B$ are bimodules,  then the isomorphism is an isomorphism of $B^F$--$B^F$--bimodules.
\end{lemma}

\begin{proof}
By Lemma \ref{first tensor iso},
\begin{equation} \label{equation: 2nd tensor iso 1}
(P\otimes_AM)\otimes_AQ\otimes_R F\cong(P\otimes_AM)^F\otimes_{A^F}Q^F.\end{equation}
Applying Lemma \ref{first tensor iso} again, we have that 
\begin{equation} \label{equation: 2nd tensor iso 2}
(P\otimes_AM)^F\cong P^F\otimes_{A^F}M^F
\end{equation}
as right $A^F$--modules.  Combining the two isomorphisms we have
\begin{equation} \label{equation: 2nd tensor iso 3}
P\otimes_AM\otimes_AQ\otimes_R F\cong P^F\otimes_{A^F}M^F\otimes_{A^F}Q^F.
\end{equation}
If we track a simple tensor  through these isomorphisms, we see that
\begin{align}
x\otimes_A&y\otimes_Az\otimes_R f\mapsto(x\otimes_Ay\otimes_R\1_F)\otimes_{A^F}(z\otimes_R f) \notag \\
	&\mapsto(x\otimes_R\1_F)\otimes_{A^F}(y\otimes_R\1_F)\otimes_{A^F}(z\otimes_R f).  \notag
\end{align}
If $_B P_A$ and $_A Q_B$ are bimodules,  then the isomorphism in (\ref{equation: 2nd tensor iso 1}) is an isomorphism of $B^F$--$B^F$--bimodules, and the isomorphism in  (\ref{equation: 2nd tensor iso 2}) is an isomorphism of 
$B^F$--$A^F$--  bimodules.  Hence the final isomorphism (\ref{equation: 2nd tensor iso 3}) is an isomorphism of
$B^F$--$B^F$--bimodules.
\end{proof}

\begin{lemma} \label{beta inj}
Let $K$ be a field and $A$ a  semisimple $K$-algebra.  Suppose that $I \subseteq A$ is a two-sided ideal and $M_A$, $_A N$ are modules.  Then the homomorphism $M\otimes_{A} I\otimes_{A} N\rightarrow M\otimes_{A} N$ defined by $x\otimes y\otimes z\mapsto x\otimes yz$ is injective.
\end{lemma}

\begin{proof}
The semisimplicity of $A$ implies that all $A$-modules are projective.  Thus $N\otimes_{A}-$ and $-\otimes_{A} M$ are exact, and 
$$N\otimes_AI\otimes_AM\rightarrow N\otimes_AA\otimes_AM\cong N\otimes_AM$$
is injective.
\end{proof}

\ignore{
\begin{remark} \label{special case of beta inj}
If $J$ is a two-sided ideal such that $I\trianglelefteq J\trianglelefteq A$, then the proof of Lemma \ref{beta inj} implies that the map $N\otimes_AI\otimes_AM\rightarrow N\otimes_AJ\otimes_AM$ is injective.
\end{remark}
}

\vbox{
\begin{proposition} \label{proposition: Phi isomorphism}  For all order ideals $\Ga$ of $\La_{n-1}$:
\begin{enumerate}
\item   
 The  map
$$\Phi_\Ga : \ A_n  e_n \otimes_{A_{n-1}}J(\Ga)\otimes_{A_{n-1}} e_n A_n\rightarrow A_n e_n J(\Ga) A_n$$
determined by $$\Phi_\Ga(a_1 e_n \otimes x\otimes  e_n a_2)=a_1 e_n xa_2$$ is an isomorphism of 
$A_{n+1}$--$A_{n+1}$--bimodules.
\item  $ A_n  e_n \otimes_{A_{n-1}}J(\Ga) \otimes_{A_{n-1}} e_n A_n$ is a free $R$--module.
\item Let $ \Ga'$ be another order  ideal containing $\Ga$,   such that $\Ga' \setminus \Ga$ is a singleton.
 Let $\iota$ denote the injection $J(\Ga) \ \to J(\Ga')$.  Then
$$
\begin{aligned}
\beta_{\Ga,\Ga'} := \id \otimes \iota \otimes \id : \ &A_n  e_n \otimes_{A_{n-1}}J(\Ga) \otimes_{A_{n-1}} e_n A_n \to \\
&A_n  e_n \otimes_{A_{n-1}}J(\Ga') \otimes_{A_{n-1}} e_n A_n
\end{aligned}
$$
is injective.
\end{enumerate}
\end{proposition}
}

We provide two lemmas on the way to proving Proposition \ref{proposition: Phi isomorphism}.

\begin{lemma} \label{lemma 1 for induction on j}
 Let \ $\Ga \subseteq \Ga'$ be two order ideals in $\La_{n-1}$ such that $\Ga' \setminus \Ga$ is a singleton.
Suppose that \  $\Phi_{\Ga}$ is an isomorphism and that   $ A_n  e_n \otimes_{A_{n-1}}J(\Ga) \otimes_{A_{n-1}} e_n A_n$ is a free $R$--module.  Then  $\beta_{\Ga, \Ga'}$ is injective and 
$ A_n  e_n \otimes_{A_{n-1}}J(\Ga') \otimes_{A_{n-1}} e_n A_n$ is a free $R$--module.
\end{lemma}

\begin{proof} Let $\{\la\} = \Ga' \setminus \Ga$.
Since $\Phi_{\Ga}$ is assumed injective,  it follows from considering the commutative diagram below that
$\beta_{\Ga,\Ga'}$ is also injective:
\begin{diagram}
A_n e_n \otimes_{A_{n-1}}J(\Ga)\otimes_{A_{n-1}} e_n A_n &&\rTo^{\Phi_\Ga}   &&A_n e_n J(\Ga) A_n  \\
\dTo_{ \beta_{\Ga, \Ga'} } & &&& \dTo\\
 A_n e_n \otimes_{A_{n-1}} J(\Ga')\otimes_{A_{n-1}} e_n A_n &&\rTo^{\Phi_{\Ga'}} &&A_n e_n J(\Ga')  A_n
\end{diagram}
\ignore{
$$\begin{CD}
A_n e_n \otimes_{A_{n-1}}J(\Ga)\otimes_{A_{n-1}} e_n A_n @>\Phi_\Ga>>  A_n e_n J(\Ga) A_n  \\
@VV  \beta_{\Ga, \Ga'} V                                @VV V  \\
A_n e_n \otimes_{A_{n-1}} J(\Ga')\otimes_{A_n} e_n A_n   @> \Phi_{\Ga'}  >> A_n e_n J(\Ga')  A_n
\end{CD}
$$
}

By the right exactness of tensor products, we have
\begin{equation} \label{equation 1 for lemma 1 of induction on j}
\begin{aligned}
&(A_n e_n\otimes_{A_{n-1}} J(\Ga')\otimes_{A_{n-1}} e_{n }A_n) / \beta_{\Ga,\Ga'}(A_n e_{n}\otimes_{A_{n-1}}J(\Ga)\otimes_{A_{n-1}} e_n A_n ) \\
&\cong A_n e_n\otimes_{A_{n-1}} (J(\Ga')/J(\Ga))    \otimes_{A_{n-1}} e_n A_n \\
&\cong   A_n e_n\otimes_{A_{n-1}}  \Delta^{\la} \otimes_R i( \Delta^{\la})  \otimes_{A_{n-1}} e_n A_n
\end{aligned}
\end{equation}
Consider $A_n e_n  =  A_{n+1} e_n$  (because of framework axiom (\ref{axiom:  An en}))   as an $A_{n+1}$--$A_{n-1}$--bimodule.  One can easily check that
$i(A_n e_n) \cong  e_n A_n $ as $A_{n- 1}$--$A_{n+1}$--bimodules.  Therefore,
\begin{equation}  \label{equation 2 for lemma 1 of induction on j}
i( \Delta^{\la})  \otimes_{A_{n-1}} e_n A_n \cong i( \Delta^{\la})  \otimes_{A_{n-1}} i( A_n e_n)
\cong i(A_n e_n\otimes_{A_{n-1}}  \Delta^{\la}),
\end{equation}
using Lemma \ref{lemma; involutions and tensor products of bimodules}.   By framework axioms 
(\ref {axiom: en An en}) and (\ref{axiom:  An en}),  $A_n e_n \cong A_n$ as $A_n$--$A_{n-1}$--bimodules.  Hence,
\begin{equation}  \label{equation 3 for lemma 1 of induction on j}
A_n e_n\otimes_{A_{n-1}}  \Delta^{\la} \cong A_n \otimes_{A_{n-1}}  \Delta^{\la} = \Ind_{A_{n-1}}^{A_n}( \Delta^{\la}),
\end{equation}
as $A_n$ modules.   Combining (\ref{equation 1 for lemma 1 of induction on j}), 
(\ref{equation 2 for lemma 1 of induction on j}), and (\ref{equation 3 for lemma 1 of induction on j}), we have
\begin{equation} \label{equation 4 for lemma 1 of induction on j}
\begin{aligned}
&(A_n e_n\otimes_{A_{n-1}} J(\Ga')\otimes_{A_{n-1}} e_{n}A_n) / \beta_{\Ga,\Ga'}(A_n e_{n}\otimes_{A_{n-1}}J(\Ga)\otimes_{A_{n-1}} e_n A_n ) \\
&\cong  \Ind_{A_{n-1}}^{A_n}( \Delta^{\la}) \otimes_R i( \Ind_{A_{n-1}}^{A_n}( \Delta^{\la})),
\end{aligned}
\end{equation}
as $A_n$--$A_n$--bimodules.

By the induction assumption on $n$,    $\Ind_{A_{n-1}}^{A_n}( \Delta^{\la})$ has a filtration
with subquotients isomorphic to cell modules for $A_n$, and in particular $\Ind_{A_{n-1}}^{A_n}( \Delta^{\la})$ is a free $R$--module.  By (\ref{equation 4 for lemma 1 of induction on j}),  
$$
(A_n e_n\otimes_{A_{n-1}} J(\Ga')\otimes_{A_n} e_{n}A_n) / \beta_{\Ga,\Ga'}(A_n e_{n}\otimes_{A_{n-1}}J(\Ga)\otimes_{A_{n-1}} e_n A_n )
$$
is a free $R$--module. Since  $A_n e_{n }\otimes_{A_{n-1}}J(\Ga)\otimes_{A_{n-1}} e_n A_n$ is free by hypothesis, and $\beta_{\Ga,\Ga'}$ is injective, 
 $$\beta_{\Ga,\Ga'}(A_n e_{n}\otimes_{A_{n-1}}J(\Ga)\otimes_{A_{n-1}} e_n A_n )$$ is a free $R$--module.  Hence 
$$
A_n e_n\otimes_{A_{n-1}} J(\Ga')\otimes_{A_{n-1}} e_{n}A_n
$$
is also a free $R$--module.
\end{proof}

\begin{lemma} \label{lemma 2 for induction on j} Let $\Ga$ be an order ideal in $\La_{n-1}$.
  If $A_n e_n\otimes_{A_{n-1}} J(\Ga) \otimes_{A_{n-1}} e_{n}A_n$ is a free $R$--module, then $\Phi_\Ga$ is an isomorphism.
\end{lemma}

\begin{proof}   $\Phi_\Ga$ is surjective, so we only have to prove  $\Phi_\Ga$ is injective.
Define
$$\a_1~:~ A_{n} e_n \otimes_{A_{n-1}} J(\Ga)\otimes_{A_{n-1}} e_n A_{n}\rightarrow A_{n}e_n\otimes_{A_{n-1}} J(\Ga)\otimes_{A_{n-1}} e_n A_{n}\otimes_R F$$
and
$$\a_2~:~ A_{n} e_n J(\Ga) A_{n}\rightarrow  A_{n} e_n J(\Ga) A_{n}\otimes_R F$$
by $x \mapsto x \otimes \1_F$.
 Since $A_{n}e_n \otimes_{A_{n-1}} J(\Ga)\otimes_{A_{n-1}} e_n A_{n}$ is a free $R$--module,  by assumption, 
  $\a_1$ is injective, according to Lemma \ref{lemma injectivity of x to x tensor 1}.   Let 
  $$
  \tau:  A_{n} e_n \otimes_{A_{n-1}} J(\Ga)\otimes_{A_{n-1}} e_n A_{n}\otimes_R F \rightarrow
  A_n^F e_n \otimes_{ A_{n-1}^F}J(\Ga)^F \otimes_{ A_{n-1}^F} e_n A_n^F
  $$
  be the isomorphism from Lemma \ref{second tensor iso}.  (We are writing $e_n$ for $e_n \otimes \1_F$.)
  Let
  $$
  \Phi_\Ga^F: A_n^F e_n \otimes_{ A_{n-1}^F}J(\Ga)^F \otimes_{ A_{n-1}^F} e_n A_n^F	\rightarrow	 A_n^F e_nJ(\Ga)^F  A_n^F
  $$
  be defined by $xe_n \otimes a \otimes e_n y \mapsto x e_n a y$.  
  
  Consider the following diagram
  \newarrow{Equals} =====
  \begin{diagram}
   A_n^F e_n \otimes_{ A_{n-1}^F}J(\Ga)^F \otimes_{ A_{n-1}^F} e_n A_n^F && \rTo^{\Phi_\Ga^F}  &&  A_n^F e_nJ(\Ga)^F  A_n^F \\
   \uTo^{\tau}&&&& \uEquals\\
   A_{n}e_n \otimes_{A_{n-1}} J(\Ga)\otimes_{A_{n-1}} e_n A_{n}\otimes_R F &&\rTo^{\Phi_\Ga\otimes id_F}&&
   A_{n} e_n J(\Ga) A_{n}\otimes_R F\\
   \uTo^{\alpha_{1}}&&&&\uTo_{\alpha_{2}}\\
   A_{n}e_n \otimes_{A_{n-1}} J(\Ga)\otimes_{A_{n-1}} e_n A_{n}&&\rTo^{\Phi_\Ga}&&A_{n} e_n J(\Ga) A_{n}.\\
  \end{diagram}
  \ignore{
$$\begin{CD}
 A_n^F e_n \otimes_{ A_{n-1}^F}J(\Ga)^F \otimes_{ A_{n-1}^F} e_n A_n^F	@>\Phi_\Ga^F>>	 A_n^F e_nJ(\Ga)^F  A_n^F \\
@A\tau AA	@|\\
 A_{n}e_n \otimes_{A_{n-1}} J(\Ga)\otimes_{A_{n-1}} e_n A_{n}\otimes_R F @>\Phi_\Ga\otimes id_F>>  A_{n} e_n J(\Ga) A_{n}\otimes_R F\\
@A\a_1AA	@AA\a_2A\\
 A_{n}e_n \otimes_{A_{n-1}} J(\Ga)\otimes_{A_{n-1}} e_n A_{n} @>\Phi_\Ga>>  A_{n} e_n J(\Ga) A_{n}.\\
\end{CD}$$
}
It is straightforward to check that $\Phi_\Ga^F \circ \tau\circ \a_1 = \a_2 \circ \Phi_\Ga$.    Thus, to prove that $\Phi_\Ga$ is injective, it suffices to show that $\Phi_\Ga^F$ is injective.

 Define 
 $$\beta~:~ A_n^F e_n \otimes_{A_{n-1}^F}J(\Ga)^F\otimes_{A_{n-1}^F}  e_n A_n^F\rightarrow A_n^F e_n\otimes_{A_{n-1}^F} e_n A_n^F$$
   by $\beta(x\otimes y\otimes z)=x\otimes yz$.  Observe that $\beta$ is  injective by
 Lemma  \ref{beta inj}.   Define  $$\phi^F : A_n^F e_n\otimes_{A_{n-1}^F} e_n A_n^F \to A_n^F e_n A_n^F$$  by
 $\phi^F( x e_n  \otimes e_n y) = x e_n y$.  Observe that $\phi^F \circ \beta = \Phi_\Ga^F$,  so to prove that $\Phi_\Ga^F$ is injective, it suffices to show that $\phi^F$ is injective.

Since $A_{n+1}^F$ is split semisimple (by framework axiom (\ref{axiom: semisimplicity})), the ideal $A_{n+1}^F e_n A_{n+1}^F$
(which equals $ A_n^F e_n A_n^F$ by framework axiom (\ref{axiom:  An en}))   is a unital algebra in its own right, and  Morita equivalent to $e_n A_{n+1}^F e_n = e_n A_n^F e_n \cong A_{n-1}^F$.
In fact,  let 
$$\psi^F :  e_n A_n \otimes_{A_n^F e_n A_n^F} A_n^F e_n \to e_n A_n^F e_n$$ be given by 
$e_n x \otimes y e_n \mapsto (1/\delta_n) e_n xy e_n$, where $e_n^2 = \delta_n e_n$.  Then
$$(e_n A_{n}^F e_n,  A_n^F e_n A_n^F,  A_n^F e_n,  e_n A_n^F,  \psi^F, \phi^F)$$
  is a Morita context, in the sense of ~\cite{Jacobson},  Section 3.12, with surjective
bimodule maps $\psi^F$ and $\phi^F$.  It follows from Morita theory, for example ~\cite{Jacobson},   Morita Theorem I, page
167,  that $\psi^F$ and $\phi^F$ are isomorphisms.
\end{proof}

\noindent
{\em Proof of Proposition \ref{proposition: Phi isomorphism}}:  \quad 
Let $\Ga$ be an order ideal of $\La_{n-1}$.  There exists a chain of order ideals
$$
\emptyset = \Ga_0 \subseteq \Ga_1 \subseteq \cdots \subseteq \Ga_s = \Ga,
$$
such that the difference between any two successive order ideals is a singleton.  Write $\beta_j$ for
$\beta_{\Ga_j, \Ga_{j+1}}$, for $0 \le j < s$.

We prove by induction  that for $0 \le j \le s$, 
$\Phi_{\Ga_j}$ is an isomorphism  and $A_n  e_n \otimes_{A_{n-1}}J(\Ga_j) \otimes_{A_{n-1}} e_n A_n$ is a free $R$--module;   and that for $0 \le j < s$,  $\beta_j$ is injective.
For $j = 0$, these statements are trivial since $J(\emptyset) = 0$.

Fix $j$ ($0 \le j < s$)  and suppose that 
$A_n  e_n \otimes_{A_{n-1}}J(\Ga_j)\otimes_{A_{n-1}} e_n A_n$ is a free $R$--module, that
$\Phi_{\Ga_j}$ is an isomorphism.  Then it follows from Lemma \ref{lemma 1 for induction on j}  that 
$ A_n  e_n \otimes_{A_{n-1}}J(\Ga_{j+1})\otimes_{A_{n-1}} e_n A_n$ is a free $R$--module.  Next, it follows from Lemma \ref{lemma 2 for induction on j} that $\Phi_{\Ga_{j+1}}$ is an isomorphism. 

We conclude that $A_n  e_n \otimes_{A_{n-1}}J(\Ga)\otimes_{A_{n-1}} e_n A_n$ is a free $R$--module and that
 $\Phi_\Ga$ is an isomorphism.   Applying Lemma 
\ref{lemma 1 for induction on j} again gives statement (3) of the Proposition.
  \qed

\medskip

We continue to work with the following assumptions:  
 $R$ is an integral domain with field of fractions $F$.   $(Q_k)_{k\ge 0}$ and
$(A_k)_{k\ge 0}$  are two towers of $R$--algebras satisfying the framework axioms of Section \ref{subsection: framework axioms}.  
The following induction assumption is in force:
For some fixed $n \ge 1$,  the conclusions (1) --(4)  of Theorem \ref{main theorem}  hold for the finite tower $(A_k)_{0 \le k \le n}$.
We use the notation of the discussion preceding Lemma \ref{second tensor iso}.

The following is a corollary of  Proposition \ref{proposition: Phi isomorphism}.

\begin{corollary}  \label{corollary:   a basic construction isomorphism}
$A_n e_n \otimes_{A_{n-1}} e_n A_n \cong A_n e_n A_n$,  as $A_{n+1}$--$A_{n+1}$ bimodules, with the isomorphism determined by $x e_n \otimes e_n y \mapsto x e_n y$.  
\end{corollary}

\begin{proof} In Proposition \ref{proposition: Phi isomorphism}, take $\Gamma = \Lambda_{n-1}$,  so 
$J(\Gamma) = A_{n-1}$.  
\end{proof}

\begin{proposition}  \mbox{}  \label{lemma:  cellularity induction step}
\begin{enumerate}
\item  $\Ga \mapsto \hat J(\Ga)$ is a $\La_{n-1}$--cell net in $A_n e_n A_n$.
\item
 $A_n e_n A_n$ is a cellular ideal in $A_{n+1}$.
\item $A_{n+1}$ is a cellular algebra.  The partially ordered set in the cell datum for $A_{n+1}$ can be realized as
$\La_{n+1} = \La_{n-1} \cup \La_{n+1}\spp 0$,  where $ \La_{n+1}\spp 0$ is the partially ordered set in the cell datum for
$Q_{n+1}$;  moreover the partial order on $\La_{n+1}$ agrees with the original partial orders on $\La_{n-1} $ and 
$\La_{n+1}\spp 0$, and satisfies $\la > \mu$ if $\la \in \La_{n-1}$ and $\mu \in \La_{n+1}\spp 0$.
\item  Let $\la \in \La_{n-1}$, and let $ \Delta^\la$ denote the corresponding cell module of $A_{n-1}$.   The cell module of $A_{n+1}$  corresponding to $\la$ is isomorphic to  
$A_n e_n \otimes_{A_{n-1}} \Delta^\la$.
\end{enumerate}
\end{proposition}

\begin{proof}  It is evident that  $\hat J(\emptyset) = \{0\}$, and that $\Ga_1 \subseteq \Ga_2$ implies $\hat J(\Ga_1) \subseteq \hat J(\Ga_2)$.
Note that  $J(\Ga_{\ge \la}) = A_{n-1}^\la$, so $\hat J(\Ga_{\ge \la})  = A_n e_n A_{n-1}^\la A_n$.
Similarly, $\hat J(\Ga_{> \la})  = A_n e_n \breve A_{n-1}^\la A_n$.  It follows that
$A_n e_n A_n = \spn\{\hat J(\Ga_{\ge \la}) : \la \in \La_{n-1}\}$ and that
for all $\la \in \La_{n-1}$,  $\hat J(\Ga_{> \la}) = \spn\{ \hat J(\Ga_{\ge \mu}) : \mu > \la\}$.  We have shown that
$\Ga \mapsto \hat J(\Ga)$  satisfies conditions (1) and (2) of Definition \ref{definition: cell net}.
 
 Next we show that $\Ga \mapsto \hat J(\Ga)$  satisfies condition (3) of Definition \ref{definition: cell net}.
Let  $\Ga \subseteq \Ga'$ be two order ideals  of $\La_{n-1}$, with $\Ga' \setminus \Ga = \{\la\}$.
From the proof of Proposition \ref{proposition: Phi isomorphism}, we already have $\hat J(\Ga')/\hat J(\Ga) \cong M^\la \otimes_R i(M^\la)$, with
$M^\la =  A_n e_n\otimes_{A_{n-1}}  \Delta^{\la} $. Let  $\chi : \hat J(\Ga')/\hat J(\Ga)  \to  M^\la \otimes_R i(M^\la)$ denote the isomorphism. We  have to check   that $\chi \circ i = i \circ \chi$.  The isomorphism
$\Phi_\Ga$ of Proposition \ref{proposition: Phi isomorphism} satisfies $i \circ \Phi_\Ga = \Phi_\Ga\circ i$.  Moreover,
$$\beta_{\Ga, \Ga'}( A_n e_n \otimes_{A_{n-1}} J(\Ga) \otimes_{A_{n-1}} e_n A_n) \subseteq A_n e_n \otimes_{A_{n-1}} J(\Ga') \otimes_{A_{n-1}} e_n A_n$$
and $\hat J(\Ga) \subseteq \hat J(\Ga')$  are $i$--invariant,  so the induced isomorphism
$$
\begin{aligned}
\tilde \Phi_\Ga :  A_n e_n \otimes_{A_{n-1}} J(\Ga')& \otimes_{A_{n-1}} e_n A_n/ \beta_{\Ga, \Ga'}( A_n e_n \otimes_{A_{n-1}} J(\Ga) \otimes_{A_{n-1}} e_n A_n) \\ &\to \hat J(\Ga')/\hat J(\Ga)
\end{aligned}
$$
satisfies $i \circ \tilde\Phi_\Ga = \tilde\Phi_\Ga\circ i$.  Next,  the map $$\pi : A_n e_n \otimes_{A_{n-1}} J(\Ga') \otimes_{A_{n-1}} e_n A_n \to
A_n e_n \otimes_{A_{n-1}} J(\Ga')/J(\Ga) \otimes_{A_{n-1}} e_n A_n$$
satisfies $i \circ \pi = \pi \circ i$, 
so the induced isomorphism
$$
\begin{aligned}
\tilde \pi:  A_n e_n \otimes_{A_{n-1}} J(\Ga') &\otimes_{A_{n-1}} e_n A_n/ \beta_{\Ga, \Ga'}( A_n e_n \otimes_{A_{n-1}} J(\Ga) \otimes_{A_{n-1}} e_n A_n) \\ &\to A_n e_n \otimes_{A_{n-1}} J(\Ga')/J(\Ga) \otimes_{A_{n-1}} e_n A_n
\end{aligned}
$$
satisfies $i \circ \tilde \pi = \tilde \pi \circ i$.  Finally, we have an isomorphism $\alpha: J(\Ga')/J(\Ga) \to \Delta^{\la} \otimes_R i(\Delta^{\la})$ satisfying $i \circ \alpha = \alpha \circ i$,  so the map
$$\begin{aligned}
\bar\alpha = \id \otimes \alpha \otimes \id : &A_n e_n \otimes_{A_{n-1}}   J(\Ga')/J(\Ga)  \otimes_{A_{n-1}} e_n A_n \to \\
&A_n e_n \otimes_{A_{n-1}} \Delta^{\la} \otimes_R i(\Delta^{\la}) \otimes_{A_{n-1}} e_n A_n
\end{aligned}$$
satisfies $i \circ \bar\alpha = \bar\alpha \circ i$.
The map $\chi$ is   $  \bar\alpha \circ \tilde \pi \circ \tilde \Phi_\Ga\inv$,  so we have $i \circ \chi = \chi \circ i$.

This completes the proof that $\Ga \mapsto \hat J(\Ga)$ is a $\La_{n-1}$--cell net in $A_n e_n A_n$.  By
Proposition  \ref{lemma: cell net characterization of cellularity},  $A_n e_n A_n$ has a cell datum with partially ordered set equal to $\La_{n-1}$.  Moreover, since
the isomorphisms  $\hat J(\Ga')/\hat J(\Ga) \cong M^\la \otimes_R i(M^\la)$ are actually 
isomorphisms of $A_{n+1}$--$A_{n+1}$--bimodules, the cellular basis $\tilde{\mathcal C}$ of $A_n e_n A_n$
satisfies the property (2) of Definition \ref{gl cell} not only for $a \in A_n e_n A_n$ but also for $a \in A_{n+1}$;
that is $A_n e_n A_n$ is a cellular ideal in $A_{n+1}$.

Statement (3) of the Lemma follows from applying  Remark \ref{remark on extensions of cellular algebras}.
Statement (4)  follows from  the isomorphism $\hat J(\Ga')/\hat J(\Ga) \cong M^\la \otimes_R i(M^\la)$.
\end{proof}

\begin{corollary} \label{corollary: p.o. set for cellular structure}

 The description of the partially ordered set given in Theorem \ref{main theorem}, point (2), is valid for $k = n+1$.
\end{corollary}

\begin{proof}
 Combining point (3) of Proposition \ref{lemma:  cellularity induction step} with the induction assumption  (specifically the description of  $\La_{n-1}$ as the union of copies of 
$\La_{n-1}\spp 0$,  $\La_{n-3}\spp 0$, etc.), we see that $\La_{n+1}$ is the union of copies of 
$\La_{n+1}\spp 0$, $\La_{n-1}\spp 0$,  $\La_{n-3}\spp 0$, etc., with the following partial order:
the partial order agrees with the original partial order on each  $\La_i \spp 0$, and 
$\la > \mu$ if $\la \in \La_i \spp 0$,  $\mu \in \La_j \spp 0$, and $i < j$.  
\end{proof}

For the remainder of Section \ref{section:  basic construction preserves cellularity},  we denote elements of $\La_k$  ($0 \le k \le n+1$)  by ordered pairs $(\la, k)$,  where it is understood that
$\la \in \La_i\spp 0$ for some $i \le k$ with $k -i$ even.

\begin{corollary} \label{corollary: cell modules and An en An}
 Point (3) of Theorem \ref{main theorem} holds for $k = n+1$.
\end{corollary}

\begin{proof}
 The cell modules of $A_{n+1}$ are of two types:   There are the cell modules
$\Delta^{(\la, n+1)}$ with $\la \in \La_{n+1}\spp 0$,  which are actually cell modules of
$A_{n+1}/(A_n e_n A_n) \cong Q_{n+1}$.  These satisfy $$A_n e_n A_n \ \Delta^{(\la, n+1)} = 0.$$
On the other hand,  there are the cell modules of the cellular ideal $A_n e_n A_n$, namely
$\Delta^{(\la, n+1)} = A_n e_n \otimes_{A_{n-1}} \Delta^{(\la, n-1)}$,  with
$\la \in \La_i\spp 0$  for some $i < n+1$ with $n+1 -i$ even.  These satisfy
$$A_n e_n A_n \ \Delta^{(\la, n+1)} = 
A_n e_n A_n e_n \otimes_{A_{n-1}} \Delta^{(\la, n-1)}.
$$
But 
$$
A_n e_n A_n e_n \otimes_R F = A_n^F A_{n-1}^F e_n = A_n^F e_n, 
$$ using framework axiom (\ref{axiom: en An en}),  so we have $$A_n e_n A_n \ \Delta^{(\la, n+1)}  \otimes_R F  =  \Delta^{(\la, n+1)} \otimes_R F ,$$
by application of Lemma \ref{first tensor iso}..
\end{proof}

\subsection{Cell filtrations of  restrictions and induced modules}

Next we show that the restriction of a cell module from $A_{n+1}$ to $A_n$,  and the induction of a cell module from
$A_n$ to $A_{n+1}$,  have cell filtrations.

\begin{proposition} \label{lemma: cell filtration of restrictions}
Let $(\la, n+1) \in \La_{n+1}$, and let $\Delta = \Delta^{(\la, n+1)}$ be the corresponding
cell module of $A_{n+1}$.  Then the restriction of $\Delta$ to $A_n$ has a cell filtration.
\end{proposition}

\begin{proof}  Write $\Res(\Delta)$ for the restriction to $A_n$.

If $A_{n+1} e_n A_{n+1} \ \Delta = 0$, then $\Delta$ is an $Q_{n+1}$--module; moreover,
by framework axiom (\ref{axiom: e(n-1) in An en An}) from Section \ref{subsection: framework axioms},  $A_n e_{n-1} A_n\   \Res(\Delta) = 0$ as well, so $\Res(\Delta)$ is a $Q_{n}$--module.
Then it follows from the assumption of coherence of $(Q_k)_{k \ge 0}$ that $\Res(\Delta)$ has a cell filtration as an $Q_n$--module, hence as an $A_n$--module.

If $A_{n+1} e_n A_{n+1}\  \Delta  \ne 0$,  then $\la \in \La_i\spp 0$  for some $i < n$, and 
$$\Delta \cong   A_n e_n \otimes_{A_{n-1}} \Delta^{(\la, n-1)}.$$ 
Since $A_n e_n \cong A_n$ as $A_n$--$A_{n-1}$ bimodules, 
 $\Res(\Delta) \cong 
\Ind_{A_{n-1}}^{A_{n}}( \Delta^{(\la, n-1)})$,  which has a cell filtration by the induction assumption.
\end{proof}

\begin{lemma} \label{lemma:  sort of flatness}
 Let $R$ be an integral domain with field of fractions $F$.   Let $A$ be a unital $R$--algebra,  $P$ a right $A$--module,  and  $N_1 \subseteq  N_2$ left  $A$--modules,  such that
\begin{enumerate}
\item  $A^F = A \otimes_R F$ is semisimple, and 
\item  $N_2$  and $P \otimes_A N_1$  are free $R$--modules.
\end{enumerate}
Let $\iota : N_1 \to N_2$ denote the injection.  Then $$\id_P \otimes \iota : P \otimes_A N_1 \to  P \otimes_A N_2$$ is injective.
\end{lemma}

\begin{proof}  First, $\iota \otimes \id_F : N_1 \otimes_R F \to N_2 \otimes_R F$ is injective by Lemma \ref{injectivity of iota tensor id(F) with free R modules}.
Write $\beta = \id_P \otimes \iota$, and let 
$$\beta^F  = \id_{P^F} \otimes (\iota \otimes \id_F) : P^F \otimes_{A^F} N_1^F \to 
P^F \otimes_{A^F} N_2^F. 
$$
Since $A^F$ is semisimple, $P^F$ is projective;  hence $\beta^F$ is injective.

Consider the following diagram: 
\begin{diagram}
 P^F \otimes_{A^F} N^F_{1} 	&&\rTo^{\beta^F} &&	 P^F \otimes_{A^F} N^F_{2} 	\\
 \uTo^{\tau_{1}} &&&& \uTo_{\tau_{2}}\\
 P \otimes_{A} N_1 \otimes_R F &&\rTo^{\beta\otimes id_F} &&  P \otimes_{A} N_2 \otimes_R F\\
\uTo^{\a_1}&&&&	\uTo_{\a_2}\\
 P \otimes_{A} N_1 &&\rTo^{\beta}&&  P \otimes_{A} N_2,\\
\end{diagram}
\ignore{
$$\begin{CD}
 P^F \otimes_{A^F} N^F_{1} 	@>\beta^F>>	 P^F \otimes_{A^F} N^F_{2} 	\\
@A\tau_1 AA	@A\tau_2 AA\\
P \otimes_{A} N_1 \otimes_R F @>\beta\otimes id_F>>  P \otimes_{A} N_2 \otimes_R F\\
@A\a_1AA	@AA\a_2A\\
 P \otimes_{A} N_1 @>\beta>>   P \otimes_{A} N_2,\\
\end{CD}$$
}
where $\alpha_i$ is determined by $x \mapsto x \otimes 1_F$   and $\tau_i$ is the isomorphism of Lemma \ref{first tensor iso} ($i = 1, 2$).  Note that 
$\alpha_1$ is injective
by Lemma \ref{lemma injectivity of x to x tensor 1}, since  $ P \otimes_{A} N_1$ is assumed to be free over $R$.  One can check that
$\beta^F \circ \tau_1 \circ \alpha_1 = \tau_2 \circ \alpha_2 \circ \beta$.  It follows that $\beta$ is injective.
\end{proof}

\begin{lemma} \label{lemma:  globalization preserves cell filtrations}
 Let $M$ be an $A_{n-1}$ module with a cell filtration:
$$
(0) = M_0 \subseteq M_1 \subseteq \cdots \subseteq M_t = M,
$$
with $M_j/M_{j-1} \cong \Delta^{(\la_j, n-1)}$ for $1 \le j \le t$.  Then for $1 \le j \le t$,
\begin{enumerate}
\item $A_n e_n \otimes_{A_{n-1}} M_{j}$ is a free $R$--module,
\item $A_n e_n \otimes_{A_{n-1}} M_{j-1}$  imbeds in $A_n e_n \otimes_{A_{n-1}} M_{j}$, and
\item  $(A_n e_n \otimes_{A_{n-1}} M_{j})/(A_n e_n \otimes_{A_{n-1}} M_{j-1}) \cong 
A_n e_n \otimes_{A_{n-1}} \Delta^{(\la_j, n-1)}$.
\end{enumerate}
Thus, the $A_{n+1}$--module $A_n e_n \otimes_{A_{n-1}} M$ has a cell filtration with subquotients
$\Delta^{(\la_j, n+1)} = A_n e_n \otimes_{A_{n-1}} \Delta^{(\la_j, n-1)}$ ($1 \le j \le t$).
\end{lemma}

\begin{proof}  We have $M_1 \cong \Delta^{(\la_1, n-1)}$,  so $A_n e_n \otimes_{A_{n-1}} M_1$ is a free $R$--module.
Fix $j \ge 2$ and suppose that $A_n e_n \otimes_{A_{n-1}} M_{j-1}$ is a free $R$--module.
Let $\iota: M_{j-1} \to M_j$ denote the injection and let $$\beta = \id_{ A_n e_n} \otimes \iota : A_n e_n \otimes_{A_{n-1}} M_{j-1} \to A_n e_n \otimes_{A_{n-1}} M_{j}.$$
Then $\beta$ is injective by an application of Lemma \ref{lemma:  sort of flatness},  with 
$A = A_{n-1}$, $P = A_n e_n$, $N_1 = M_{j-1}$,  and $N_2 = M_j$.
The quotient  $$(  A_{n}e_n \otimes_{A_{n-1}} M_{j},)/\beta(  A_{n}e_n \otimes_{A_{n-1}} M_{j-1})$$ is free over $R$,
because
$$
\begin{aligned}
(  A_{n}e_n \otimes_{A_{n-1}} &M_{j})/\beta(  A_{n}e_n \otimes_{A_{n-1}} M_{j-1}) \\
&\cong
A_{n}e_n \otimes_{A_{n-1}}  (M_j/ M_{j-1}) \\
&\cong A_{n}e_n \otimes_{A_{n-1}}  \Delta^{(\la_j, n-1)}.
\end{aligned}
$$
Consequently,  $ A_{n}e_n \otimes_{A_{n-1}} M_{j}$ is free over $R$.  All the assertions of the lemma now follow by  induction on $j$.
\end{proof}

\begin{lemma} \label{induction to An en An  from An}
 Let  $M$ be an  $A_n$--module, 
and let $\Res(M)$ denote the restriction of $M$ to $A_{n-1}$.  We have
$$A_n e_n A_n \otimes_{A_n} M \cong  A_n e_n \otimes_{A_{n-1}} \Res(M),
$$
as $A_{n+1}$ modules.
\end{lemma}

\begin{proof}  
\ignore{The map from $A_n e_n A_n \times M$ to  
$A_n e_n \otimes_{A_{n-1}} \Res(M)$ determined by
$$
(\sum_i x'_i e_n x''_i, m) \mapsto \sum_i (x'_i e_n \otimes x''_i m)
$$
is $R$--bilinear and $A_n$--balanced, so yields a linear map $\psi: A_n e_n A_n \otimes_{A_n} M \to  A_n e_n \otimes_{A_{n-1}} \Res(M)$.    Likewise the map from $A_n e_n \times M$ to
$A_n e_n A_n \otimes_{A_n} M$ determined by $(x e_n, m) \mapsto x e_n \otimes m$ is
$R$--bilinear and $A_{n-1}$--balanced, so gives a linear map
$\phi : A_n e_n \otimes_{A_{n-1}} \Res(M) \to A_n e_n A_n \otimes_{A_n} M $.   It is straightforward to check that $\psi$ and $\phi$ are inverses.
}
By Corollary \ref{corollary:   a basic construction isomorphism},   we have $A_n e_n A_n \cong  A_n e_n \otimes_{A_{n-1}}  e_n A_n \cong
 A_n e_n \otimes_{A_{n-1}}   A_n$  as $A_{n+1}$--$A_n$  bimodules.   Thus
 $$A_n e_n A_n \otimes_{A_n} M  \cong A_n e_n \otimes_{A_{n-1}}   A_n \otimes_{A_n} M \cong
 A_n e_n \otimes_{A_{n-1}} \Res(M).$$
\end{proof}

\begin{proposition} \label{proposition: cell filtration of induced modules}
 Let ${(\mu, n)} \in \La_n$ and let $\Delta^{(\mu, n)}$ be the corresponding cell module of 
$A_n$.  
\begin{enumerate}
\item
 $A_n e_n A_n \otimes_{A_n} \Delta^{(\mu, n)}$ has  cell filtration (as an
$A_{n+1}$--module).  In particular,    $A_n e_n A_n \otimes_{A_n} \Delta^{(\mu, n)}$ is free as an $R$--module.
\item  $A_n e_n A_n \otimes_{A_n} \Delta^{(\mu, n)}$  imbeds in $\Ind_{A_n}^{A_{n+1}}(\Delta^{(\mu, n)})$,
and $$\Ind_{A_n}^{A_{n+1}}(\Delta^{(\mu, n)})/(A_n e_n A_n \otimes_{A_n} \Delta^{(\mu, n)}) \cong
Q_{n+1} \otimes_{A_n} \Delta^{(\mu, n)}.$$
\item   $Q_{n+1} \otimes_{A_n} \Delta^{(\mu, n)}$ has  cell filtration (as a  $Q_{n+1}$--module, hence as an 
$A_{n+1}$--module).  
\item  $\Ind_{A_n}^{A_{n+1}}(\Delta^{(\mu, n)})$ has a cell filtration.
\end{enumerate}
\end{proposition}

\begin{proof}  For point (1), 
let $\Res(\Delta^{(\mu, n)})$ denote the restriction to $A_{n-1}$.
By Lemma \ref{induction to An en An  from An}, we have
$A_n e_n A_n \otimes_{A_n} \Delta^{(\mu, n)} \cong A_n e_n \otimes_{A_{n-1}} \Res(\Delta^{(\mu, n)})$, as
$A_{n+1}$ modules.  By the induction assumption stated at the beginning of Section 
\ref{section:  basic construction preserves cellularity},  $ \Res(\Delta^{(\mu, n)})$ has cell filtration,
$$
(0) = M_0 \subseteq M_1 \subseteq \cdots \subseteq M_t =  \Res(\Delta^{(\mu, n)}),
$$
with $M_j/M_{j-1} \cong \Delta^{(\la_j, n-1)}$ for some  $(\la_j, n-1) \in \La_{n-1}$.  By Lemma \ref {lemma:  globalization preserves cell filtrations},
$A_n e_n \otimes_{A_{n-1}} \Res(\Delta^{(\mu, n)})$ has a cell filtration with subquotients
$\Delta^{(\la_j, n+1)} =   A_n e_n \otimes_{A_{n-1}} \Delta^{(\la_j, n-1)}$.

Point (2) follows from Lemma \ref{lemma:  sort of flatness} (with left and right modules interchanged),  taking
$A = A_n$,  $P = \Delta^{(\mu, n)}$,  $N_1 = A_n e_n A_n$,  and $N_2 = A_{n+1}$.  Note that
$A_{n+1}$ is a  free $R$--module by Proposition \ref{lemma:  cellularity induction step}, and $A_n e_n A_n \otimes_{A_n} \Delta^{(\mu, n)} $ is a free $R$--module by point (1).  The statement regarding the quotient follows from the right exactness of tensor products.

For $n=1$,  $A_1 = Q_1$, and $\Delta^{(\mu, n)}$ is an $Q_1$--cell module;  statement (3) follows from
the assumption of coherence of $(Q_k)_{k \ge 0}$.   If $n \ge 2$,  then by the induction assumption,
either $A_n e_{n-1} A_n\  \Delta^{(\mu, n)} = \Delta^{(\mu, n)}$,  or $A_n e_{n-1} A_n \ \Delta^{(\mu, n)} = (0)$.  In the former case,
$$
\begin{aligned}
Q_{n+1} \otimes_{A_n} \Delta^{(\mu, n)} &= Q_{n+1} \otimes_{A_n}A_n e_{n-1} A_n\   \Delta^{(\mu, n)} \\
&= Q_{n+1}A_n e_{n-1} A_n \otimes_{A_n}  \Delta^{(\mu, n)} = 0,\\
\end{aligned}
$$
because  $e_{n-1} \in A_{n+1} e_n A_{n+1}$, by the framework axiom (\ref{axiom: e(n-1) in An en An}).   In the latter case,  $A_n e_{n-1} A_n$ annihilates both $Q_{n+1}$ and $\Delta^{(\mu, n)}$,  so both are $A_n/(A_n e_{n-1} A_n) \cong Q_n$--modules.
Thus $Q_{n+1} \otimes_{A_n} \Delta^{(\mu, n)} = Q_{n+1} \otimes_{Q_n} \Delta^{(\mu, n)}$, which has
an $Q_{n+1}$--cell filtration by the assumption of coherence of $(Q_k)_{k \ge 0}$.  This proves point (3).

Finally, we have an exact sequence
$$
0 \to A_n e_n A_n \otimes_{A_n} \Delta^{(\mu, n)} \to \Ind_{A_n}^{A_{n+1}}(\Delta^{(\mu, n)}) \to
Q_{n+1}  \otimes_{A_n} \Delta^{(\mu, n)} \to 0,
$$
where both $A_n e_n A_n \otimes_{A_n} \Delta^{(\mu, n)}$ and $Q_{n+1}  \otimes_{A_n} \Delta^{(\mu, n)} $ have $A_{n+1}$--cell filtrations.  Hence $ \Ind_{A_n}^{A_{n+1}}(\Delta^{(\mu, n)})$ has an  $A_{n+1}$--cell filtration.
\end{proof}

\begin{corollary} \label{corollary:  Ak coherent tower}  
The finite tower $(A_k)_{0 \le k \le n+1}$ is a coherent tower of cellular algebras.
\end{corollary}

\begin{proof} Combine the induction hypothesis,  Proposition \ref{lemma:  cellularity induction step}, Proposition \ref{lemma: cell filtration of restrictions}, and Proposition \ref{proposition: cell filtration of induced modules}.
\end{proof}

\begin{corollary}  \label{corollary: branching diagram obtained by reflections}
  The branching diagram for the finite tower 
  $(A_k^F)_{0 \le k \le n + 1}$ is that obtained by reflections from the branching diagram
of the finite tower $(Q_k^F)_{0 \le k \le n + 1}$.
\end{corollary}

\begin{proof}  From the induction hypothesis, we already know that the branching diagram
for   $(A_k^F)_{0 \le k \le n }$ is  obtained by reflections from the branching diagram
of the finite tower $(Q_k^F)_{0 \le k \le n}$.  So we have only to consider the branching diagram
for $A_{n-1}^F \subseteq A_n^F \subseteq A_{n+1}^F$;  specifically, we need to show that
if $\la \in \La_i \spp 0$ with $i < n+1$ and $n+1 -i $ even,  and $(\mu, n) \in \La_n$ is arbitrary, 
then $$(\mu, n) \nearrow (\la, n+1) \text{\quad if, and only if \quad } (\la, n-1) \nearrow (\mu, n),$$
in the branching diagram for $A_{n-1}^F \subseteq A_n^F \subseteq A_{n+1}^F$,  and the number of 
edges connecting  $(\mu, n)$  and $(\la, n+1)$  is the same as the number of edges connecting 
$(\la, n-1)$  and $ (\mu, n)$.
But this follows from Lemma \ref{lemma: multiplicities in cell filtrations} and the proof of either Proposition \ref{lemma: cell filtration of restrictions},  or Proposition \ref{proposition: cell filtration of induced modules}, point (1).
\end{proof}

\medskip
\noindent
{\em  Conclusion of the proof of Theorem \ref{main theorem}.} \ \   Under the assumption that statements (1)--(4) of the theorem are valid for the finite tower  $(A_k)_{0 \le k \le n}$, for some fixed $n$,  we had to show that they are also valid for the tower  $(A_k)_{0 \le k \le n + 1}$. This was  verified in Corollary
\ref{corollary:  Ak coherent tower},  Corollary \ref{corollary: p.o. set for cellular structure}, Corollary 
\ref{corollary: cell modules and An en An}, and Corollary \ref{corollary: branching diagram obtained by reflections}.

\section{Examples}

\subsection{Preliminaries on tangle diagrams}  \label{subsection:  preliminaries on tangle diagrams}
Several of our examples involve {\em tangle diagrams} in the rectangle $\mathcal R = [0, 1] \times [0, 1]$.
Fix points $a_i \in [0, 1]$,  $i \ge 1$,  with $0 < a_1 < a_2 < \cdots$. Write
$\p i = (a_i, 1)$ and $\overline{ \p i} = (a_i, 0)$.

Recall that a {\em knot diagram} means a collection of piecewise smooth closed curves in the plane
which may have intersections and self-intersections, but only simple
transverse intersections.  At each intersection or crossing, one of the
two strands (curves) which intersect is indicated as crossing
over the other.  

An {\em $(n,n)$--tangle diagram}  is a piece of a
knot diagram in $\mathcal R$  consisting of exactly $n$ topological intervals and possibly some number of closed curves, such that:  (1)   the endpoints of the intervals are the points $\p 1, \dots \p n, \pbar 1, \dots, \pbar n$, and these are the only points of intersection of the family of curves with the boundary of the rectangle, and (2)  each interval intersects the boundary of the rectangle transversally.

An  {\em $(n,n)$--Brauer diagram} is a ``tangle" diagram  containing no closed curves, 
in which information about over and under crossings is ignored.  Two Brauer diagrams are identified if the pairs of boundary points joined by curves is the same in the two diagrams.
By convention, there is a unique $(0, 0)$--Brauer diagram,  the empty diagram with no curves.
For $n \ge 1$,  the number of $(n,n)$--Brauer diagrams is $(2n-1)!! = (2n-1)(2n-3)\cdots (3)(1)$.

A {\em Temperley--Lieb} diagram is a Brauer diagram without crossings.  For $n \ge 0$,  the  number of $(n, n)$--Temperley--Lieb diagrams is the Catalan number $\frac{1}{n+1} {2n \choose n}$.

For any of these types of diagrams, we call $P = \{\p 1, \dots, \p n, \pbar 1,\dots,  \pbar n\}$ the set of {\em vertices} of the diagram,  $P^+ =    \{\p 1, \dots, \p n\}$ the set of {\em top vertices},  and
$P^- =  \{\pbar 1,\dots,  \pbar n\}$ the set of {\em bottom vertices}.  A curve or {\em strand} in the diagram is called a {\em vertical} or {\em through} strand if it connects a top vertex and a bottom vertex,  and a {\em horizontal} strand if it connects two top vertices or two bottom vertices.

\subsection{The Brauer algebras}
\subsubsection{Definition of the Brauer algebras} \label{subsection: Brauer algebras}

Let $S$ be a commutative ring with identity, with a distinguished element $\delta$.
The Brauer algebra $\br_n(S, \delta)$ is the free $S$--module with basis the set of $(n, n)$--Brauer diagrams, and with multiplication defined as follows.
The product of two Brauer diagrams is defined
to be a certain multiple of another Brauer diagram.  Namely, given two
Brauer diagrams $a, b$, first ``stack" $b$ over $a$; the result is a planar tangle that may contain some number of closed curves.    Let $r$ denote the number of closed curves, and let $c$ be the Brauer
diagram obtained by removing all the closed curves.  Then
$
a b = \delta^r c.
$

\begin{definition}\rm
For $n \ge 1$, the {\em Brauer algebra} $\br_n(S, \delta)$ over $S$ with parameter $\delta$ is the free $S$-module with basis the set of                     
$(n,n)$-Brauer diagrams, with the bilinear product determined by the
multiplication of Brauer diagrams.  In particular, $\br_0(S, \delta) = S$.
\end{definition}

Note that the Brauer diagrams with only vertical strands are in
bijection with permutations of $\{1, \dots, n\}$, and that the
multiplication of two such diagrams coincides with the multiplication of
permutations.  Thus the  Brauer algebra contains the group algebra $S\S_n$ of
the permutation group $\mathfrak S_n$.  The identity element of the Brauer algebra is the diagram corresponding to the trivial permutation.

\subsubsection{Brief history of the Brauer algebras}
The Brauer algebras were introduced by  Brauer~\cite{Brauer} as a device
for studying the invariant theory of orthogonal and symplectic groups.
Wenzl ~\cite{Wenzl-Brauer} observed that generically,  the sequence of Brauer algebras (over a field) is obtained
by repeated Jones basic constructions from the symmetric group algebras;  he used this to show that
$\br_n(k, \delta)$ is semisimple, when $k$ is a field of characteristic zero and $\delta$ is not an integer.
Graham and Lehrer ~\cite{Graham-Lehrer-cellular}  showed that the Brauer algebras are cellular, and classified the simple modules of $\br_n(k, \delta)$ when $k$ is a field and $\delta$ is arbitrary.  Another illuminating proof of cellularity of
the Brauer algebras was given by K\"onig and Xi  ~\cite{KX-Brauer}.    Enyang's two proofs of cellularity for 
Birman--Wenzl algebras ~\cite{Enyang1, Enyang2}  also apply to the Brauer algebras.

\subsubsection{Some properties of the Brauer algebras}  In this section, write $\br_n$ for $\br_n(S, \delta)$.  
For $n \ge 1$, let $\iota$ denote  the map  from  $(n,n)$--Brauer diagrams to
$(n+1, n+1)$--Brauer  diagrams that adds an additional strand  to a diagram,  connecting $\p {n+1}$ to
$\pbar {n+1}$.
$$
\iota: \quad \inlinegraphic{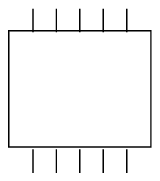} \quad \mapsto \quad 
\inlinegraphic{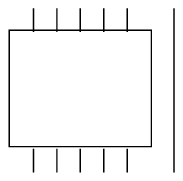}
$$
The linear extension of $\iota$ to $\br_n$ is an injective unital homomorphism into
$\br_{n+1}$.  Using $\iota$,  we identify $\br_n$  with its image in $\br_{n+1}$.

For $n \ge 1$ define a map     $\cl$  from  $(n,n)$--Brauer diagrams into $\br_{n-1}$ as follows.  First ``partially close" a given $(n,n)$--Brauer diagram by adding an additional smooth curve connecting $\p n$ to $\pbar n$,
$$
 \inlinegraphic{tangle_box2} \quad \mapsto \quad 
\inlinegraphic{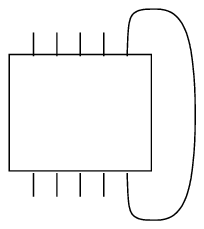}.
$$
In case the resulting ``tangle" contains a closed curve (which happens precisely when the original diagram already had a strand connecting $\p n$ to $\pbar n$), remove this loop and replace it with a factor of $\delta$.  The linear extension of $\cl$ to $\br_n$ is a (non-unital)
$\br_{n-1}$--$\br_{n-1}$ bimodule map, and $\cl\circ\iota(x)  = \delta\  x$ for $x \in 
\br_n$.

If $\delta$ is invertible in $S$, we can define $\eps_n = (1/\delta) \cl$, which is a 
conditional expectation, that is, a unital $\br_{n-1}$--$\br_{n-1}$ bimodule map.  We have
${\eps_{n+1}}\circ\iota(x)  =  x$ for $x \in 
\br_n$.   The map $\eps = \eps_1\circ \cdots \circ \eps_n : \br_n \to \br_0 \cong S$ is a normalized trace; that is, $\eps(\1) = 1$ and $\eps(a b) = \eps(b a)$ for all $a, b$.    The value of $\eps$ on a Brauer diagram $d$ is obtained as follows:  first close all the strands of $d$ by introducing new curves joining $\p j$ to $\pbar j$ for all $j$;   let $c$  be the number of components (closed loops) in the resulting
$(0,0)$--tangle;  then $\eps(d) = \delta^{c - n}$  if $d \in \br_n$.  The trace and condition expectation play an essential role in Wenzl's treatment of the structure of the Brauer algebra over $\Q(\deltabold)$
~\cite{Wenzl-Brauer},  and thus implicitly in our verification of the framework axioms in Proposition
\ref{proposition: framework axioms for Brauer}.

The involution $i$ on $(n, n)$--Brauer diagrams which reflects a diagram in the axis $y = 1/2$
extends linearly to an algebra involution of $\br_n$.  We have $\iota \circ i = i \circ \iota$  and $\cl\circ i = i \circ \cl$.

The products $a b$ and $b a$  of two Brauer diagrams have  at most as many through strands as $a$.  Consequently, the span of diagrams with at most $r$ through strands    ($r \le n$ and $n-r$ even) is a two--sided ideal $J_r$  in $\br_n$.    $J_r$ is $i$--invariant.

Let $e_j$ and $s_j$ denote the $(n, n)$--Brauer diagrams:
$$
e_j =  \inlinegraphic[scale=.7]{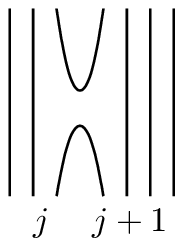}\qquad
s_j =  \inlinegraphic[scale= .7]{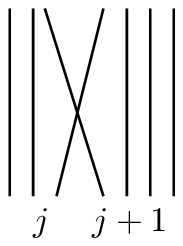} 
$$
Note that $e_j^2 = \delta e_j$,  so $e_j$ is an essential idempotent if $\delta \ne 0$, and nilpotent if $\delta = 0$.
We have $i(e_j) = e_j$ and $i(s_j) = s_j$.  
It is easy to see that $e_1, \dots, e_{n-1}$ and $s_1, \dots, s_{n-1}$ generate $\br_n$ as an algebra.   

Let $r \le n$ with $n - r$ even, and let $f_r = e_{r+1} e_{r+3} \cdots e_{n-1}$.
Any Brauer diagram with exactly $r$ through strands can be factored as $\pi_1  f_r \pi_2$,  where  $\pi_i$  are permutation diagrams.  Consequently, $J_r$ is generated by $f_r$.  In particular the
ideal  $J = J_{n-2}$ spanned by diagrams with fewer than $n$ through strands is generated by $e_{n-1}$.  We have   $\br_n/J \cong S\S_n$, as algebras with involutions.

\begin{lemma}  \label{lemma:  Brauer  axiom 6} Write $\br_n$  for $\br_n(S, \delta)$.
\begin{enumerate}
\item 
For $n \ge 2$, $e_{n} \br_{n} e_{n} = \br_{n-1} e_{n}$.
\item  $e_1 \br_1 e_1 =  \delta B_0  e_1$
\item  For $n \ge 2$,  
$e_{n}$  commutes with $ \br_{n-1} $.  
\end{enumerate}
\end{lemma}

\begin{proof}     For $n \ge 2$,   if $x$ is an $(n, n)$--Brauer diagram, then $e_{n} x e_{n} \in \br_{n-1}\, e_{n}$.  Thus, $e_{n} \br_{n}  \,e_{n} \break \subseteq  \br_{n-1} \,e_{n}$.   On the other hand, for $x \in \br_{n-1}$, we 
have     $e_{n} x e_{n-1} e_{n} =  x e_{n}$.  Hence, 
 $e_{n} \br_{n}  e_{n}  \supseteq   \br_{n-1} \,e_{n}$.  This proves (1).  Points (2) and (3) are obvious.
\end{proof}

\begin{lemma} \label{B(n+1) e(n) = B(n) e(n)  for Brauer algebras}  Write $\br_n$  for $\br_n(S, \delta)$.
For $n \ge 1$, 
 $\br_{n+1}\, e_{n} = \br_{n} \, e_{n}$.  Moreover,  
 $x \mapsto  x e_{n}$  is injective from $ \br_{n}$ to $ \br_{n+1}$.
\end{lemma}

\begin{proof}  
By ~\cite{Wenzl-Brauer}, Proposition 2.1,  any $(n+1, n+1)$--Brauer diagram is either already in 
$\br_{n}$, or can be written in the form $a \chi_{n} b$, with $a, b \in \br_{n}$ and
$\chi_{n} \in \{e_{n}, s_{n}\}$.  Applying this again to $b$,  either $b \in \br_{n-1}$, or $b$ can be factored as $b_1 \chi_{n-1} b_2$,  with $b_i \in \br_{n-1}$ and $\chi_{n-1} \in \{e_{n-1}, s_{n-1}\}$.  Since $e_{n}^2 = \delta e_{n}$  and $s_{n} e_{n} = e_{n}$,  it follows that if
$b \in \br_{n-1}$,  then $a \chi_{n} b e_{n} = a b  \chi_{n}  e_{n} \in \br_{n} e_{n}$.
If $b = b_1 \chi_{n-1} b_2$, then $a \chi_{n} b e_{n} =  a b_1  \chi_{n} \chi_{n-1} e_{n} b_2$.
Now we can apply the following identities:    $e_{n} \chi_{n-1} e_{n} = e_{n}$ for 
$\chi_{n-1} \in \{e_{n-1}, s_{n-1}\}$,  $s_{n} e_{n-1} e_{n} = s_{n-1} e_{n}$,  and
 $s_{n} s_{n-1} e_{n} = e_{n-1} e_{n}$ to conclude that  $a \chi_{n} b e_{n}  \in \br_{n} e_{n}$.  This shows that $\br_{n+1} e_{n} = \br_{n} e_{n}$. 

For $x \in \br_{n}$,  we have
$\cl( x e_{n}) = x$,  so the map $x \mapsto x e_{n}$ is injective from
$\br_{n}$  to $\br_{n}e_{n}$.  \end{proof}

\subsubsection{Verification of framework axioms for the Brauer algebras}

We take $R = \Z[\deltabold]$,  where $\deltabold$ is an indeterminant.   Then $R$ is the universal ground ring for the Brauer algebras;  for any commutative ring $S$ with distinguished element $\delta$, we have
$\br_n(S, \delta) \cong  \br_n(R, \deltabold)\otimes_R  S$.  Let $F = \Q(\deltabold)$ denote the field of fractions of $R$.  Write $\br_n = \br_n(R, \deltabold)$.

\begin{proposition} \label{proposition: framework axioms for Brauer} The two sequence of $R$--algebras  $(\br_n)_{n \ge 0}$ and $(R \S_n)_{n \ge 0}$  satisfy the framework axioms of  Section \ref{subsection: framework axioms}.
\end{proposition}

\begin{proof} According to Example \ref{example: Hn coherent tower},  $(R\S_n)_{n \ge 0}$ is a coherent tower of cellular algebras, so axiom (\ref{axiom Hn coherent}) holds.
Framework axioms (\ref{axiom: involution on An}) and (\ref{axiom: A0 and A1}) are evident.
$\br_n^F$ is split semisimple by ~\cite{Wenzl-Brauer}, Theorem 3.2, so axiom (\ref{axiom: semisimplicity}) holds.

We take $e_{n-1} \in \br_n$ to be the element defined in the previous section.   Let us verify the axioms (\ref{axiom:  idempotent and Hn as quotient of An})--(\ref{axiom: e(n-1) in An en An}) involving $e_{n-1}$.  As observed above,  $e_{n-1}$ is $i$--invariant,  $J = \br_n e_{n-1} \br_n$ is the ideal spanned by diagrams with fewer than $n$ through strands,  and $\br_n/J \cong R\S_n$ as algebras with involution.  This verifies axiom (\ref{axiom:  idempotent and Hn as quotient of An}).
Axiom (\ref{axiom: en An en}) follows from Lemma \ref{lemma:  Brauer  axiom 6} and axiom (\ref{axiom:  An en}) from 
 Lemma \ref{B(n+1) e(n) = B(n) e(n)  for Brauer algebras}.   Axiom (\ref{axiom: e(n-1) in An en An}) holds because $e_{n-1} e_n e_{n-1} = e_{n-1}$.
\end{proof}

\begin{corollary} For any commutative ring $S$ and for any $\delta \in S$, 
the sequence of  Brauer algebras $(B_n(S, \delta))_{n \ge 0}$ is a coherent tower of cellular algebras.
$B_n(S, \delta)$  has cell modules indexed by all Young diagrams of size $n$, $n-2$, $n-4, \dots$.   The cell module labeled by
a Young diagram $\la$ has a basis labeled by up--down tableaux of length $n$ and shape $\la$.
\end{corollary}

\subsection{The Jones--Temperley--Lieb algebras}

\subsubsection{Definition of the Jones--Temperley--Lieb algebras} 
Let $S$ be a commutative ring with identity, with distinguished element $\delta$.  The Jones--Temperley--Lieb algebra $\tl_n(S, \delta)$ is the unital $S$--algebra with generators $e_1, \dots, e_{n-1}$ satisfying the relation:
\begin{enumerate}
\item $e_j^2 = \delta e_j$,
\item $e_j e_{j \pm 1} e_j = e_j$,
\item  $e_j e_k =  e_k e_j$, if $|j - k| \ge 2$,
\end{enumerate}
whenever all indices involved are in the range from $1$ to $n-1$.

\subsubsection{Diagramatic realization of the Jones--Temperley-Lieb algebras}

The $S$--span  \break  $\tilde \tl_n(S, \delta)$ of Temperley--Lieb diagrams is a subalgebra of the Brauer algebra.  We have an algebra map $\varphi$  from $\tl_n(S, \delta)$ to $\tilde \tl_n(S, \delta)$, determined by $e_j \mapsto e_j$ for $1 \le j \le n-1$.
Kauffman shows (\cite{Kauffman}, Theorem 4.3)  that the map is an isomorphism.  In fact, to show that $\varphi$ is surjective, it suffices to show that any \TL diagram can be written as a product of $e_j$'s.  Kauffman indicates by example how this is to be done, and it is not difficult to invent a measure of complexity of \TL diagrams and to show this formally, by induction on complexity.  For injectivity,  Jones shows (\cite{Jones-index}, p.\  14)  that  $\tl_n(S, \delta)$ is spanned by a family $\mathbb B$ of $\frac{1}{n+1}{2n \choose n}$ reduced words in the $e_j$'s.  
Since $\varphi$ is surjective and $\tilde \tl_n(S, \delta)$  is a free $S$--module of rank $\frac{1}{n+1}{2n \choose n}$, it follows easily that $\mathbb B$ is a basis and $\varphi$ is an isomorphism.  Because of this, we will no longer distinguish between $\tl_n(S, \delta)$ and  $\tilde \tl_n(S, \delta)$.

\subsubsection{Brief history of the Jones--Temperley--Lieb algebras}  The Jones-Temperley-Lieb algebras were introduced by Jones in his study of subfactors ~\cite{Jones-index}  and then employed by him to define the Jones link invariant ~\cite{jones-invariant}.    The name derives from the appearance of specific representations of the algebras in statistical mechanics that had been found some years earlier.  By now, there is a huge literature related to these algebras because of their multiple roles in subfactor theory,  invariants of links and 3-manifolds, statistical mechanics and quantum field theory.  The Jones--Temperley--Lieb algebras were shown to be cellular in ~\cite{Graham-Lehrer-cellular}.  Several other proofs of cellularity are known, for example ~\cite{wilcox-cellular, green-martin-tabular}.

\subsubsection{Some properties of the Jones--Temperley--Lieb algebra}  The Brauer algebra maps
$\iota$,  $\cl$,  $\vareps_n$  (when $\delta$ is invertible),  and $i$  restrict to maps
of the Jones--\TL algebras having similar properties.  For example, $i$ is an algebra involution on each
$\tl_n(S, \delta)$ and $i \circ \iota = \iota\circ i$.

The span of \TL diagrams having at least one horizontal strand is an ideal $J$ in $\tl_n(S, \delta)$, and
$\tl_n(S, \delta)/J \cong S$.
The proof of surjectivity of $\varphi$ sketched above shows that any \TL diagram with at least one horizontal edge
can be written as a non-trivial product of $e_j$'s; so $J$ is equal to the ideal generated by all of the $e_j$'s.
However, the identities $e_j e_{j+1} e_j = e_j$ imply that $J$ is the ideal generated by $e_{n-1}$.

\subsubsection{Verification of the framework axioms for the Jones--\TL algebras}
We take $R = \Z[\deltabold]$,  where $\deltabold$ is an indeterminant.   Then $R$ is the universal ground ring for the  Jones--\TL algebras;  for any integral domain $S$ with distinguished element $\delta$, we have
$\tl_n(S, \delta) \cong  \tl_n(R, \deltabold)\otimes_{R} S$.  Let $F = \Q(\deltabold)$ denote the field of fractions of $R$.  Write $\tl_n = \tl_n(R, \deltabold)$.

\begin{proposition} \label{proposition: framework axioms for TL} The two sequences of $R$--algebras  $(\tl_n)_{n \ge 0}$ and $(R)_{n \ge 0}$  satisfy the framework axioms of  Section \ref{subsection: framework axioms}.
\end{proposition}

\begin{proof} Axioms (\ref{axiom Hn coherent}), (\ref{axiom: involution on An}), and  (\ref{axiom: A0 and A1}) are obvious.    
For semisimplicity of $\tl_n^F$,  see  ~\cite{GHJ},  
Theorem 2.8.5.    This gives axiom  (\ref{axiom: semisimplicity}).
We checked axiom (\ref{axiom:  idempotent and Hn as quotient of An})  in the previous section.  The proof for
axiom (\ref{axiom: en An en}) is the same as for the Brauer algebras.

According to ~\cite{Jones-index}, Lemma 4.1.2, any 
 $(n+1, n+1)$--\TL diagram is either already in 
$\tl_{n}$, or can be written in the form $a e_{n} b$, with $a, b \in \tl_{n}$.   Given this,  the verification of
axiom (\ref{axiom:  An en}) is the same as for the Brauer algebras;  we have to use only the identity
$e_{n} e_{n-1} e_{n} = e_{n}$  in place of several similar identities for the Brauer algebras.  

As for the Brauer algebras, axiom 
(\ref{axiom: e(n-1) in An en An})  follows from the identity $e_{n-1} e_n e_{n-1} = e_{n-1}$.
 \end{proof}
 
 \begin{corollary}  For any ring $S$ and $\delta \in S$, the sequence of Jones--Temperley--Lieb algebras $(\tl_n(S, \delta))_{n \ge 0}$ is a coherent tower of cellular algebras.   The cell modules of $\tl_n(S, \delta)$  can be labeled by Young diagrams with one or two rows and size $n$, and the basis of the cell module labeled by $\lambda$ by standard tableaux of shape $\la$.
 \end{corollary}
 
 \begin{proof}  We only have to remark that the vertices on the $n$-th row of the branching diagram for
 $(\tl_k^F)_{k\ge 0}$  (see ~\cite{GHJ},  Lemma 2.8.4) can be labeled by Young diagrams of size $n$ with no more than 2 rows, and the paths on the branching diagram by standard tableaux. 
 (Alternatively,  the vertices on the $n$-th row of the branching diagram can be labeled by Young diagrams with one row and size $n$, $n-2$, $n-4, \dots$,  and the paths on the branching diagram by up--down tableaux.)
 \end{proof}

\subsection{The Birman--Wenzl--Murakami (BMW) algebras}

\subsubsection{Definition of the BMW algebras}
The BMW algebras were first introduced by Birman and Wenzl ~\cite{Birman-Wenzl} and independently by Murakami ~\cite{Murakami-BMW} as abstract algebras defined by generators and relations.  The version of the presentation given here follows ~\cite{Morton-Wassermann}.

\def\hods{\unskip\kern.55em\ignorespaces}
\def\mods{\vskip-\lastskip\vskip4pt}
\def\ods{\vskip-\lastskip\vskip4pt plus2pt}
\def\bods{\vskip-\lastskip\vskip12pt plus2pt minus2pt}

\begin{definition}  \label{definition: BMW algebra}
Let $S$ be a commutative unital ring with invertible elements $\rho$ and $q$ and an element $\delta$ satisfying $\rho\inv - \rho = (q\inv -q)(\delta -1)$.  The {\em Birman--Wenzl--Murakami algebra}
$\bmw n(S; \rho, q, \delta)$ is the unital $S$--algebra 
 with generators $g_i^{\pm 1}$  and
$e_i$ ($1 \le i \le n-1$) and relations:
\begin{enumerate}
\item (Inverses) \hods $g_i g_i\inv = g_i\inv g_i = 1$.
\item (Essential idempotent relation)\hods $e_i^2 = \delta e_i$.
\item (Braid relations) \hods $g_i g_{i+1} g_i = g_{i+1} g_i g_{i+1}$ 
and $g_i g_j = g_j g_i$ if $|i-j|  \ge 2$.
\item (Commutation relations)  \hods $g_i e_j = e_j g_i$  and
$e_i e_j = e_j e_i$  if $|i-j|\ge 2$. 
\item (Tangle relations)\hods $e_i e_{i\pm 1} e_i = e_i$, $g_i
g_{i\pm 1} e_i = e_{i\pm 1} e_i$, and $ e_i  g_{i\pm 1} g_i=   e_ie_{i\pm 1}$.
\item (Kauffman skein relation)\hods  $g_i - g_i\inv = (q - q\inv) (1 - e_i)$.
\item (Untwisting relations)\hods $g_i e_i = e_i g_i = \rho\inv e_i$,
and $e_i g_{i \pm 1} e_i = \rho e_i$.
\end{enumerate}
\end{definition}

\subsubsection{Geometric realization of the BMW algebras}
A geometric realization of the BMW algebra is as the algebra of framed $(n, n)$--tangles in the disc cross the interval, modulo certain skein relations.  It is more convenient, at least for our purposes, to describe this geometric version in terms of tangle diagrams.

First, tangle diagrams can be multiplied by stacking, as for Brauer or \TL diagrams (but closed loops are allowed, and there is no reduction by removing closed loops after stacking).   Recall that our convention is that the product $ab$ of tangle diagrams is given by stacking $b$ over $a$.  This makes $(n, n)$--tangle diagrams into a monoid, the identity being the tangle diagram in which each top vertex $\p j$ is connected to the bottom vertex $\pbar j$ by a vertical line segment, when $n \ge 1$.
(The identity for the monoid of $(0, 0)$--tangle diagrams is the empty tangle.)

\medskip
\vbox{
\begin{eqnarray*}
\text{I}&\quad\ &\inlinegraphic{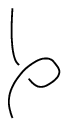} \quad \longleftrightarrow
 \quad \inlinegraphic{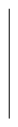} \quad
  \longleftrightarrow  \quad \inlinegraphic{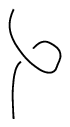}\\
\text{II}&\quad\ &\inlinegraphic[scale =.5] {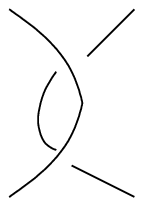} \quad 
\longleftrightarrow 
\quad \inlinegraphic[scale=1.75]{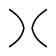} \\
  \text{III}&\quad\ &\inlinegraphic{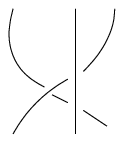} \quad \longleftrightarrow \quad 
\inlinegraphic{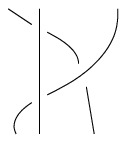} 
\end{eqnarray*} 

\centerline{Reidemeister moves}
}

\medskip
Two tangle diagrams are said to be {\em regularly isotopic} if they are related by a sequence of Reidemeister moves of types II and III, followed by an isotopy of $\mathcal R$ fixing the boundary.
(Reidemeister moves of type I are not allowed.)  See the figure above for the Reidemeister moves.

Stacking of tangle diagrams respects regular isotopy; thus one obtains a monoid structure on the regular isotopy classes of $(n, n)$--tangle diagrams. Let us denote this monoid by $\mathcal U_n$.
Let $S$ be a ring with elements $\rho$, $q$ and $\delta$ as  in the definition of the BMW algebras.
The {\em Kauffman tangle algebra} $\kt n(S; \rho, q, \delta)$ is the monoid algebra $S\ \mathcal U_n$ modulo the following skein relations:
\begin{enumerate}
\item Crossing relation:
$
\quad \inlinegraphic[scale=.6]{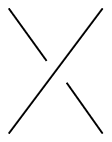} - \inlinegraphic[scale=.3]{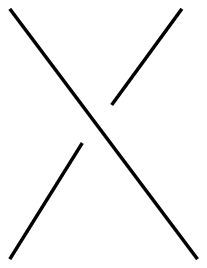} 
\quad = 
\quad
(q\inv - q)\,\left( \inlinegraphic[scale=1.2]{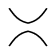} - 
\inlinegraphic[scale=1.2]{id_smoothing}\right).
$
\item Untwisting relation:
$\quad 
\inlinegraphic{right_twist} \quad = \quad \rho \quad
\inlinegraphic{vertical_line} \quad\ \text{and} \quad\ 
\inlinegraphic{left_twist} \quad = \quad \rho\inv \quad
\inlinegraphic{vertical_line}. 
$
\item  Free loop relation:  $T\, \cup \, \bigcirc = \delta \, T, $  where $T\, \cup \, \bigcirc$ means the union of a tangle diagram $T$ and a closed loop having no crossings with $T$.
\end{enumerate}

Let $E_j$ and $G_j$ denote the following $(n,n)$--tangle diagrams:
$$
E_j =  \inlinegraphic[scale=.7]{ordinary_E_j}\qquad
G_j =  \inlinegraphic[scale= .7]{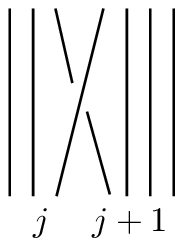} 
$$
Morton and Wassermann \cite{Morton-Wassermann} showed that the assignments $e_j \mapsto E_j$ and $g_j \mapsto G_j$   determine an isomorphism from $\bmw n(S; \rho, q, \delta)$ to
$\kt n(S; \rho, q, \delta)$.  Given this, we will no longer distinguish between the BMW algebras and the Kauffman tangle algebras.  (However, we remark that it is possible to use our techniques to recover
the theorem of Morton and Wasserman, using only results  in the original paper of Birman and Wenzl;  we prove the analogous isomorphism theorem  for the cyclotomic BMW algebras in Section \ref{The cyclotomic Birman--Wenzl--Murakami (BMW) algebras}, and the result for the ordinary BMW algebras is a special case.)

\subsubsection{Brief history of the BMW algebras} The origin of the BMW algebras was in knot theory.  Kauffman defined \cite{Kauffman} an invariant of regular isotopy for links in $S^3$, determined by skein relations.  Birman and Wenzl ~\cite{Birman-Wenzl} and Murakami ~\cite{Murakami-BMW}  then defined the BMW algebras in order to give an algebraic setting for the Kauffman invariant.  The BMW algebras were implicitly modeled on algebras of tangles.  The definition of the Kauffman tangle algebra was made explicit by Morton and Traczyk ~\cite{Morton-Traczyk}, who also showed that $\kt n(S; \rho, q, \delta)$ is free as an $S$--module of rank
$(2n-1)!!$.  Morton and Wassermann ~\cite{Morton-Wassermann} showed that the BMW algebras and Kauffman tangle algebras are isomorphic.

Xi showed ~\cite{Xi-BMW} that the tangle basis of Morton and Traczyk is a cellular basis.  Enyang has exhibited two cellular bases of BMW algebras;  the first ~\cite{Enyang1} is a tangle type basis, and the second ~\cite{Enyang2} is a basis indexed by up--down tableaux, which demonstrates the coherence of the cellular structures on $(\bmw n)_{n \ge 0}$.

\subsubsection{Some properties of the BMW algebras} 
In the following, we write $\bmw n$  for  \break $\bmw n(S; \rho, q, \delta)$.

 The BMW algebras have an algebra involution 
$i$ uniquely determined by $i(e_j) = e_j$ and $i(g_j) = g_j$ for all $j$.  The action of $i$ on tangle diagrams is by the rotation through the axis $y = 1/2$.  (It is by rotation rather than reflection, since the reflection would take $g_j \mapsto g_j\inv$.)

For $n \ge 0$, there is a unique homomorphism $\iota$ from $\bmw n$  to $\bmw {n+1}$  determined by $e_i \mapsto e_i$  and $g_i \mapsto g_i$  for $1 \le i \le n-1$.   On the level of tangle diagrams, the map is given by adding
a new vertical strand connecting $\p {n+1}$  and $\pbar {n+1}$,  as for the Brauer algebras.

For $n \ge 1$, a map $\cl$ from $(n, n)$--tangle diagrams to 
$(n-1, n-1)$--tangle diagrams can be defined  as for Brauer diagrams.  The linear extension of this map respects
regular isotopy and the Kauffman skein relations, so determines a linear map from 
$\bmw n$ to $\bmw {n-1}$.   We have $i  \circ \cl = 
 \cl \circ i$ and $ \cl \circ \iota = \delta \, x$.  Moreover,  for   $x \in \bmw{n}$, we have
 $x  = {\rm cl} ( \iota(x) e_n)$,   so it follows that $\iota: \bmw n \to \bmw {n+1}$  is injective.
 The involution  $i$ and inclusion $\iota$ satisfy $i\circ \iota = \iota\circ i$.
  Using $\iota$, we identify $\bmw n$ as a subalgebra of $\bmw {n+1}$.

 If $\delta$ is invertible in $S$, we can define $\eps_n = (1/\delta) \cl$, which is a conditional expectation, that is, an unital $\bmw {n-1}$--$\bmw {n-1}$  bimodule map.  We have  $\eps_{n+1} \circ \iota(x) = x$ for $x \in \bmw n$.
 
 The ideal $J$  in $\bmw n$  generated by $e_{n-1}$  contains $e_j$ for all $j$ because of  the relations
 $e_j e_{j+1} e_j = e_j$.    It follows from the BMW relations that $\bmw n/J $ is isomorphic to the Hecke algebra $H_n(S; q^2)$  with the quadratic relation $g_j - g_j\inv = q - q\inv$,  or 
 $(g_j - q) (g_j + q\inv) = 0$. 
 
 \begin{lemma} \label{lemma:  axiom 6 for BMW} \mbox{}
 \begin{enumerate}
\item 
For $n \ge 2$, $e_{n} \bmw {n} e_{n} = \bmw {n-1} e_{n}$.
\item  $e_1 \bmw 1 e_1 =  \delta\, \bmw 0  \,e_1$
\item  For $n \ge 1$,  
$e_{n}$  commutes with $ \bmw {n-1} $.  
\end{enumerate}
 \end{lemma}
 
 \begin{proof}   The proof is the same as that of Lemma \ref{lemma:  Brauer  axiom 6} for the Brauer algebras, using the tangle realization of the BMW algebras.
 \end{proof}
 
 \begin{lemma} \label{Bn e(n-1) = B(n-1) e(n-1)  for BMW algebras}.  For $n \ge 1$, 
$\bmw {n+1}  \,e_{n} =  \bmw {n} \,e_{n}$.  Moreover,  
 $x \mapsto  x e_{n}$  is injective from $\bmw {n}$ to $\bmw {n} e_{n}$.
\end{lemma}

\begin{proof}  According to ~\cite{Birman-Wenzl}, Lemma 3.1, any $(n+1, n+1)$--tangle is already in $\bmw {n}$, or it can be written as  a linear combination of elements
$a \chi_{n} b$, with $a, b \in \bmw {n}$ and $\chi_{n} \in \{e_{n}, g_{n}\}$. Given this, the proof of 
the lemma is the same as the proof of Lemma \ref{B(n+1) e(n) = B(n) e(n)  for Brauer algebras}  
 for the Brauer algebras, using the tangle relations and untwisting relations of Definition \ref{definition: BMW algebra} in place of similar identities for the Brauer algebras.
\end{proof}
 
 \subsubsection{Verification of the framework axioms for the BMW algebras}
 The generic or universal ground ring for the BMW algebras is
 $$
 R = \Z[\rhobold^{\pm1}, \qbold^{\pm1}, \deltabold]/\langle \rhobold\inv - \rhobold = (\qbold\inv - \qbold)(\deltabold -1) \rangle,
 $$
 where $\rhobold$, $\qbold$, and $\deltabold$ are indeterminants over $\Z$.  Suppose that $S$ is an appropriate ground ring for the BMW algebras; that is, 
 $S$ is a  commutative unital ring with invertible elements $\rho$ and $q$ and an element $\delta$ satisfying $\rho\inv - \rho = (q\inv -q)(\delta -1)$.  Then $\bmw n(S; \rho, q, \delta) \cong \bmw n(R; \rhobold, \qbold, \deltabold) \otimes_R S$.   

$R$ is an integral domain whose field of fractions is $F \cong \Q(\rhobold, \qbold)$  (with $\deltabold =  \break
 (\rhobold\inv - \rhobold)/(\qbold\inv - \qbold) + 1$ in $F$.)  We write $\bmw n$ for 
 $\bmw n(R; \rhobold, \qbold, \deltabold)$  and $H_n$ for $H_n(R; \qbold^2)$  in this section.

\begin{proposition}\label{proposition: framework axioms for BMW}
 The two sequences of algebras $(\bmw n)_{n \ge 0}$  and $(H_n)_{n \ge 0}$ satisfy the framework axioms of  Section \ref{subsection: framework axioms}.
\end{proposition}

\begin{proof}  According to example \ref{example: Hn coherent tower} ,  $(H_n)_{n \ge 0}$ is a coherent tower of cellular algebras, so axiom (\ref{axiom Hn coherent}) holds.  Axioms (\ref{axiom: involution on An}) and (\ref{axiom: A0 and A1})  are evident.
$\bmw n^F$ is semisimple by ~\cite{Birman-Wenzl}, Theorem 3.7,  or ~\cite{Wenzl-BCD}, Theorem 3.5.  Thus axiom (\ref{axiom: semisimplicity})  holds.

We observed above that $\bmw n/\bmw n e_{n-1} \bmw n \cong H_n$;  it is easy to check that the isomorphism respects the involutions.  Thus axiom (\ref{axiom:  idempotent and Hn as quotient of An}) holds.  Axiom 
(\ref{axiom: en An en})  follows from  Lemma \ref{lemma:  axiom 6 for BMW} and axiom (\ref{axiom:  An en}) from Lemma \ref{Bn e(n-1) = B(n-1) e(n-1)  for BMW algebras}.
Finally, axiom (\ref{axiom: e(n-1) in An en An}) holds again because of the relation $e_{n-1} e_n e_{n-1} = e_{n-1}$.
\end{proof}

\begin{corollary}  Let $S$ be any ground ring for the BMW algebras, with parameters $\rho$, $q$, and 
$\delta$.  
The sequence of  BMW algebras $(\bmw n(S; \rho, q, \delta))_{n \ge 0}$ is a coherent tower of cellular algebras.
$\bmw n(S; \rho, q, \delta)$  has cell modules indexed by all Young diagrams of size $n$, $n-2$, $n-4, \dots$.   The cell module labeled by
a Young diagram $\la$ has a basis labeled by up--down tableaux of length $n$ and shape $\la$.
\end{corollary}

\subsection{The cyclotomic Birman--Wenzl--Murakami (BMW) algebras}
\label{The cyclotomic Birman--Wenzl--Murakami (BMW) algebras}

\subsubsection{Definition of the cyclotomic BMW algebras}

In general, our notation will follow  ~\cite{GH3}.  In order to simplify statements, we establish the following convention.

\begin{definition} Fix an integer $r \ge 1$.   A {\em ground ring} 
$S$ is a commutative unital ring with parameters $\rho$, $q$, $\delta_j$   ($j \ge 0$),  and
$u_1, \dots, u_r$, with    $\rho$, $q$,   and $u_1, \dots, u_r$  invertible,  and with $\rho\inv - \rho=   (q\inv -q) (\delta_0 - 1)$.
\end{definition}

\begin{definition} \label{definition:  cyclotomic BMW}
Let $S$ be a ground ring with
parameters $\rho$, $q$, $\delta_j$   ($j \ge 0$),  and
$u_1, \dots, u_r$.
The {\em cyclotomic BMW algebra}  $\bmw{n, S, r}(u_1, \dots, u_r)$  is the unital $S$--algebra
with generators $y_1^{\pm 1}$, $g_i^{\pm 1}$  and
$e_i$ ($1 \le i \le n-1$) and relations:
\begin{enumerate}
\item (Inverses)\hods $g_i g_i\inv = g_i\inv g_i = 1$ and 
$y_1 y_1\inv = y_1\inv y_1= 1$.
\item (Idempotent relation)\hods $e_i^2 = \delta_0 e_i$.
\item (Affine braid relations) 
\begin{enumerate}
\item[\rm(a)] $g_i g_{i+1} g_i = g_{i+1} g_ig_{i+1}$ and 
$g_i g_j = g_j g_i$ if $|i-j|  \ge 2$.
\item[\rm(b)] $y_1 g_1 y_1 g_1 = g_1 y_1 g_1 y_1$ and $y_1 g_j =
g_j y_1 $ if $j \ge 2$.
\end{enumerate}
\item[\rm(4)] (Commutation relations) 
\begin{enumerate}
\item[\rm(a)] $g_i e_j = e_j g_i$  and
$e_i e_j = e_j e_i$  if $|i-
j|
\ge 2$. 
\item[\rm(b)] $y_1 e_j = e_j y_1$ if $j \ge 2$.
\end{enumerate}
\item[\rm(5)] (Affine tangle relations)\vadjust{\vskip-2pt\vskip0pt}
\begin{enumerate}
\item[\rm(a)] $e_i e_{i\pm 1} e_i = e_i$,
\item[\rm(b)] $g_i g_{i\pm 1} e_i = e_{i\pm 1} e_i$ and
$ e_i  g_{i\pm 1} g_i=   e_ie_{i\pm 1}$.
\item[\rm(c)\hskip1.2pt] For $j \ge 1$, $e_1 y_1^{ j} e_1 = \delta_j e_1$. 
\vadjust{\vskip-
2pt\vskip0pt}
\end{enumerate}
\item[\rm(6)] (Kauffman skein relation)\hods  $g_i - g_i\inv = (q - q\inv)(1- e_i)$.
\item[\rm(7)] (Untwisting relations)\hods $g_i e_i = e_i g_i = \rho \inv e_i$
 and $e_i g_{i \pm 1} e_i = \rho  e_i$.
\item[\rm(8)] (Unwrapping relation)\hods $e_1 y_1 g_1 y_1 = \rho e_1 = y_1 
g_1 y_1 e_1$.
\item[\rm(9)](Cyclotomic relation) \hods $(y_1 - u_1)(y_1 - u_2) \cdots (y_1 - u_r) = 0$.
\end{enumerate}
\end{definition}

Thus,  a cyclotomic BMW algebra is the quotient of the affine BMW algebra  ~\cite{GH1}, by the cyclotomic relation $(y_1 - u_1)(y_1 - u_2) \cdots (y_1 - u_r) = 0$.

 \subsubsection{Geometric realization}  \label{subsubsection:  cyclotomic BMW geometric realization}
 We recall from ~\cite{GH1} that the affine BMW algebra is isomorphic to the affine Kauffman tangle algebra, which is an algebra of ``affine tangle diagrams," modulo Kauffman skein relations.    An affine $(n,n)$--tangle diagram is just an ordinary $(n+1, n+1)$--tangle diagram with a fixed
 vertical strand connecting $\p 1$  and $\pbar  1$, as in the following figure.
 $$
\inlinegraphic[scale=1.5]{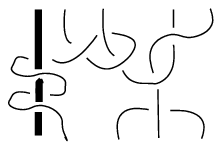}
$$
 The affine Kauffman tangle algebra is generated by the following affine tangle diagrams:
$$
X_1 = \inlinegraphic[scale= .7]{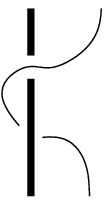}
\qquad
G_i =  \inlinegraphic[scale=.6]{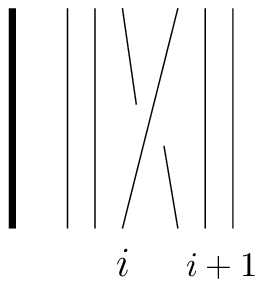}\qquad
E_i =  \inlinegraphic[scale= .7]{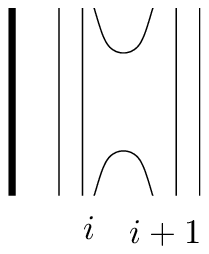} .
$$

 One can also define a  cyclotomic Kauffman tangle algebra 
 $\kt{n, S, r}(u_1, \dots, u_r)$ as the quotient of the affine Kauffman tangle algebra   by a cyclotomic skein relation,  which is a ``local" version of the cyclotomic relation of 
 Definition \ref{definition:  cyclotomic BMW} (9).   See
 ~\cite{GH2}  for the precise definition.     We denote the images of $X_1$, $E_i$ and $G_i$ in the cyclotomic Kauffman tangle algebra by the same letters.  
The assignments
 $e_i \mapsto E_i$,   $g_i \mapsto G_i$ and $y_1 \mapsto \rho X_1$  defines a surjective homorphism
 from $\varphi : \bmw {n, S, r}(u_1, \dots, u_r) \to \kt{n, S, r}(u_1, \dots, u_r)$,   see  ~\cite{GH2}, page 1114. 
 
 It is shown in ~\cite{GH2, GH3}  and in ~\cite{Wilcox-Yu3}  that
 the the map $\varphi$ is an isomorphism,  assuming admissibility conditions on the ground ring (see Section  \ref{subsubsection: admissibility}).  However, we are {\em not} going to assume this result here, but will give a new proof of the isomorphism.

\subsubsection{Brief history of cyclotomic BMW algebras}  Affine and cyclotomic BMW algebras were introduced by H\"aring--Oldenberg ~\cite{H-O2}   and have recently been studied by three groups of mathematicians:
Goodman and  Hauschild Mosley ~\cite{GH1, GH2, GH3,  goodman-2008},  Rui, Xu, and Si  ~\cite{rui-2008, rui-2008b},   and Wilcox and Yu  ~\cite{Wilcox-Yu, Wilcox-Yu2, Wilcox-Yu3, Yu-thesis}.  Under (slightly different) admissibility assumptions  on the ground ring (see Section \ref{subsubsection: admissibility})  all three groups have shown that the algebra
$\bmw{n, S, r}$  is free over $S$ of rank $r^n (2n-1) ! !$ and in fact is cellular.   (Wilcox and Yu produced cellular basis satisfying the strict equality $i(c_{s, t}^\la) = c_{t, s}^\la$, while the other groups only established cellularity in the weaker sense of Definition \ref{gl cell}.)   The cellular bases produced by all three groups are essentially tangle bases, i.e.,  cyclotomic analogues of the basis of Morton, Traczyk, and Wassermann for the ordinary BMW algebras.  Goodman \& Hauschild Mosley and Wilcox \& Yu have shown that the algebras can be realized as algebras of tangles, when the ground ring is admissible.   Rui et.\  al.\   have achieved additional representation theoretic results.   Further background on cyclotomic BMW algebras, motivation for the study of these algebras,  relations to other mathematical topics (quantum groups, knot theory), and further literature citations  can be found in ~\cite{GH2} and in the other papers cited above.

\subsubsection{Advantages of our approach to cellularity}
One of our motivations in undertaking the current work was to produce a Murphy type cellular basis for the cyclotomic BMW algebras, indexed by up--down tableaux.  As mentioned in the introduction, this has not been done previously, and it would be involved to extend Enyang's method for ordinary BMW algebras ~\cite{Enyang2} to the cellular case.   

 It turns out that our proof of cellularity is actually more direct than the previous proofs cited above, in that it bypasses the lengthy proof  (in ~\cite{GH2}, Proposition 3.7,  or ~\cite{Wilcox-Yu3}, Theorem 3.2)   that these algebras have a finite spanning set of the appropriate cardinality.  Our method does not  depend the isomorphism of the cyclotomic BMW algebras and cyclotomic Kauffman tangle algebras ~\cite{GH2, GH3} or ~\cite{Wilcox-Yu3};  in fact, we can give a new proof of this isomorphism.   
 
 One might say that the difficulty in our proof has been displaced, because instead of the finite spanning set result cited above, we require Mathas' recent theorem on coherence of cellular structures for cyclotomic Hecke algebras ~\cite{mathas-2009}.

\subsubsection{Admissibility conditions on the ground ring.}   \label{subsubsection: admissibility}

The cyclotomic BMW algebras can be defined over arbitrary ground rings.   However, it is necessary to impose conditions on the parameters in order to get a satisfactory theory.

One can see by a simple computation why one has to expect  conditions on the parameters.  First, one can show that there are elements $\delta_{-j}$  in the ground ring  $S$  for $j \ge 1$ such that
$e_1 y_1^{-j} e_1 = \delta_{-j} \,e_1$;  moreover,   $\delta_{-j}$  is  a  polynomial  in
$\rho^{-1}$,   $q - q\inv$,  and $\delta_0, \delta_1,  \dots, \delta_j$; see  \cite{GH3}, Lemma 2.5.
If one now multiplies the cyclotomic relation, Definition \ref{definition:  cyclotomic BMW} (9),  by $y_1^a$ and
pre--  and post--multiplies by $e_1$,  one gets  
$
(\sum_{k = 0}^r a_k  \delta_{k + a} )  e_1 = 0,
$
for $a \in \Z$,   where the $a_k$  are signed elementary symmetric polynomials in $u_1, \dots, u_r$.  Therefore,  either $e_1$ is a torsion element over $S$, or the following {\em weak admissibility conditions} hold:
$$
\sum_{k = 0}^r a_k  \delta_{k + a}  = 0, \quad \text{for $a \in \Z$}.
$$
If $S$ is a field and the weak admissibility conditions do not hold, then $e_1 = 0$;   it follows that all the $e_i$ are zero, and the algebra reduces to the cyclotomic Hecke algebra over $S$  with parameters $q^2$ and
$u_1,  \dots, u_r$.

The weak admissibility conditions  are  complicated and  not strong   enough to give satisfactory results on the representation theory of the algebras.  Therefore,  one wishes  to find conditions that are both simpler and stronger.  
Two apparently different conditions have been proposed, one by Wilcox and Yu ~\cite{Wilcox-Yu},  and another by Rui and Xu ~\cite{rui-2008}.  It has been shown in ~\cite{goodman-admissibility}  that the two conditions are equivalent in the case of greatest interest,  when $S$ is an integral domain with $q - q\inv \ne 0$. 
We consider only this case from now on.

\begin{definition}    Let $S$ be an integral  ground ring with
parameters $\rho$, $q$, $\delta_j$ ($j \ge 0$)  and $u_1, \dots, u_r$,  with $q - q\inv \ne 0$.
One says that $S$ is {\em admissible} (or that the parameters are {\em admissible})   if  $\{e_1, y_1 e_1, \dots, y_1^{r-1} e_1\} \subseteq  \bmw{2, S, r}$ is linearly independent over $S$.
\end{definition}

It is shown in ~\cite{Wilcox-Yu} that admissibility is equivalent to finitely many (explicit) polynomial relations on the parameters.  Moreover,  these relations give $\rho$  and $(q - q\inv) \delta_j$  as  Laurent polynomials
in the remaining parameters $q, u_1, \dots, u_r$;    see ~\cite{Wilcox-Yu}  and ~\cite{GH3}  for details.

\subsubsection{Morphisms of ground rings and a universal admissible ground ring} 
\label{subsubsection:  morphisms and generic ring}
We consider what are the appropriate morphisms between ground rings for cyclotomic BMW algebras.   The obvious notion would be that of a ring homomorphism taking parameters to parameters;  that is,  if $S$ is a ground ring with parameters $\rho$, $q$, 
etc.,  and $S'$ another ground  ring with parameters $\rho'$, $q'$, etc.,  then a morphism $\varphi : S \to S'$  would be required to map $\rho \mapsto \rho'$,  
$q \mapsto q'$,  etc.  

However,  it is better to require less, for the following reason:  The parameter $q$ enters into the cyclotomic BMW
 relations only in the expression $q\inv -q$, and the transformation $q \mapsto -q\inv$ leaves this expression invariant.  Moreover,  the transformation $g_i  \mapsto -g_i$,  $\rho \mapsto -\rho$,  $q \mapsto -q$ (with all other generators and parameters unchanged)   leaves the cyclotomic BMW relations unchanged.

Taking this into account, we arrive at the following notion:

\begin{definition}  \label{definition: parameter preserving}
Let $S$ be a ground ring with
parameters $\rho$, $q$, $\delta_j$   ($j \ge 0$),  and
$u_1, \dots, u_r$.
Let  $S'$ be another ground   ring with parameters $\rho'$, $q'$, etc.  

A unital ring homomorphism $\varphi : S \rightarrow S'$ is a {\em morphism of ground rings}  if it maps
$$
\begin{cases}
&\rho \mapsto \rho',  \text{ and}\\
& q \mapsto q' \text{ or } q \mapsto    -{q'}\inv,
\end{cases}
$$
or
$$
\begin{cases}
&\rho \mapsto  -\rho',  \text{ and}\\
& q \mapsto -q' \text{ or } q \mapsto    {q'}\inv,
\end{cases}
$$
and strictly preserves all other parameters.
\end{definition}

 Suppose there is a morphism of ground rings $\psi : S \rightarrow S'$.
  Then $\psi$ extends to a 
homomorphism  from 
$\bmw{n, S, r}$  to $\bmw{n, S', r}$.     Moreover,  $\bmw{n, S, r} \otimes_S S' \cong \bmw{n,S', r}$  as $S'$--algebras.   These statements are discussed in ~\cite{GH3},  Section 2.4.

Let $S$ be a ground ring with admissible   parameters  $\rho$, $q$, $\delta_j$   ($j \ge 0$),  and
$u_1, \dots, u_r$.  Then  
$$
\rho, -q\inv, \delta_j \  (j \ge 0),  \text{ and }  u_1, \dots, u_r
$$
and
$$
-\rho, -q, \delta_j \  (j \ge 0),  \text{ and }  u_1, \dots, u_r
$$
are also sets of admissible parameters.  
 Suppose that  $S$ is an integral ground ring with admissible  parameters, with $q - q\inv \ne 0$,   and that $S'$ is another integral ground ring; 
  if  $\varphi : S \rightarrow S'$ is 
a morphism of ground rings  such that
 $\varphi(q - q\inv) \ne 0$,  then $S'$ is also admissible.
 
 It is easy to show (see \cite{GH3}, Theorem 3.19)    that there is a universal integral admissible ground ring $R$, with parameters $\rhobold$,  $\qbold$,  $\deltabold_j$ ($j \ge 0$), and $\ubold_1, \dots, \ubold_r$,  with the following properties:
 \begin{enumerate}
 \item The parameters    $\qbold$, $\ubold_1$, \dots, $\ubold_r$ of  
$R$  are algebraically independent over $\Z$.
\item    $R$ is generated as a  ring by 
 $\qbold^{\pm 1}$, $\rhobold \powerpm$, $\deltabold_0$, $\deltabold_1$, \dots $\deltabold_{r-1} $, and $\ubold_1^{\pm 1}, \dots, \ubold_r^{\pm 1}$.
\item Whenever $S$ is an integral  ground ring with admissible
parameters,   with $q - q\inv \ne 0$,  there exists a  morphism of ground rings from $R$ to  $S$;  thus
$\bmw{n, S, r} \cong \bmw{n, R, r} \otimes_{R} S$. 
\item  The field of fractions of $R$  is $\Q(\qbold, \ubold_1, \dots, \ubold_r)$.
\item  Let $\bm p = \prod_{j = 1}^r \ubold_j$.  Then one has $\rhobold = \bm p$ if $r$ is even and
$\rhobold = \qbold\inv  \bm p$ if $r$ is odd.   Since $\rhobold\inv - \rhobold=   (\qbold\inv -\qbold) (\deltabold_0 - 1)$, and  $\qbold$, $\ubold_1$, \dots, $\ubold_r$  are algebraically independent,
one has $\deltabold_0 \ne 0$.  
 \end{enumerate}
 
 \subsubsection{Some properties of cyclotomic BMW and Kauffman tangle algebras.}  We restrict attention to the case of an integral admissible ground ring $S$ with $q - q\inv \ne 0$.     We write $\bmw n$  for $\bmw {n, S, r}(u_1, \dots, u_r)$ and $\kt n$ for  $\kt {n, S, r}(u_1, \dots, u_r)$.
 
 The cyclotomic BMW algebras have an algebra involution 
$i$ uniquely determined by $i(e_j) = e_j$ and $i(g_j) = g_j$ for all $j$, and $i(y_1) = y_1$.  Likewise, the
 cyclotomic Kauffman tangle algebras have an algebra involution $i$, whose  action on affine  tangle diagrams is by the rotation through the axis $y = 1/2$.   The surjective homomorphism $\varphi : \bmw n \to \kt n$  respects the involutions.  
 
 For $n \ge 0$, there is a  homomorphism (of involutive algebras)  $\iota$ from $\bmw n$  to $\bmw {n+1}$  determined by $e_i \mapsto e_i$  and  $g_i \mapsto g_i$  for $1 \le i \le n-1$, and $y_1 \mapsto y_1$;  it is not clear
 {\em a priori} that $\iota$ is injective.     
  
 Likewise, there is a homomorphism (of involutive algebras)  $\iota$  from $\kt n$ to $\kt {n+1}$.  
 On the level of affine tangle diagrams, the map is given by adding
a new vertical strand connecting $\p {n+1}$  and $\pbar {n+1}$,  as for the Brauer algebras.   This map
is injective, as we will now explain.  

 For $n \ge 1$, a map $\cl$ from affine  $(n, n)$--tangle diagrams to  affine  
 $(n-1, n-1)$--tangle diagrams can be defined  as for Brauer diagrams and ordinary tangle diagrams.  The linear extension of this map respects
regular isotopy and all the  skein relations defining the cyclotomic Kauffman tangle algebras, so determines a linear map from 
$\kt n$ to $\kt {n-1}$.   (See ~\cite{GH1}, Section 2.7, and ~\cite{GH2}, Section 3.3 for details.)  
 The map $\cl$ respects the involutions,  $i  \circ \cl = 
 \cl \circ i$.  Moreover,  for   $x \in \kt{n}$, we have
 $x  = {\rm cl} ( \iota(x) e_n)$,   so it follows that $\iota: \kt n \to \kt {n+1}$  is injective.
  Using $\iota$, we identify $\kt n$ as a subalgebra of $\kt {n+1}$.  

 If $\delta_0$ is invertible in $S$, we can define $\eps_n = (1/\delta_0) \cl$, which is a conditional expectation, that is, an unital $\kt {n-1}$--$\kt {n-1}$  bimodule map.  We have  $\eps_{n+1} \circ \iota(x) = x$ for $x \in \bmw n$.

\subsubsection{The cyclotomic Hecke algebra} 
We recall the definition of the affine and cyclotomic Hecke algebras,   see ~\cite{ariki-book}.

\begin{definition}
Let $S$ be a commutative unital ring with an invertible element $q$.  The {\em affine Hecke algebra} 
$\ahec{n,S}(q^2)$ 
over $S$
is the $S$--algebra with generators $t_1,  g_1,  \dots, g_{n-1}$, with relations:
\begin{enumerate}
\item  The generators $g_i$ are invertible,  satisfy the braid relations,  and   $g_i - g_i\inv =  (q - q\inv)$.
\item  The generator $t_1$ is invertible,  $t_1 g_1 t_1 g_1 = g_1 t_1 g_1  t_1$  and $t_1$ commutes with $g_j$  for $j \ge 2$.
\end{enumerate}
Let $u_1,  \dots, u_r$  be additional  elements in $S$.   The {\em cyclotomic Hecke algebra}
\break  $\hec{n, S, r}(q^2; u_1,  \dots, u_r)$  is the quotient of the affine Hecke algebra $\ahec{n, S}(q^2)$ by the polynomial relation    $(t_1 - u_1) \cdots (t_1 - u_r) = 0$.
\end{definition}

We remark that since the generator $t_1$ can be rescaled by an arbitrary invertible element of $S$, only the ratios of the parameters $u_i$ have invariant significance in the definition of the cyclotomic Hecke algebra.  The affine and cyclotomic Hecke algebras have unique algebra involutions determined by
$g_i \to g_i$  and $t_1 \to t_1$.  

Now let $S$  be a ground  ring with parameters $\rho$, $q$,  $\delta_j$,  and $u_1, \dots, u_r$.    For each $n$,  let $I_n$  be the two sided ideal in $\bmw{n, S, r}$   generated by $e_{n-1}$.   
Because of the relations  $e_j e_{j\pm1} e_j = e_j$,   the ideal $I_n$  is generated by any $e_i$  ($1 \le i \le n-1$)  or by all of them.    It is easy to check that the quotient of $\bmw{n, S, r}$  by $I_n$  is isomorphic (as involutive algebras) to the cyclotomic Hecke algebra  $\hec{n, S, r}(q^2; u_1, \dots, u_r)$.

Let $\labold = (\la^{(1)}, \dots, \la^{(r)})$ be an $r$--tuple of Young diagrams.  The total size of $\labold$ is
$|\labold| = \sum_i |\la^{(i)}|$.    If $\mubold$ and $\labold$ are  $r$--tuples of Young diagrams of total size $f-1$  and $f$ respectively,  we write $\mubold \subset \labold$ if $\mubold$ is obtained from $\labold$ by removing one box from one component of $\labold$.

\begin{theorem}[\cite{ariki-book}]   \label{theorem:  cyclotomic hecke split semisimple}
Let $F$ be a field.   The cyclotomic Hecke algebra  $\hec{n, F, r}(q^2; u_1, \dots, u_r)$  is split semisimple for all $n$  as long  as $q$ is not a proper root of unity and, for all
$i \ne j$,
$u_i/u_j$  is not an integer power of $q$  .  In this case,   the simple components of \break $\hec{n, F, r}(q; u_1, \dots, u_r)$   are labeled by $r$--tuples of Young diagrams of total size $n$,   and a simple $\hec{n, F, r}$  module $V_\labold$ decomposes as a
$\hec{n-1, F, r}$  module as the direct sum of all $V_\mubold$  with $\mubold \subset \labold$.
\end{theorem}

Let us call the branching diagram for the cyclotomic Hecke algebras, as described in the theorem,  the
{\em $r$--Young lattice}.   Note that, as for the usual Young's lattice,  the $r$--Young lattice has no multiple edges.

 \begin{theorem}[Ariki, Koike, Dipper, James, Mathas]  \label{propositon:  cyclotomic Hecke coherent tower}
 The sequence of cyclotomic Hecke algebras  $(\hec{n, S, r}(q^2; u_1, \dots, u_r))_{n \ge 0}$ is a coherent tower  of cellular algebras.
 \end{theorem}
 
 \begin{proof}  Write $\hec n$ for $ \hec{n, S, r}(q^2; u_1, \dots, u_r)$.   Ariki and Koike showed that the cyclotomic Hecke algebras are free as $S$ modules  ~\cite{ariki-koike},  which implies that $\hec n$ imbeds naturally in 
 $\hec {n+1}$.   Moreover,  the algebras $\hec n$ have involutions that are consistent with the inclusions. Dipper, James and Mathas ~\cite{dipper-james-mathas} constructed a cellular basis of the cyclotomic Hecke algebras,  generalizing the Murphy basis of ordinary Hecke algebras.  Ariki and Mathas
 showed ~\cite{ariki-mathas}, Proposition 1.9,  that restrictions of cell modules from $\hec {n+1}$ to $\hec n$ have cell filtrations.  Finally,  Mathas has shown ~\cite{mathas-2009}  that the module obtained from inducing a cell module from
 $\hec n$  to $\hec {n+1}$  has a cell filtration.
 \end{proof}

 \subsubsection{Verification of the framework axioms for the cyclotomic BMW algebras}
Let $R$ be the generic admissible integral ground ring, with parameters 
$\rhobold$,  $\qbold$,  $\deltabold_j$ ($j \ge 0$), and $\ubold_1, \dots, \ubold_r$, as  introduced at the end of Section \ref{subsubsection:  morphisms and generic ring}.  In this section, 
we write  $\bmw n$  for $\bmw{n, R, r}(\ubold_1, \dots,  \ubold_r)$,
 $\kt n$  for $\kt{n, R, r}(\ubold_1, \dots,  \ubold_r)$,
and $\hec n$  for
$\hec{n, R, r}(\qbold^2; \ubold_1, \dots, \ubold_r)$.  Recall that the field of fractions of $R$ is
$F = \Q(\qbold, \ubold_1, \dots, \ubold_r)$.      Let $\bmw n^F = \bmw n \otimes_R F$, and similarly for the other algebras.  

If we would assume the isomorphism of $\bmw n$  and $\kt n$,  then we could verify the framework axioms for the pair of sequences $(\bmw n)_{n \ge 0}$ and $(\hec n)_{n \ge 0}$  without difficulty,
using elementary observations and some deeper results from the literature, and consequently apply Theorem \ref{main theorem} to the cyclotomic BMW algebras.  However, we wish to give
an independent proof of the isomorphism.   Consequently, we have to verify the framework axioms and prove the isomorphism $\bmw n \cong \kt n$ inductively,  in tandem with the inductive step in the proof of Theorem \ref{main theorem}.   

\begin{lemma} \label{lemma: Axiom on A0 for cyclotomic BMW}
 $\bmw 0 \cong \kt 0 \cong R$.
\end{lemma}

\begin{proof}  Wilcox and Yu ~\cite{Wilcox-Yu3}, Proposition 6.2,  show that $\kt 0$ is a free $R$ module with basis $\{\emptyset\}$,  where
$\emptyset$ denotes the empty affine tangle diagram, 
which is also  the identity element of $\kt 0$.
\end{proof}

\begin{lemma} \label{lemma:  isomorphism varphi implies injectivity of iota}
 If for some $n$  and for some admissible ground ring $S$, we have
$\varphi : \bmw n^S \to \kt n^S$ is an isomorphism,  then $\iota: \bmw n^S \to \bmw {n+1}^S$ is injective.
\end{lemma}

\begin{proof}  $\varphi \circ \iota = \iota \circ \varphi : \bmw n^S \to
\kt {n+1}^S$ is injective, because $\varphi: \bmw n^S \to \kt n^S$  and $\iota: \kt n^S \to \kt {n+1}^S$ are injective. Thus
 $\iota : \bmw n^S \to \bmw {n+1}^S$ is injective.
\end{proof}

\begin{lemma} \label{lemma: generic semisimplicty of cycotomic BMW} For all $n \ge 0$,   $\bmw n^F  \cong \kt n^F$,  $\bmw n^F$ is split semisimple of dimension \break   $r^n (2n - 1)!!$, and $\iota : \bmw n^F \to \bmw {n+1}^F$ is injective.  
\end{lemma}

\begin{proof}  This is proved in ~\cite{GH3},  Theorem 4.8.  We stress that the result is independent of the finite spanning set theorem, ~\cite{GH2}, Proposition 3.7.   One thing that is not made clear in the  proof of 
~\cite{GH3},  Theorem 4.8 is why $\iota : \bmw n^F \to \bmw {n+1}^F$ is injective.   But if one assumes
inductively that the conclusions of the theorem hold for $\bmw f^F, f \le n$, for some fixed $n$, and in particular that
$\varphi : \bmw n^F \to \kt n^F$ is an isomorphism,  then   $\iota : \bmw n^F \to \bmw  {n+1}^F$ is injective by Lemma \ref{lemma:  isomorphism varphi implies injectivity of iota}.
One can then continue with the proof of the inductive step of ~\cite{GH3},  Theorem 4.8. 
\end{proof}  

\begin{lemma}  \label{lemma: cardinality of basis of bmw n}
If for some $n$,  $\bmw n$ is a free $R$--module, then its rank is $r^n (2n-1)!!$.
\end{lemma}

\begin{proof}   $x \mapsto x \otimes 1$  takes an $R$--basis of $\bmw n$ to an $F$--basis of
$\bmw n \otimes_R F = \bmw n^F$.  
\end{proof}

\begin{lemma} \label{lemma:  finite spanning set implies isomorphism and freeness}
 If for some $n$,  $\bmw n$  has a spanning set  $A$ of cardinality $r^n (2n -1)!!$, 
then $\varphi : \bmw n \to \kt n$ is an isomorphism,  and $A$ is an $R$--basis of $\bmw n$.   
\end{lemma}

\begin{proof}  Say $\bmw n$ has a spanning set $A$ of cardinality  $r^n (2n -1)!!$.  To prove both conclusions, it suffices to show that $\varphi(A)$ is linearly independent in $\kt n$.    But  $$ \{\varphi(a) \otimes 1: a \in A\} \subseteq   \kt  n \otimes_R F = \kt n^F$$ is a spanning set of cardinality $r^n (2n -1)!!$,  which is  the dimension of  $\kt n^F$, according to Lemma \ref{lemma: generic semisimplicty of cycotomic BMW}.   Therefore $\{\varphi(a) \otimes 1: a \in A\} $ is linearly independent in $\kt n^F$,  and hence $\varphi(A)$ is linearly independent in $\kt n$.  
\end{proof}

\begin{lemma}  \label{lemma:  W1 isomorphic to KT1 and to H1}
$\bmw 1 \cong \kt 1 \cong \hec 1$,  $\bmw 1$ is a free $R$--module of rank $r$, and both
$\iota : \bmw 0 \to \bmw 1$  and  $\iota : \bmw 1 \to \bmw 2$ are injective.  
\end{lemma}

\begin{proof}  By definition, $\bmw 1 \cong \hec 1 \cong R[X]/((X-u_1)\cdots (X-u_r))$,  and these algebras are free $R$--modules of rank $r$. Hence $\varphi : \bmw 1 \to \kt 1$ is an isomorphism by
Lemma \ref{lemma:  finite spanning set implies isomorphism and freeness}.    The injectivity statements follow from Lemma \ref{lemma:  isomorphism varphi implies injectivity of iota}.

 \end{proof}

\begin{lemma} \label{lemma:  axioms 2 6 7 for cyclotomic BMW} 
Suppose that for some $n \ge 1$ one has $\bmw k \cong \kt k$ for $0 \le k \le n$.  Then
the maps $\iota: \bmw {k} \to \bmw {k+1}$  are injective  for $0 \le k \le n$.  Using the maps $\iota$, regard $\bmw k$ as a subalgebra of $\bmw {k+1}$  for  $0 \le k \le n$.     One has:
 \begin{enumerate}
 \item  $  \deltabold_0 R \,e_1 \subseteq
 e_1 \bmw 1 e_1  \subseteq   R \,e_1$.

 \item 
For $2 \le k \le n$, $e_{k} \bmw {k} e_{k} = \bmw {k-1} e_{k}$.

\item  For $1 \le k \le n$,  
$e_{k}$  commutes with $ \bmw {k-1} $.  
\item  For $1 \le k \le n$, 
$\bmw {k+1}  \,e_{k} =  \bmw {k} \,e_{k}$.   Moreover,  
 $x \mapsto  x e_{k}$  is injective from $\bmw {k}$ to $\bmw {k} e_{k}$.
\end{enumerate}
 \end{lemma}
 
 \begin{proof}    The statement about injectivity of the maps $\iota$ follows from  Lemma \ref{lemma:  isomorphism varphi implies injectivity of iota}.  
 
 Point (1) follows from the relations $e_1 y_1^j e_1 = \deltabold_j e_1$  for $j \ge 0$. 
 Point (2) and the first part of point (4)  follows from the corresponding facts for the affine BMW algebras,  ~\cite{GH1}, Proposition 3.17, and Proposition 3.20.    
Point (3) follows from the defining relations for the cyclotomic BMW algebras.  For the injectivity statement
in point (4),  note that for $x \in \bmw k$, 
$$
\cl(\varphi(x e_k)) = \cl(\varphi(x) E_k) = \varphi(x).
$$  
Since $\varphi : \bmw k \to \kt k$ is injective, so is $x \mapsto x e_k$.  

\end{proof}

\vbox{
\begin{theorem}\label{proposition: framework axioms for CBMW}  \mbox{}
\begin{enumerate}
\item
 The two sequences of algebras $(\bmw k)_{k \ge 0}$  and $(H_k)_{k \ge 0}$ satisfy the framework axioms of Section \ref{subsection: framework axioms}.
 \item  For all $k \ge 0$,  $\varphi : \bmw k \to \kt k$ is an isomorphism, and $\iota: \bmw k \to \bmw {k+1}$ is injective.
 \item The conclusions of Theorem \ref{main theorem} are valid for the sequence $(\bmw k)_{k \ge 0}$. 
 \end{enumerate}
\end{theorem}
}

\begin{proof}   According to Proposition  \ref{propositon:  cyclotomic Hecke coherent tower},  $(H_k)_{k \ge 0}$ is a coherent tower of cellular algebras, so axiom (\ref{axiom Hn coherent}) of the framework axioms holds.   Axiom (\ref{axiom: A0 and A1}) holds by  Lemmas \ref{lemma: Axiom on A0 for cyclotomic BMW} and \ref{lemma:  W1 isomorphic to KT1 and to H1}.   
Axiom (\ref{axiom: semisimplicity})  holds by Lemma \ref{lemma: generic semisimplicty of cycotomic BMW}.   We observed above that $\bmw k/\bmw k e_{k-1} \bmw k \cong H_k$ as involutive algebras;   thus axiom (\ref{axiom:  idempotent and Hn as quotient of An}) holds.  
Axiom (\ref{axiom: e(n-1) in An en An}) holds  because of the relation $e_{k-1} e_k e_{k-1} = e_{k-1}$.

Suppose that for some $n \ge 0$, it is known that the maps $\varphi : \bmw k \to \kt k$ are isomorphisms for $0 \le k \le n$.  Then, from Lemma \ref{lemma:  axioms 2 6 7 for cyclotomic BMW} ,  we have the following versions of framework axioms (\ref{axiom: involution on An}), (\ref{axiom: en An en}) and (\ref{axiom:  An en}):
\par\noindent (2$'$) \ \  $\bmw k$ is an $i$--invariant subalgebra of $\bmw {k+1}$  for $0 \le k \le n$.  
\par\noindent (6$'$) \ \ For $1 \le k \le n$,   $e_{k}$ commutes with $\bmw {k-1}$ and $e_{k} \bmw {k} e_{k}  \subseteq \bmw {k-1} e_{k}$.
\par\noindent (7$'$) \ \ For $ 1 \le k \le n$,  $\bmw {k+1} 	e_{k} = \bmw {k} e_{k}$,  and the map $x \mapsto x e_{k}$ is injective from
$\bmw {k}$ to $\bmw {k} e_{k}$.

Now we consider the following:
\par\smallskip
\noindent {\bf Claim:}\quad
  For all $n \ge 0$,  \par
\noindent (a) \ \
for $0\le k \le n$,  the maps $\varphi : \bmw k \to \kt k$ are isomorphisms,
and therefore $\bmw k$ may be regarded as an $i$--invariant subalgebra of $\bmw {k+1}$, and
\par\noindent (b) \ \ the statements  (1) --(4)  of Theorem \ref{main theorem}  hold for the finite tower $(\bmw k)_{0 \le k \le n}$.
\smallskip

For $n = 0$ and $n = 1$,  the claim follows from Lemmas \ref{lemma: Axiom on A0 for cyclotomic BMW} and \ref{lemma:  W1 isomorphic to KT1 and to H1}.    We assume the claim holds for some $n \ge 1$ and show that it also holds for $n + 1$.  
Then, by the discussion above,  the framework axioms hold for the finite tower
$(\bmw k)_{0\le k \le n}$  with axioms (\ref{axiom: involution on An}), (\ref{axiom: en An en}) and (\ref{axiom:  An en}) replaced by the finite versions (2$'$), (6$'$), and (7$'$).
Now  the  inductive step in the proof of Theorem \ref{main theorem} goes through without change and
yields part (b) of the claim for the tower $(\bmw k)_{0 \le k \le n+1}$.    In particular, 
$\bmw {n+1}$ is a cellular algebra;  the cardinality of its cellular basis is   $r^{n+1} (2n + 1)!!$,
by Lemma \ref{lemma: cardinality of basis of bmw n}.      But then
Lemma \ref{lemma:  finite spanning set implies isomorphism and freeness} gives that
$\varphi: \bmw {n+1} \to \kt {n+1}$ is an isomorphism, so part (a) of the claim also holds for $n + 1$.  
\end{proof}

\begin{corollary}  Let $S$ be any admissible integral ground ring with $q - q\inv \ne 0$.  
\begin{enumerate}
\item
The sequence of  cyclotomic BMW algebras $(\bmw {n, S, r})_{n \ge 0}$ is a coherent tower of cellular algebras.
$\bmw {n, S, r}$  has cell modules indexed by all  $r$--tuples of Young diagrams of total size $n$, $n-2$, $n-4, \dots$.   The cell module labeled by
an $r$--tuple of  Young diagrams  $\labold$ has a basis labeled by up--down tableaux of length $n$ and shape $\labold$.
\item  $\bmw {n, S, r} \cong  \kt {n, S, r}$  for all $n \ge 0$.  
\end{enumerate}
\end{corollary}

\begin{remark}  It is possible to combine our results with the  results of  Wilcox and Yu
~\cite{Wilcox-Yu2}  to obtain
Murphy type bases of the cyclotomic BMW algebras that are strictly cellular, i.e.  $i(c_{s, t}^\la) = c_{t, s}^\la $ for all $\la, s, t$.      To do this, all we need, according to Remark \ref{remark: conditions for cell net to give strict cellular basis}, is an $i$--invariant $R$--module complement to  the ideal
$\breve {\bmw n}^{(\labold, n)}$  in $ {\bmw n}^{(\labold, n)}$.  However, one can check that the
ideals  $\breve {\bmw n}^{(\labold, n)}$  and $ {\bmw n}^{(\labold, n)}$ for our cellular structure are the same as for the cellular structure of  Wilcox and Yu, and therefore, since their cellular basis satisfies the strict equality $i(c_{s, t}^\la) = c_{t, s}^\la $ for all $\la, s, t$, the desired
$i$--invariant $R$--module complement exists. 
\end{remark}

 \begin{remark}  Our framework also applies to 
  the degenerate
cyclotomic BMW algebras (cyclotomic Nazarov Wenzl algebras)  studied in   ~\cite{ariki-mathas-rui}.   For the details, see ~\cite{GG2}.
 \end{remark}

\subsection{The walled Brauer algebras} \label{subsection: walled Brauer algebras}
\subsubsection{Definition of the walled Brauer algebras} 
Let $S$ be a commutative ring with identity, with a distinguished element $\delta$.
The walled (or rational)  Brauer algebra $\br_{r, s}(S, \delta)$
 is a unital subalgebra of the Brauer algebra $\br_{r+s}(S, \delta)$ spanned by certain Brauer diagrams.
 Divide the $r+s$ top vertices  into a left cluster consisting of the leftmost $r$ vertices and a right cluster consisting of the remaining $s$ vertices,  and similarly for the bottom vertices.      The walled Brauer diagrams are those in which no vertical strand connects a left vertex and a right vertex,  and every horizontal strand connects a left vertex and a right vertex.   (If we draw a vertical line--the wall--separating left and right vertices,  then vertical strands are forbidden to cross the wall, and horizontal strands are required to cross the wall.)  One can easily check that the span of walled Brauer diagrams is a unital subalgebra of $\br_{r+s}(S, \delta)$.
 
 \subsubsection{Brief history of the walled Brauer algebras}   The walled  Brauer algebras were introduced by Turaev ~\cite{turaev-operator-invariants} and  by Koike ~\cite{koike-tensor-products},  and studied by Benkart et.\ al.\  ~\cite{benkart-walled-brauer}  and by Nikitin ~\cite{nikitin-walled-brauer}.   
 The walled Brauer algebras arise in connection with the invariant theory of the general linear group acting on mixed tensors.  
 Cellularity of walled Brauer algebras was proved by Green and Martin ~\cite{green-martin-tabular} and by Cox et.\ al.\  ~\cite{cox-walled-brauer};  the latter authors show that
 walled Brauer algebras can be arranged into coherent cellular  towers.
 
 \subsubsection{Some properties of the walled Brauer algebras}  The walled Brauer algebra $\br_{r, s}$ is invariant under the involution $i$ of the Brauer algebra $\br_{r+s}$.  Moreover,  the inclusion map
 $\iota : \br_{r+s} \to \br_{r+s+1}$  maps $\br_{r, s}$ to $\br_{r, s+1}$, and the closure map
 $\cl : \br_{r+s} \to \br_{r + s -1}$  maps $\br_{r, s}$ to $\br_{r, s-1}$,  when $s \ge 1$.  If $\delta$ is invertible, $\eps_{r, s} = (1/\delta)\, \cl : \br_{r, s} \to \br_{r, s-1}$ is a conditional expectation, and, of course, the trace $\eps$ on $\br_{r+s}$  restricts to a trace on $\br_{r, s}$. 
 
 The Brauer algebras have an involutive inner automorphism $\rho$  which maps each Brauer diagram to its reflection in the vertical line  $x = 1/2$. (We might as well take the vertical line to coincide with our wall.)
 It is clear that $\rho$ restricts to an isomorphism from $\br_{r, s}$ to $\br_{s, r}$.  Given this, we can define ``left versions" of $\iota$, $\cl$ and $\eps_{r, s}$ by
 $\iota' = \rho \circ \iota \circ  \rho : \br_{r, s} \to \br_{r+1, s}$,   $\cl' =  \rho \circ \cl \circ \rho : \br_{r, s} \to \br_{r-1, s}$, and $\eps' =  \rho \circ \eps \circ \rho : \br_{r, s} \to \br_{r-1, s}$.  Note that
 $\iota'$  adds a vertical strand on the left, and $\cl'$  partially closes diagrams on the left.
 
 Let $e_{a, b}$  be the Brauer diagram with horizontal strands connecting $\p a$ to $\p {b}$  and
 $\pbar a$ to $\pbar {b}$  and vertical strands connecting $\p j$ to $\pbar j$  for all $j \ne a, b$.
 One can easily check the following properties:   
 \begin{lemma}  \mbox{} \label{lemma: properties of e(r, s) in walled brauer algebra}
 \begin{enumerate}
 \item  $e_{a, b}^2 = \delta e_{a, b}$.
 \item  $e_{a, b} \,e_{a, b\pm 1}\, e_{a, b} = e_{a, b}$
 and $e_{a, b} \,e_{a\pm 1, b}\, e_{a, b} = e_{a, b}$. 
  \item For $e_{a, b} \in \br_{r, s}$,  $\iota(e_{a, b}) = e_{a, b}$  and $\iota'(e_{a, b}) = e_{a+1, b+1}$.
  \item For $x \in \br_{r, s+1}$, we have $e_{1, r+ s+2} \,\iota'(x)\, e_{1, r+ s+2}  =  \iota'\circ \iota\circ\cl(x)\,  e_{1, r+ s+2}$.
 \item   For $x \in \br_{r+1, s}$, we have $e_{1, r+ s+2} \,\iota(x)\,e_{1, r+ s+2}=  \iota'\circ \iota\circ \cl'(x) \,e_{1, r+ s+2} $.
 \item  $e_{1, r+s + 2}$ commutes with  $\iota'\circ \iota(x)$ for all $x  \in \br_{r, s}$.
 \end{enumerate}
 \end{lemma}
 
 The following statement is also easy to check:
 
 \begin{lemma}  \label{lemma:  quotient of A_n by J for walled brauer}
The ideal $J$  in $\br_{r, s}(S, \delta)$  generated by $e_{1, r+s}$ is the ideal spanned by diagrams with fewer than $r + s$ through strands,  and $\br_{r, s}(S, \delta)/J \cong S (\S_r \times \S_s)$.
  \end{lemma}

 \begin{lemma} \mbox{} \label{lemma:  An e(n-1) equals A(n-1) e(n-1)  for walled Brauer}
 \begin{enumerate}
 \item  $\br_{r, s+1}\, e_{1, r+s+1} = \iota(\br_{r, s})  \,e_{1, r+s+1}$.
 \item  $\br_{r+1, s} \,e_{1, r+s+1} = \iota'(\br_{r, s})\, e_{1, r+s+1}$.
 \end{enumerate}
 \end{lemma}
 
 \begin{proof}  To prove  part (1), we
  have to show that if $d$ is a diagram in  $\br_{r, s+1}$,  then there is a diagram
 $d' \in \iota(\br_{r, s})$  such that $d  \,e_{1, r+s+1} = d'  \,e_{1, r+s+1}$.  We can suppose that $d$ is not already in $ \iota(\br_{r, s})$;  therefore,  the vertex $ \pbar {r+s + 1}$ in $d$  is connected to some vertex $v$ other than $\p 1$ and $\p {r+s + 1}$.  There are two cases to consider.

 The first is that the vertices $\p 1$ and $\p {r+s + 1}$ are not connected to each other in $d$;    let $a$ and $b$ be the vertices connected to
 $\p 1$ and $\p {r+s + 1}$.  Now let $d'$ be the diagram in which $a$ and $b$ are connected to each other;  $\p {r + s + 1}$ is connected to $ \pbar {r+s + 1}$; $\p 1$ is connected to $v$; and all other strands are as in $d$.  Then we have $d  \,e_{1, r+s+1} = d'  \,e_{1, r+s+1}$.
 The case that the vertices $\p 1$ and $\p {r+s + 1}$ are  connected to each other is similar and will be omitted.
 
 Part (2) is proved  by applying the map $\rho$ to both sides of the equality in part (1) and then
 interchanging the roles of $r$ and $s$.
 \end{proof}

 A unital trace $\eps$ on an $S$--algebra $A$ is {\em non-degenerate} if for every non--zero $x \in A$  there exists a $y \in A$ such that $\eps(xy) \ne 0$.
 
 \begin{lemma}  \mbox{}  \label{lemma: non-degeneracy of trace on Brauer and walled Brauer}
 \begin{enumerate}
 \item
 The trace $\eps$ on $\br_n(\Q(\deltabold), \deltabold)$ is non-degenerate, for any $n$.
 \item  The trace $\eps$ on      $\br_{r, s}(\Q(\deltabold), \deltabold)$ is non-degenerate, for any $r, s$.
 \end{enumerate}
 \end{lemma}
 
 \noindent {\em Sketch of proof.}  The argument for part (1)  is from ~\cite{Morton-Traczyk}.  It suffices to show that
 the determinant of the Gram matrix $\eps(d d')_{d, d'}$, where $d, d'$ run over the list of all Brauer diagrams (in some order),  is non-zero.   Recall that $\eps(d d')$  is  $\qbold^{c(d d')- n}$,  where
 $c(d d')$ is the number of components in the tangle obtained by closing all the strands of $d d'$. 
 One can check that $c(d \,i(d) ) = n$ and $c(d d') < n$ for all diagrams other than $i(d)$.  Therefore, each
 row and column of the Gram matrix has exactly one entry equal to $1$ and all other entries have the form 
 $\qbold^{-k}$ for some $k >0$.  
 
 The argument for part (2) is identical.
 \qed
 
 \medskip

 \subsubsection{Verification of the framework axioms for the walled Brauer algebras}
 
To fit the walled Brauer algebras to our framework, we have to reduce the double sequence of algebras to a single sequence.  We adopt the following scheme, as in ~\cite{nikitin-walled-brauer}, or
~\cite{cox-walled-brauer}:  
  Fix some integer  $t \ge 0$. For any $S$ and $\delta \in S$,  we consider the sequence of walled Brauer algebras
 $A_n = A_n(S, \delta)$,  where
 $A_{2 k}(S, \delta)  = \br_{k,  k+ t}(S, \delta)$,  and $A_{2 k + 1}(S, \delta) =  \br_{k, k+t + 1}(S, \delta)$,  with the inclusions
 $$
 A_{2 k} \stackrel{\iota}{\longrightarrow} A_{2k + 1} \stackrel{\iota'}{\longrightarrow}  A_{2k + 2}.
 $$
We put $f_{2k-1} = e_{1, 2k+t} \in A_{2k}$  and $f_{2k} = e_{1, 2k+ t + 1} \in A_{2k+1}$.  
  We identify $A_n$ as a subalgebra of $A_{n+1}$ via these embeddings.
 With these conventions,  Lemma \ref{lemma: properties of e(r, s) in walled brauer algebra}, points (2) and (3)  give $f_n f_{n\pm 1} f_n = f_n$.
 Moreover, if we write $\cl_n = \cl$  when $n$ is even and $\cl_n = \cl'$ when $n$ is odd, then
 we have  $f_{n-1} x f_{n-1} = \cl_{n-1}(x) f_{n-1} $ for $x \in A_{n-1}$, by Lemma \ref{lemma: properties of e(r, s) in walled brauer algebra}, points (4) and (5).    Point (6) of the Lemma says that $f_{n-1}$ commutes with
 $A_{n-2}$.

 If $J$ is the ideal in $A_n$ generated by
 $f_{n-1}$, then we have $A_{2k}/J \cong S(\S_k \times \S_{k + t})$,  and $A_{2k + 1}/J \cong S(\S_k \times \S_{k + t + 1})$.  So we set $Q_{2k}(S) =   S(\S_k \times \S_{k + t})$ and $Q_{2k + 1}(S) =  S(\S_k \times \S_{k + t + 1})$,  with the natural embeddings.
 
 Since $A_0 = \br_{0, t} \cong S \S_t$,  and $A_1 = \br_{0, t+1} \cong S \S_{t + 1}$, we cannot hope to satisfy our framework axiom (\ref{axiom: A0 and A1}).  However, we can replace axiom (\ref{axiom: A0 and A1}) with the weaker
 
 \medskip
 \noindent(3$'$) \quad   $A_0 \cong Q_0$, and $A_1 \cong Q_1$.
  \medskip
  
\noindent We  also  have to drop our usual convention (see Definition \ref{def of branching}) regarding branching diagrams that the 
$0$--th row of the branching diagram has a single vertex.     Our conclusions will have to be modified, but not severely.
  
  \medskip
  We now take $R = \Z[\deltabold]$ and $\delta = \deltabold$.    $R$ is the generic ground ring for
  walled Brauer algebras;  if $S$ is any commutative unital ring with parameter $\delta$,  then
  $\br_{r, s}(S, q) = \br_{r, s}(R, \qbold) \otimes_R S$.  Let $F = \Q(\deltabold)$.
 In the remainder of this section, we write
 $A_n = A_n(R, \deltabold)$  and $Q_n = Q_n(R)$.  (Recall that $Q_n(R) = R(\S_k \times \S_{k + t})$ if $n = 2k$,  and
$ Q_n(R)  =    R(\S_k \times \S_{k + t + 1})$ if $n = 2 k + 1$.)

 \begin{lemma} \label{lemma: generic semisimplicity for walled Brauer}
 The walled Brauer algebra $\br_{r, s}(\Q(\deltabold), \deltabold)$  is split semisimple.
 \end{lemma}
 
 \noindent{\em Sketch of proof.}   It suffices to show that (for any $t$) the algebras in the sequence $A_n$ are split semisimple.  This was proved by Nikitin in ~\cite{nikitin-walled-brauer}, following Wenzl's method for the Brauer algebra in ~\cite{Wenzl-Brauer}. 
 Nikitin's proof involves obtaining the weights of the trace  $\eps$, but little detail is given.  For our purposes, we can bypass this issue, and use Lemma \ref{lemma: non-degeneracy of trace on Brauer and walled Brauer} instead.  Then the method of proof of Theorem 3.2 from ~\cite{Wenzl-Brauer} applies.
 \qed
 \medskip 
 
 \begin{proposition} \label{proposition: framework axioms for walledBrauer} The two sequence of $R$--algebras  $(A_n)_{n \ge 0}$ and $(Q_n)_{n \ge 0}$  satisfy the framework axioms of  Section \ref{subsection: framework axioms},  with axiom (3) replaced by (3\,$'$), specified above,  and with the elements
 $f_n$ taking the role of the elements $e_n$ in the list of framework axioms.
\end{proposition}

\begin{proof}   The sequence $(Q_n)_{n \ge 0}$ is clearly a coherent tower of cellular algebras, so axiom (\ref{axiom Hn coherent}) holds.  Axiom (\ref{axiom: involution on An}) is evident, and we have remarked about substituting axiom (\ref{axiom: A0 and A1}$'$) for axiom (\ref{axiom: A0 and A1}).   
$A_n^F$ is split semisimple by Lemma 
\ref{lemma: generic semisimplicity for walled Brauer}.   Thus axiom (\ref{axiom: semisimplicity}) holds.

We have $f_{n-1}$ is an essential idempotent with $i(f_{n-1}) = f_{n-1}$.  We have   \break $A_n/ (A_n f_{n-1} A_n) \cong Q_n$ by Lemma  \ref{lemma:  quotient of A_n by J for walled brauer}, which gives axiom 
(\ref{axiom:  idempotent and Hn as quotient of An}).

We have seen that $f_{n-1}$ commutes with $A_{n-2}$  and $f_{n-1} A_{n-1} f_{n-1} \subseteq A_{n-2} f_{n-1}$.  Moreover, if $x \in A_{n-2}$,  then  $f_{n-1} x f_{n-1} = \deltabold x  f_{n-1}$, so
$f_{n-1} A_{n-1} f_{n-1} \supseteq  \deltabold A_{n-2} f_{n-1}$.  Therefore, 
$f_{n-1} A_{n-1}^F f_{n-1} = A_{n-2}^F f_{n-1}$, so axiom (\ref{axiom: en An en}) holds.

Axiom (\ref{axiom:  An en})  results from Lemma \ref{lemma:  An e(n-1) equals A(n-1) e(n-1)  for walled Brauer},  and
axiom (\ref{axiom: e(n-1) in An en An})  from $f_{n-1} f_n f_{n-1} = f_{n-1}$.
\end{proof}

\begin{remark} The branching diagram  for the sequence
$(Q_n^F)$ is the following:  Each row has vertices labeled by pairs of Young diagrams; on an even row $2k$,  the  the first  Young diagram in a pair has  $k$ boxes and the second  $k+t$ boxes;  on an odd row $2k+1$,   the first Young diagram has $k$ boxes and the second $k+t+1$ boxes; finally,  there is an edge between pairs of Young diagrams in successive rows that differ by exactly one box.
\end{remark}

\begin{corollary}  \mbox{}    Let $S$ be any commutative unital ring with parameter $\delta$.
\begin{enumerate}
\item The walled Brauer algebras $\br_{r, s}(S, \delta)$ are cellular algebras. 
\item  The family is coherent in the sense that the restriction of a cell module from $\br_{r, s}(S, \delta)$ to $\br_{r-1, s}(S, \delta)$ or to $\br_{r, s-1}(S, \delta)$  and induction of a cell module from $\br_{r, s}(S, \delta)$ to
$\br_{r+1, s}(S, \delta)$ or to $\br_{r, s+1}(S, \delta)$ have filtrations by cell modules.  
\item The cell modules
of  $\br_{r, s}(S, \delta)$  are labeled by pairs of Young diagrams $(\la^{(1)}, \la^{(2)})$, where
$|\la^{(2)}| - |\la^{(1)}| = s - r$ and $|\la^{(2)}| + |\la^{(1)}| \le s + r$.
\end{enumerate}
\end{corollary}

A basis for any cell module for $\br_{r, s}$ can be labeled by paths on a certain branching diagram.
Suppose without loss of generality that $t = s - r \ge 0$.   Let $(A_n)_{n \ge 0}$ and $(Q_n)_{n \ge 0}$ be the two sequences of algebras defined above, depending on $t$,  so in particular,  $\br_{r, s} = A_{2 r}$.
Let $\B_0$ be the branching diagram for $(Q_n^F)_{n \ge 0}$,  which was described above, and let 
$\B$ be that obtained by reflections from $\B_0$.  On the $0$--th row,  $\B$ has vertices labeled
by all pairs $(\emptyset, \la)$, where $\la$ is a Young diagram of size $t$.  Finally, augment $\B$ with a copy of Young's lattice up to the $(t-1)$--st level,  with vertices labeled by pairs $(0, \mu)$ with $0 \le |\mu| \le t-1$.   The pairs of Young diagrams labeling the cell modules of $\br_{r, s}$
are located on the $r + s$--th row of the augmented branching diagram,  and a basis of any cell module can be labeled by paths on the augmented branching diagram from $(\emptyset, \emptyset)$ to the pair in question.

We note that several of the results of Section 3 of ~\cite{cox-walled-brauer} follow from the application of our method to the walled Brauer algebras.

\subsection{Partition algebras}  

\subsubsection{Definition of the partition algebras}  Let $n$ be an integer,  $n \ge 1$.     Let 
$[\p n] = \break\{\p 1, \dots, \p n\}$  and $[\pbar n] = \{ \pbar 1, \dots, \pbar n\}$  be disjoint sets of size $n$,  and 
let $X_n$  be the family of all set partitions of   $[\p n] \cup [\pbar n]$.    

We can represent an element  $x$ of $X_n$  by any graph with vertex set equal to $[\p n] \cup [\pbar n]$ whose connected components are the blocks or classes of the partition $x$.  We picture such a graph as a diagram
in the rectangle $\mathcal R$,  with the vertices in $[\p n]$ arranged on the top edge and those in
$[\pbar n]$  arranged on the bottom edge of $\mathcal R$, as in the tangle diagrams discussed in Section \ref{subsection:  preliminaries on tangle diagrams}.

Let $S$  be any commutative ring with identity, with a distinguished element $\delta$.   
We define a product on $X_n$ as follows:  
Let $x$ and $y$  be elements of $X_n$.  Realize $y$ as a set partition of $[\p n] \cup [\p n']$   (with $[\p n']$ the set of bottom vertices).  Realize $x$ as a set partition of $[\p n'] \cup [\pbar n]$   (with $[\p n']$ the set of top vertices).
Let $E_x$  and $E_y$  be the corresponding  equivalence relations,  regarded as equivalence relations on
$[\p n] \cup [\p n'] \cup [\pbar n]$.    Let $E$  be the smallest equivalence relation on $[\p n] \cup [\p n'] \cup [\pbar n]$ containing $E_x \cup E_y$.   Let $r$  be the number of equivalence classes of $E$ contained in $[\p n']$.
Let $E_{x y}$  be the equivalence relation obtained by restricting $E$ to $[\p n] \cup [\pbar n]$, and let
$z$ be the corresponding set partition of $[\p n] \cup [\pbar n]$.   Then $x y$ is defined to be $\delta^r  z$.

 Here is an example of two set partitions represented by graphs and their product.

$$
y = \inlinegraphic[scale= .3]{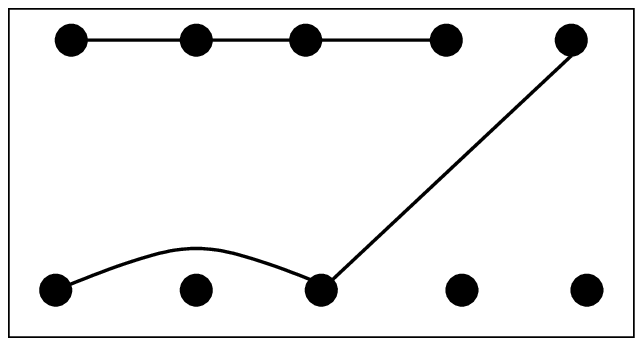}
\qquad
x =  \inlinegraphic[scale=.3]{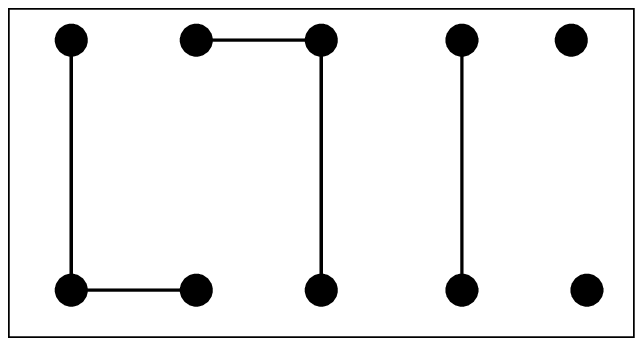}\qquad
xy  = \delta \  \inlinegraphic[scale= .3]{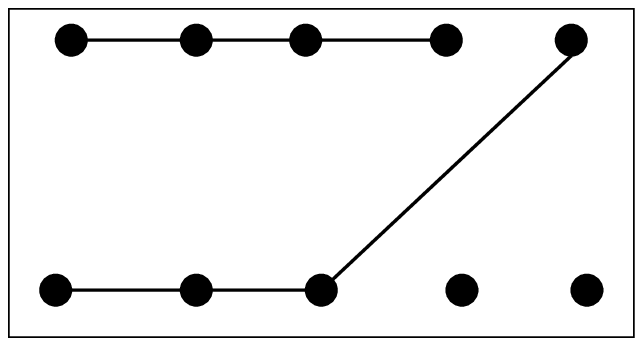} 
$$ 

We let $A_{2n}(S, \delta)$  be the free $S$ module with basis $X_n$.  We give $A_{2n}(S, \delta)$ the bilinear product extending the product defined on $X_n$.  One can check the multiplication is associative.  Note that $A_0(S, \delta)  \cong S$.  
For $n \ge 1$, the  multiplicative identity of $A_{2n}(S, \delta)$ is the partition with blocks $\{ \p i, \pbar i\}$  for $1 \le i \le n$.

For $n \ge 1$,  Let $X'_n \subset X_n$  be the family of set partitions with $\p n$  and $\pbar n$  in the same block. The
$S$--span of $X'_n$  is a unital subalgebra of $A_{2n}(S, \delta)$,   which we denote by  \break$A_{2n -1}(S, \delta)$.  

The algebras $A_k(S, \delta)$  for $k \ge 0$   are called the {\em partition algebras}.  

Note that  the set partitions $x \in X_n$ each of whose blocks has size $2$ can be identified with Brauer diagrams on $2n$ vertices,  and  the product of such diagrams in the Brauer algebra $B_n(S, \delta)$  agrees with the product in $A_{2n}$.   Thus $B_n(S, \delta)$ can be identified with a unital subalgebra of $A_{2n}(S, \delta)$.

\subsubsection{Brief history of the partition algebras}  The partition algebras  $A_{2n}$  were introduced independently by Martin ~\cite{martin-partition-JKTR, martin-structure-j.algebra}  and Jones   ~\cite{jones-partition}.   Partition algebras arise as centralizer algebras for the symmetric group $\mathfrak S_k$  acting
as a subgroup of ${\rm GL}(k, \C)$ on tensor powers of $\C^k$ ~\cite{jones-partition,  martin-partition-potts}.
The algebras $A_{2n+1}$  have been used as an auxiliary device for studying the partition algebras, by Martin and others.   Halverson and Ram ~\cite{halverson-ram-partition}   emphasized putting the even and odd algebras on an equal footing, which reveals the role played by the basic construction.
Cellularity of the partition algebras was proved in ~\cite{Xi-Partition, doran-partition, wilcox-cellular}.  For further literature citations, see the review article ~\cite{halverson-ram-partition}.

\subsubsection{Some properties of the partition algebras}  Fix a ground ring $S$  and $\delta \in S$.
In this section write $A_k$  for $A_k(S, \delta)$.

For $n \ge 1$, $A_{2n-1}$  is defined as a subalgebra of 
$A_{2n}$.  The map  $\iota : X_{n} \to X'_{n+1}$ which adds the additional block
$\{\p {n+1}, \pbar  {n+1}\}$ to $x \in X_n$ is an imbedding;  the linear extension of $\iota$ to $A_{2n}$
is a unital   algebra monomorphism into $A_{2n + 1}$.  Using $\iota$,  we identify $A_{2n}$ with its image in $A_{2n + 1}$.

For $n \ge 1$,  let  $p_{2n -1} \in A_{2n}$  
be the set partition of $[\p n] \cup [\pbar n]$  with blocks $\{\p n\}$,   $\{\pbar n\}$, and $\{\p i, \pbar i\}$  for  $1 \le i \le n-1$,  .   The element
$p_{2n -1}$  satisfies $p_{2n-1}^2 =  \delta\,  p_{2n-1}$.    Let $p_{2n} \in A_{2n + 1}$   be the set partition of $[\p {n+1}] \cup [\pbar {n+1}]$  with blocks  $\{ \p n, \p {n+1},  \pbar n, \pbar {n+1}\}$ and $\{\p i, \pbar i\}$  for  $1 \le i \le n-1$.   Then $p_{2n} $  is an idempotent.   

Here are graphs representing the $p_k$  for $k$ even and odd:
$$
p_8 = \inlinegraphic[scale= .3]{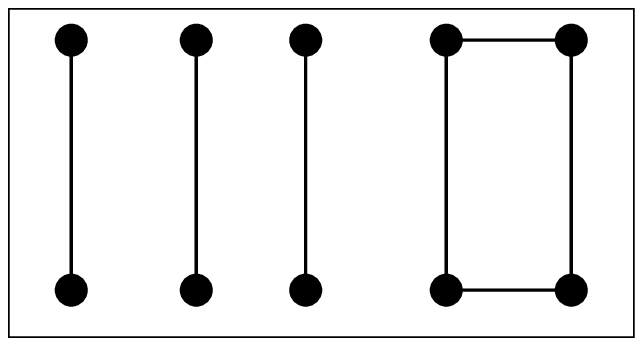}
\qquad
p_9 =  \inlinegraphic[scale=.3]{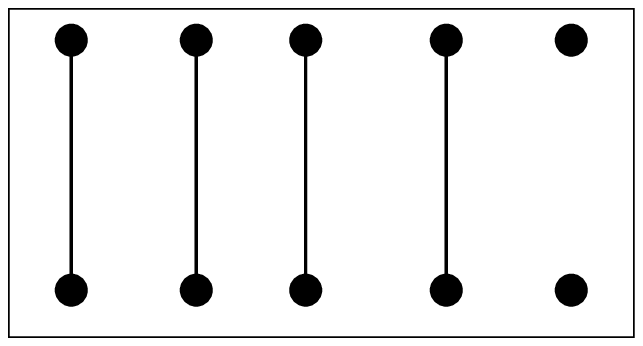}\
$$ 
One can check that 
\begin{equation}
p_k p_{k \pm 1} p_k = p_k  \quad \text{for all $k$}.
\end{equation}

Define an involution $i$  on $X_n$  by interchanging $\p j$  with $\pbar j$  for each $j$.   The map $i$   reflects a graph $d(x)$  representing $x \in X_n$ in the line $y = 1/2$.   The linear extension of $i$  to 
$A_n$  is an algebra involution.   Note that $X'_n$  and  $A_{2n-1}$  are invariant under $i$.  
The embeddings of $A_k$ in $A_{k+1}$  commute with the involutions.    The elements $p_k$ are 
invariant under $i$.  

Define a map ${\rm cl} : X_n \to X'_n$ by merging the blocks containing $\p n$  and $\pbar n$, and define
${\rm cl} : A_{2n} \to A_{2n-1}$  as the linear extension of the map ${\rm cl} : X_n \to X'_n$.

Define a map 
${\rm cl} : X'_n \to A_{2n-2}$ as follows:  For $x \in X'_n$,  if  $\{\p n, \pbar n\}$ is a block of $x$,  then
${\rm cl}(x) =  \delta \,x'$, where $x' \in X_{n-1}$  is obtained by removing the block  $\{\p n, \pbar n\}$.
Otherwise,  ${\rm cl}(x) \in X_{n-1}$  is obtained by intersecting each block of $x$  with $[\p {n-1}] \cup [\pbar {n-1}]$.    Define 
${\rm cl} : A_{2n-1} \to A_{2n-2}$  as the linear extension of the map ${\rm cl} : X'_n \to A_{2n-2}$.

One can check that for all $k$,  ${\rm cl} : A_k \to A_{k-1}$  is a non--unital  $A_{k-1}$--$A_{k-1}$ bimodule map.   Moreover,  ${\rm tr} = {\rm cl} \circ {\rm cl} \circ \cdots \circ {\rm cl} : A_k \to A_0 \cong S$  is a non--unital trace.    The trace ${\rm tr}$  can be computed as follows:   given $x \in X_n$,  let $d(x)$  be any graph representing $x$  and let $d'(x)$  be the graph augmented by drawing edges between each pair
of vertices $\{\p j, \pbar j\}$;  then  ${\rm tr}(x) =  \delta^r$,  where $r$ is the number of components of $d'(x)$.

The maps ${\rm cl}$  commute with the algebra involutions $i$,  and ${\rm tr} (a)  =  {\rm tr}(i(a))$.  
Moreover,
\begin{equation}
p_k x p_k  = {\rm cl}(x) p_k   \quad \text{for all $x \in A_k$,  $k \ge 1$}.
\end{equation}

If $\delta$ is invertible,  define $\eps_{2n} : A_{2n} \to A_{2n -1}$  by $\eps_{2n} = {\rm cl}$,  and
$\eps_{2n-1} : A_{2n-1} \to A_{2n -2}$  by $\eps_{2n-1} = \delta\inv \, {\rm cl}$.   Then the maps 
$\eps_k$  are unital conditional expectations, and the map $\eps = \eps_1 \circ \cdots \eps_k : A_k \to A_0 \cong S$  is a unital trace.

\def\pn{{\rm pn}}
Let $x \in X_n$.  Call a block of $x$  a {\em through block}  if the block has non--empty intersection with both
$[\p n ]$  and $[\pbar n]$.    The number of through blocks of $x$  is called the propagating number
of $x$, denoted ${\rm pn}(x)$.   Clearly, $\pn(x) \le n$ for all $x \in X_n$.   The only $x \in X_n$  with propagating number equal to $n$  are Brauer diagrams with only vertical strands,  i.e.\  permutation diagrams.

 If $x, y \in X_n$  and $x y = \delta^r z$,  then
${\rm pn}(z)  \le \min\{ {\rm pn}(x), {\rm pn}(y)\}$.   Hence the span of the set of $x \in X_n$  with
$\pn(x) < n$ is an ideal $J_{2n} \subset A_{2n}$.   Moreover,  $J_{2n -1}  := J_{2n}  \cap A_{2n -1}$  is the span
of $x \in X'_n$  with $\pn(x) < n$.  

\begin{lemma}  \label{lemma:  axiom 5 for partition algebras}

 For $n \ge 1$, $A_{2n}/J_{2n}   \cong S \mathfrak S_n$,  and $A_{2n-1}/J_{2n-1}   \cong S \mathfrak S_{n-1}$, as algebras with involution.
\end{lemma}

\begin{proof}  The span of permutation diagrams is a linear complement to  $J_{2n}$,  and is an $i$--invariant subalgebra of $A_{2n}$  isomorphic to $S \mathfrak S_n$;   hence, $A_{2n}/J_{2n}   \cong S \mathfrak S_n$.  The span of permutation diagrams $\pi$  with $\pi(n) = n$  is a linear complement to $J_{2n-1}$  in $A_{2n-1}$;  hence
$A_{2n-1}/J_{2n-1}   \cong S \mathfrak S_{n-1}$.
\end{proof}

\begin{lemma}  For $k \ge 2$,     $J_{k} =  A_{k-1}  p_{k-1}  A_{k-1}$.  
\end{lemma}

\begin{proof}  It is straightforward to check that if $x \in X_n$  has propagating number strictly less than $n$,  then 
$x$   can be factored as $x = x' p_{2n-1}  x''$,  with $x', x'' \in X'_n$.     Likewise,  if $n \ge 2$  and $x \in X'_n$
has propagating number strictly less than $n$,  then $x$  can be factored as
$x  = x' p_{2n-2} x''$  with $x', x'' \in X'_{n-1}$.    
\end{proof}

\begin{lemma}\label{lemma:  axiom 6 for partition algebras}
 \mbox{}
 \begin{enumerate}

\item 
For $k \ge 3$, $p_{k-1} A_{k-1} p_{k-1} = A_{k-2} p_{k-1}$.
 \item  $p_1 A_1 p_1 =  \delta\, A_0  \,p_1$.
\item  For $k \ge 2$,  
$p_{k-1}$  commutes with $ A_{k-2} $.  
\end{enumerate}
 \end{lemma}

\begin{proof}  Let $x \in A_{2n}$  with $n \ge 1$.   Then $p_{2n-1} x p_{2n-1}$  is contained in the span of $y \in X_n$  such
that  $\{\p n\}$  and $\{ \pbar n\}$  are blocks of $y$, and any such $y$  can be written as $y = z p_{2n-1}$,  where $z \in A_{2n-2}$.      

Now consider $x \in A_{2n + 1}$  with $n \ge 1$.   Then
$p_{2n} x p_{2n}$  is contained in the span of $y \in X'_{n+1}$  such that $\{\p n, \p {n+1}, \pbar n, \pbar {n+1}\}$  is contained in one block of $y$.    Any such $y$  can be written as $y =  z p_{2n}$  where
$z \in  A_{2n-1}$.  

This shows that  $p_{k-1} A_k p_{k-1}  \subseteq  A_{k-2}  p_{k-1}$  for all $k \ge 3$.
On the other hand, if $x \in A_{k-2}$   then $x p_{k-1} =  x  p_{k-1} p_{k-2} p_{k-1} =
p_{k-1} x p_{k-2}  p_{k-1}  \in  p_{k-1} A_{k}  p_{k-1}$,   so $p_{k-1} A_k p_{k-1}  \supseteq  A_{k-2}  p_{k-1}$.
This proves (1).

Points (2) and (3) are easy to check.
\end{proof}

\begin{lemma}  \label{lemma: axiom 7 for partition algebras}
For $k \ge 2$,    $A_k p_{k-1}  = A_{k-1} p_{k-1}$.  Moreover,  
 $x \mapsto  x e_{k-1}$  is injective from $ A_{k-1}$ to $A_{k}$.
\end{lemma}

\begin{proof}  For $k =2$,  we have $A_2 p_1 = S p_1 = A_1 p_1$.    For $k \ge 3$,     we have
$$
\begin{aligned}
A_k p_{k-1} &= A_k p_{k-1} p_{k-2} p_{k-1}  \\& \subseteq J_k p_{k-1}  = A_{k-1} p_{k-1} A_{k-1} p_{k-1} 
\\& \subseteq
 A_{k-1} A_{k-2} p_{k-1}  = A_{k-1} p_{k-1}.
 \end{aligned}
 $$
 Checking $k$ odd and even separately,  one can check that $x =  {\rm cl}(x p_{k-1})$  for $k \ge 2$  and
 $x \in A_{k-1}$.
\end{proof}

\begin{lemma} \label{lemma:  non-degeneracy of trace partition algebra}
  The trace $\eps$ on $A_k(\Q(\deltabold), \deltabold)$  is non--degenerate.
\end{lemma}

\begin{proof}   For any set partition $x \in X_n$,  let $r(x)$  be the number of blocks of $x$.   Let $E_x$
be the equivalence relation on
$[\p n] \cup [\pbar n]$  whose equivalence classes are the blocks of $x$.  

 For any $x, y \in X_n$,  define an integer $r(x, y)$  as follows:   Let $E(x, y)$ be the smallest equivalence relation on $[\p n] \cup [\pbar n]$   containing  $E_x \cup E_{i(y)}$  and let $r(x, y)$  be the number of equivalence classes of $E(x, y)$.    Clearly, $r(x, y) \le \min\{r(x), r(y)\}$.   Moreover,  if $r(x) = r(y)$,  then
 $r(x, y) < r(x)$  unless $y =  i(x)$,  and $r(x, i(x)) = r(x)$.    
 
 It is not hard to see that  ${\rm tr}(x y) =  \deltabold^{r(x, y)}$, so $\eps(x, y) =   \deltabold^{r(x, y)-n}$.
 It follows that the Gram determinant  $\det(\eps(x y))_{x, y}$  is a Laurent polynomial that has a unique
 term of highest degree namely $\pm  \prod_x  \eps(x\, i(x))$.   In particular the Gram determinant is non--zero.
 This shows that the trace on $A_{2n}(\Q(\deltabold), \deltabold)$  is non--degenerate,  and the same method shows that the restriction of the trace to $A_{2n-1}(\Q(\deltabold), \deltabold)$ is non--degenerate.
 \end{proof}

\begin{lemma} \label{lemma:  generic semisimplicity for partition algebras}   $A_k(\Q(\deltabold), \deltabold)$  is  split semisimple.    The branching diagram for \break
$(A_k(\Q(\deltabold), \deltabold))_{k \ge 0}$  has vertices on levels $2n$ and $2n+1$  labeled by all Young diagrams of size $j$,  $0 \le j \le n$.  There is an edge connecting $\la$ on level $2n$  and $\mu$  on level
$2n \pm 1$ if, and only if, $\la = \mu$  or $\mu$ is obtained by removing one box from $\lambda$.  \end{lemma}

\begin{proof}  This is proved by Martin  ~\cite {martin-partition-JKTR}.   It can also be proved using the method of Wenzl from ~\cite{Wenzl-Brauer}, using
Lemma \ref{lemma:  non-degeneracy of trace partition algebra}.
\end{proof}

\subsubsection{Verification of framework axioms for the partition algebras}

We take $R = \Z[\deltabold]$,  where $\deltabold$ is an indeterminant.   Then $R$ is the universal ground ring for the partition algebras;  for any commutative ring $S$ with distinguished element $\delta$, we have
$A_k(S, \delta) \cong  A_k(R, \deltabold)\otimes_R  S$.  Let $F = \Q(\deltabold)$ denote the field of fractions of $R$.  Write $A_k = A_k(R, \deltabold)$.  Define $Q_{2n}  = Q_{2n + 1} =  R \mathfrak S_n$.

\begin{proposition} \label{proposition: framework axioms for partition algebras} The two sequence of $R$--algebras  $(A_k)_{k \ge 0}$ and $(Q_k)_{k \ge 0}$  satisfy the framework axioms of  Section \ref{subsection: framework axioms}.
\end{proposition}

\begin{proof}
According to Example \ref{example: Hn coherent tower},  $(Q_k)_{k \ge 0}$ is a coherent tower of cellular algebras, so axiom (\ref{axiom Hn coherent}) holds.
Framework axioms (\ref{axiom: involution on An}) and (\ref{axiom: A0 and A1}) are evident.
$A_k^F$ is split semisimple by  Lemma \ref{lemma:  generic semisimplicity for partition algebras}.  This verifies axiom (\ref{axiom: semisimplicity}).

We take $p_{k-1} \in A_k$ to be the element defined in the previous section.  Then   
$p_{k-1}$  is an  $i$--invariant essential idempotent.
With   $J_k = A_k p_{k-1} A_k$, we have $A_k/J_k \cong  Q_k$ as algebras with involution by Lemma \ref{lemma:  axiom 5 for partition algebras}.  
  This verifies axiom (\ref{axiom:  idempotent and Hn as quotient of An}).
  
Axiom (\ref{axiom: en An en}) follows from Lemma  \ref{lemma:  axiom 6 for partition algebras}, and axiom (\ref{axiom:  An en}) from 
 Lemma \ref{lemma: axiom 7 for partition algebras}.   Axiom (\ref{axiom: e(n-1) in An en An}) holds because $p_{n-1} p_n p_{n-1} = p_{n-1}$.
 \end{proof}

\begin{corollary} For any commutative ring $S$ and for any $\delta \in S$, 
the sequence of  partition  algebras $(A_n(S, \delta))_{n \ge 0}$ is a coherent tower of cellular algebras.
$A_n(S, \delta)$  has cell modules indexed by all Young diagrams of size  $j$,  $0 \le j \le n$.
   The cell module labeled by
a Young diagram $\la$ has a basis labeled by paths on the branching diagram for $(A_k(\Q(\deltabold), \deltabold))_{k \ge 0}$,  described in  Lemma \ref{lemma:  generic semisimplicity for partition algebras}.
\end{corollary}

\subsection{Contour algebras}
We define generalizations of the {\em contour algebras} of Cox et.\ al.\   ~\cite{cox-towers}, which in turn include several sorts of diagram algebras.   The algebras are obtained as a sort of wreath product of the Jones--Temperley--Lieb algebras with  some other algebra  $A$ with involution;  varying $A$ gives a wide variety of examples.

\subsubsection{Definition of contour algebras}     Let $S$ be a  commutative ring with distinguished element $\delta$.   Let $A$  be an $S$--algebra with involution $i$ and with a  unital $S$--valued trace $\eps$.  We first define the $A$--Temperley--Lieb algebras $\tl_n(A)$ and then  the contour algebras $\cont  n d A$ as subalgebras of $\tl_n(A)$.      In case we need to emphasize the ground ring $S$ and parameter $\delta$, we
write $\cont n d {A, S, \delta}$.  

An $A$--Temperley--Lieb diagram is a \TL (TL)   diagram with strands labeled by elements of $A$.    For convenience,  we adopt the convention that an unlabeled strand is the same as a strand labeled with the identity of $A$.

We will define the product of two $A$--\TL diagrams.   
First we note that  ordinary TL diagrams have an inherent orientation.  Label the top vertices of a TL diagram by
$\p 1, \dots, \p n$  and the bottom vertices by $\pbar 1, \dots, \pbar n$.   Place a small arrow pointing down at each
odd numbered vertex (top or bottom)  and a small arrow pointing up at each even numbered vertex.  Then  because of the planarity of  TL diagrams,  each strand of a TL diagram  must connect  one arrow pointing into the rectangle $\mathcal R$ of the diagram with one arrow pointing out of $\mathcal R$;  the strand can be thought of as oriented from the inward pointing arrow to the outward pointing arrow. When two TL diagrams are multiplied by stacking,  the orientation of composed strands agrees.  

Now consider two $A$--\TL diagrams $X$ and $Y$.   To form the product $X Y$,  stack $Y$ over $X$ as for tangles, forming a composite diagram $X\circ Y$.   Label each non--closed composite strand with the product of the labels  of its component strands from $X$ and $Y$, taken in the order of their occurrence as the strand is traversed according to its orientation.      For each closed strand $s$ in $X\circ Y$,  let $\eps(s)$  be the trace of the product of
the labels of its component strands;  the product is unique up to cyclic permutation of the factors,  so the trace 
is uniquely determined.    Let $r$  be the number of closed strands and let $Z$ be the labeled diagram obtained by removing all the closed strands.  Then $XY =  \delta^r ( \prod_s   \eps(s))\,   Z$.  

As an $S$--module,  $\tl_n(A)$ is $A^{\otimes n}  \otimes \tl_n(S, \delta) = \bigoplus_x   (A^{\otimes n}  \otimes x)$,  where the sum is over ordinary \TL diagrams $x$.    We identify a simple tensor  $a_1 \otimes \cdots \otimes a_n \otimes x$
with a labeling of $x$  with the labels $a_1, \dots, a_n$.  We have to specify how to place the labels.  We fix an ordering of the vertices, for example $\p 1 < \cdots < \p n <  \pbar 1 <  \cdots \pbar n$,  and then order the strands of $x$  according to the order of the initial vertex of each (oriented) strand.   The simple tensor
$a_1 \otimes \cdots \otimes a_n \otimes x$
is identified with the  diagram with underlying TL diagram $x$, with the $j$--th strand of $x$ labeled by $a_j$   for each $j$.

Fix TL  diagrams $x$ and $y$.    The  product of $A$--\TL diagrams with underlying TL diagrams $x$ and $y$, defined above,  determines a multilinear map $A^{2n} \to   A^{\otimes n} \otimes xy$,  and hence a bilinear
map $( A^{\otimes n} \otimes x) \times ( A^{\otimes n} \otimes y) \to  A^{\otimes n} \otimes xy$.  This product extends to a bilinear product on $\tl_n(A)$,  which one can check to be associative.

Next we define an involution on $\tl_n(A)$.     Define $i$ on an $A$--labeled TL diagram by flipping the diagram over the line $y = 1/2$ and applying the involution in $A$  to the label of each strand.  For a fixed TL diagram $x$, this gives a multilinear map from $A^n$  to $A^{\otimes n} \otimes i(x)$, and hence a linear map from
$A^{\otimes n} \otimes x$ to $A^{\otimes n} \otimes i(x)$.   Now $i$  extends to a linear map on 
$\tl_n(A)$.    One can check that $i$ is an algebra involution.

This completes the definition of the $A$--Temperley--Lieb algebra, as an algebra with involution.

Next we  define the $A$--contour algebras.    We assign a depth to each strand in an ordinary TL diagram $x$, as follows:    Draw a curve from a point on a given strand $s$  to the western boundary of $\mathcal R$, having only transverse intersections with
any  strands of $x$.    The depth of $s$ is the minimum,  over all such curves  $\gamma$,  of the number of points of intersection of $\gamma$  with the strands of $x$  (including $s$).      The depth of an $A$--labeled TL diagram is the maximum depth of the strands with non--identity labels.

Fix $d \le n$.  As an $S$--module $\cont n d A$ is the span of  those $A$--labeled TL diagrams of depth no greater than $d$.    It is easy to check as in
~\cite{cox-towers}  Lemma 2.1  that $\cont n d A$ is an $i$--invariant subalgebra of $\tl_n(A)$.  

For $a \in A$  and $1 \le j \le n$  let $a^{(j)}$  be the identity TL diagram in $\tl_n(A)$  with the $j$--th strand labeled with $a$  (and the other strands unlabeled).    We have $a\power j$ and $b \power k$  commute if $j \ne k$. 
Also $a \power j$  commutes with $e_k$ unless $j \in \{k, k+1\}$   and $e_k a \power k = e_k a \power {k+1}$, and, likewise,   $ a \power k  e_k= a \power {k+1} e_k$. 
Note that $a \mapsto  a \power k$  is an algebra homomorphism if $k$  is odd,  but an algebra anti-homomorphism if $k$  is even.

\begin{lemma}  \label{lemma; generating contour algebras}
$\cont n d  A$ is generated as an algebra by $e_1, \dots, e_{n-1}$  and by 
 $\{ a \power k :  1 \le k \le d\}$. 
\end{lemma}

\noindent{\em Sketch:}   It is enough to show that if $x$ is a \TL diagram and  $X = x a \power k$   has depth
$r$,  then   $X$  can be rewritten as a product of  $a \power r$ and TL diagrams.    First one can check that
$X$  can be written as  $x_1 \, x_2 a \power {k'} \, x_3$   where the $x_i$  are TL diagrams,    $x_2$ is a product of commuting  $e_i$'s,  and the depth of $ x_2 \,a \power {k'}  $  is $r$.    Finally, it suffices to show that
$ x_2 \,a \power {k'}  $   can be written as a product of TL diagrams with $a \power r$.    We give an example that
captures  the idea:   $e_1 e_3  a \power 6$  has depth $2$.    We have
$$
\begin{aligned}
e_1 e_3  a \power 6 &=  (e_1 e_3)  (e_2 e_4) (e_1 e_3) a \power 6 \\
&=  (e_1 e_3)  (e_2 e_4) (e_3 e_5) (e_2 e_4) (e_1 e_3) a \power 6 \\
&= (e_1 e_3)  (e_2 e_4) (e_3 e_5)a \power 2 (e_2 e_4) (e_1 e_3),
\end{aligned}
$$
by repeated use of the relations listed before the statement of the lemma.

\subsubsection{Brief history of contour algebras}     The contour algebras introduced by Cox et.\ al.\   ~\cite{cox-towers}   are the special case with $A$  the group algebra of the cyclic group $\Z_m$.    On the other hand, the
$A$--\TL  algebras  $\tl_n(A)$  have been considered in  ~\cite{jones-planar},  Example 2.2.   The contour subalgebras of $\tl_n(A)$  were discussed in ~\cite{green-martin}.

\subsubsection{Some properties of  $A$--\TL and contour algebras}   We deal with the contour algebras and the
$A$--\TL algebras together;   regard $\tl_n(A)$  as $\cont n {\infty}  A$.  

    We define maps $\iota: \cont n d A \to \cont {n+1} d A$  as for other classes of 
diagram or tangle algebras, and likewise   maps $\cl :  \cont n d A \to \cont {n-1} d A$;  if closing the rightmost strand of an $A$--\TL diagram produces a closed loop,  remove the loop and multiply the resulting diagram
by $\delta$  times the trace of the product of labels along the loop.   The map $\iota$ is injective, since
$x = \cl( \iota(x) e_n)$  for  $x \in \cont n d A$.    The maps $\iota$  and $\cl$  commute with the involutions.

If $\delta$ is invertible in $S$, we can define $\eps_n = (1/\delta) \cl : \cont n d A \to \cont {n-1} d A$, which is a  unital
conditional expectation.  We have
${\eps_{n+1}}\circ\iota(x)  =  x$ for $x \in 
\cont n d A$.   The map $\eps = \eps_1\circ \cdots \circ \eps_n : \cont n d A \to \cont 0 d A  \cong S$ is a normalized trace.  The value of $\eps$ on an $A$-\TL diagram $X$ with $n$  strands  is obtained as follows:  first close all the strands of $X$ by introducing new curves joining $\p j$ to $\pbar j$ for all $j$;   let $r$  be the number of closed loops in the resulting
diagram;  then $\eps(X) = \delta^{r - n} \prod_s  \eps(s)$,  where the product is over the collection of closed loops $s$,  and $\eps(s)$  denotes the trace in $A$  of the product of labels along the loop $s$.

The span $J$   of $A$--\TL  diagrams of depth $\le d$ and  with at least one horizontal strand is an ideal in $\cont n d A$.    
By Lemma \ref{lemma; generating contour algebras},  any  $A$--\TL  diagram with depth   $ \le d$ can be written as a word in the $e_i$'s and in elements $a \power k$  with $k \le d$;   the diagram is in $J$ if, and only if,
some $e_i$ appears in the word.  Thus $J$ is the ideal generated by the $e_i$'s.   Because of the relations
$e_i e_{i \pm 1} e_i = e_i$,   $J$ is generated by $e_{n-1}$.     The quotient $\cont n d A/ J$ is isomorphic
(as algebras with involution)   to the subalgebra generated by the $a \power k$  with $k \le d$,  and thus to $A^{\otimes d}$  if $n \ge d$  and
$A^{\otimes n}$  if $n < d$.  

 \begin{lemma} \label{lemma:  axiom 6 for contour} \mbox{}
 \begin{enumerate}
\item 
For $n \ge 3$, $e_{n-1} \, \cont {n-1} d A  \,  e_{n-1} = \cont {n-2} d A  \, e_{n-1}$.
\item  $e_1 \, \cont 1 d A \, e_1 =  \delta\, S  \,e_1$
\item  For $n \ge 2$,  
$e_{n-1}$  commutes with $ \cont {n-2} d A$.  
\end{enumerate}
 \end{lemma}
 
 \begin{proof}   The proof is the same as that of Lemma \ref{lemma:  Brauer  axiom 6} for the Brauer algebras.
 \end{proof}
 
 \begin{lemma} \label{Bn e(n-1) = B(n-1) e(n-1)  for contour algebras}  For $n \ge 2$, 
$\cont n d A  \,e_{n-1} =  \cont {n-1} d A \,e_{n-1}$.  Moreover,  
 $x \mapsto  x e_{n-1}$  is injective from $\cont {n-1} d A$ to $\cont {n-1} d A \,e_{n-1}$.
\end{lemma}

\begin{proof}    Any $A$--TL diagram  in $\cont n d A$ is either already in $\cont {n-1} d A$,  or it can be written
as $ \alpha  \chi  \beta$,   with   $\alpha, \beta \in \cont {n-1} d  A$,  and 
$\chi \in \{ e_{n-1} , a \power n \}$  if $n \le  d$,  or $\chi = e_{n-1}$   if $n > d$.   

The remainder of the proof is the same as the proof of Lemma \ref{B(n+1) e(n) = B(n) e(n)  for Brauer algebras}  for the Brauer algebras, using  the identities:   $a \power {n}  x  e_{n-1} =     x  a \power {n-1}  e_{n-1}$,    and $e_{n-1}  x e_{n-1} =  \cl(x) e_{n-1}$  for $x \in \cont {n-1} d A$.  
\end{proof}

\subsubsection{Hypotheses on the algebra $A$}    \label{subsubsection:  hypotheses on A for generic contour algebras}

We will suppose that the algebra $A$ has a generic version defined over an integral domain $R_0$.     
Let $F_0$ be the field of fractions of $R_0$.  We suppose that $A = A(R_0)$  satisfies the following hypotheses:
\begin{enumerate}
\item  $A = A(R_0)$ is cellular.
\item  $A(F_0) = A(R_0) \otimes_{R_0} F_0$  is split semisimple.
\item  The trace $\eps$ on $A({R_0})$  is non--degenerate.
\end{enumerate}

We  take $R = R_0[\deltabold]$, where $\deltabold$ is an indeterminant, and 
let $F = F_0(\deltabold)$  denote the field of fractions of $R$.   
We will show that  $(\cont n d {A, R, \deltabold})_{n \ge 0} $  is a coherent tower of cellular algebras.

\subsubsection{Special instances}    The cellular algebra $A$ in Section  \ref{subsubsection:  hypotheses on A for generic contour algebras} can be taken to be the generic version of any of the diagram or tangle algebras treated in this paper.      $A$  could be taken to be a generic Hecke algebra or cyclotomic Hecke algebra,  or the group ring  of  a symmetric group over $R_0 = \Z$.    

The contour algebras of Cox et.\  al.\  ~\cite{cox-towers}  are recovered by taking
$R_0 =    \Z[  \deltabold_1,  \dots,  \deltabold_{m-1}]$  and  $A$  the group algebra of $\Z_m$   over $R_0$. The trace on $A$ is determined by $\eps([k]) =  \deltabold_k$  for $[k] \ne [0]$  and $\eps([0] )= 1$.   The parameter $\delta_0$ in  ~\cite{cox-towers}  becomes identified with our $\delta$.

\subsubsection{Verification of the framework axioms for contour algebras}  Adopt the hypotheses and notation of Section  \ref{subsubsection:  hypotheses on A for generic contour algebras}.   

\begin{lemma} \label{lemma:  non degenerate trace on contour algebras}
 The trace $\eps$  on $\cont n d {A, F,  \deltabold}$  is non--degenerate.
\end{lemma}

\begin{proof}    We take any basis $\mathbb A$ of $A$ over $F_0$  with $\bm 1 \in \mathbb A$.   As a basis $\mathbb B$   of  $\cont n d A$  over $F$  we take all $n$--strand TL diagrams decorated up to depth $d$  with elements of $\mathbb A$.    We consider the modified Gram determinant   $\det [\eps( X i(Y))]_{X, Y \in \mathbb B}$.    If $X$ and $Y$  have different underlying TL diagrams, then  $\eps( X i(Y)) \in \delta\inv F_0$.    

Next consider matrix entries $\eps( X i(Y))$  where $X$ and $Y$  have  the same underlying TL diagram, say $x$.  Suppose $x$  has $\ell$ strands at depth $d$ or less and these strands are decorated by basis elements
$a_{1},  \dots, a_{\ell}$ in $X$,   respectively  $b_{1}, \dots, b_{\ell}$ in $Y$.  Then $\eps(X i(Y)) = \prod_{j=1}^{\ell} \eps(a_{j} i(b_{j}))$.     The determinant of the square submatrix of  $ [\eps( X i(Y))]$  consisting of those entries for which $X$ and $Y$ both have underlying TL diagram $x$  is  therefore $D^{\ell}$,  where
$D$ is the determinant of    $[\eps(a i(b))]_{a, b \in \mathbb A}$.   
It follows that $\det [\eps( X i(Y))]_{X, Y \in \mathbb B}$ is equal to a power of $D$  modulo
$\delta\inv R_{0}$,  and is therefore non--zero. 
\end{proof}

Consider $$Q_{n} =  \cont n d A/ J  \cong   \begin{cases}  A^{\otimes n} &\text{if $n < d$}  \\
A^{\otimes d} &\text{if $n \ge d$.} \\ \end{cases}
$$
By the assumptions in Section  \ref{subsubsection:  hypotheses on A for generic contour algebras},
$Q_{n}(R)$ is cellular and $Q_{n}(F)$  is split semisimple.    Moreover, it is easy to see that
$(Q_{n})_{n \ge 0}$ is a coherent tower of cellular algebras.

\begin{lemma} \label{lemma:  generic semisimplicity for contour algebras} 
$\cont n d {A, F, \deltabold}$ is split semisimple for all $n$.  
\end{lemma}

\begin{proof}  The method of Wenzl  from ~\cite{Wenzl-Brauer} applies,  using the non--degeneracy of the trace and the split semisimplicity of $Q_{n}(F)$  for all $n$.   
\end{proof}

\begin{proposition}   The pair of sequences $(\cont n d {A, R, \deltabold})_{n \ge 0}$  and
$(Q_{n}(R))_{n \ge 0}$  satisfy the framework axioms  of  Section \ref{subsection: framework axioms}.
Hence,  $(\cont n d {A, R, \deltabold})_{n \ge 0}$ is a coherent tower of cellular algebras.  
\end{proposition}

\begin{proof}
We observed above that $(Q_k)_{k \ge 0}$ is a coherent tower of cellular algebras, so axiom (\ref{axiom Hn coherent}) holds.
Framework axioms (\ref{axiom: involution on An}) and (\ref{axiom: A0 and A1}) are evident.
Framework axiom  (\ref{axiom: semisimplicity}) follows from   Lemma \ref{lemma:  generic semisimplicity for contour algebras}.  

The elements $e_{k}$ are  $i$--invariant essential idempotents. 
With   $J =   \cont k d A e_{k-1} \cont k d A$, we have $\cont k d A/J \cong  Q_k$ as algebras with involution.  
  This verifies axiom (\ref{axiom:  idempotent and Hn as quotient of An}).
  Axiom (\ref{axiom: en An en}) follows from Lemma  \ref{lemma:  axiom 6 for contour}, and axiom (\ref{axiom:  An en}) from 
 Lemma \ref{Bn e(n-1) = B(n-1) e(n-1)  for contour algebras}.   Axiom (\ref{axiom: e(n-1) in An en An}) holds because $e_{n-1} e_n e_{n-1} = e_{n-1}$.
 \end{proof}

\bibliographystyle{amsplain}
\bibliography{PathModel}

\end{document}